\documentclass[amscd, leqno, 12pt]{amsart}

\usepackage{amssymb}
\usepackage{amsmath}
\usepackage{amscd}
\usepackage[all]{xy}
\usepackage{amsthm}
\usepackage{graphicx}
\usepackage{mathrsfs}
\usepackage[margin=30truemm]{geometry}
\usepackage{enumitem}
\usepackage{comment}
\usepackage{color}
 
\newtheorem{dfn}{Definition}[section]
\newtheorem{thm}[dfn]{Theorem}
\newtheorem{prop}[dfn]{Proposition}
\newtheorem{lem}[dfn]{Lemma}
\newtheorem{cor}[dfn]{Corollary}
\newtheorem*{cor*}{Corollary}
\newtheorem{claim}{Claim}
\theoremstyle{definition}

\newtheorem{exa}[dfn]{Example}
\newtheorem{assumption}[dfn]{Assumption}

\newcommand{\K}{K\"ahler }

\newcommand{\X}{\mathscr{X}}
\newcommand{\XX}{\mathscr{X}^{\iota}}
\newcommand{\xs}{\mathscr{X}/S}
\newcommand{\xss}{\mathscr{X}^{\iota}/S}

\newcommand{\elm}{\operatorname{elm}}
\newcommand{\mon}{\operatorname{Mon}^2}
\newcommand{\ktm}{\operatorname{KT}(M)}

\newcommand{\vol}{\operatorname{Vol}}

\newcommand{\rk}{\operatorname{rk}}
\newcommand{\sign}{\operatorname{sign}}
\newcommand{\hilb}{\operatorname{Hilb}_P(\mathbb{P}^N)}
\newcommand{\hilbi}{\operatorname{Hilb}_P^I(\mathbb{P}^N)}
\newcommand{\G}{\mu_2}
\newcommand{\D}{\varDelta}

\title[Analytic torsion for IHS 4-folds with involution]{Analytic torsion for irreducible holomorphic symplectic fourfolds with involution, II: the singularity of the invariant}
\author{Dai Imaike}
\address{
Department of Mathematics,
Faculty of Science,
Kyoto University,
Kyoto 606-8502,
Japan}
\email{imaike.dai.22s@st.kyoto-u.ac.jp}

\begin{document} 

\maketitle
\centerline{with an Appendix by Ken-Ichi Yoshikawa}

\numberwithin{equation}{section}

\begin{abstract}
	We study the boundary behavior of the invariant of $K3^{[2]}$-type manifolds with antisymplectic involution, which we obtained using equivariant analytic torsion.
	We show the algebraicity of the singularity of the invariant by using the asymptotic of equivariant Quillen metrics and equivariant $L^2$-metrics. 
	We prove that, in some cases, the invariant coincides with Yoshikawa's invariant for 2-elementary K3 surfaces.
	Hence, in these cases, our invariant is expressed as the Petersson norm of a Borcherds product and a Siegel modular form.
\end{abstract}

\tableofcontents

\setcounter{section}{-1}

\section{Introduction}\label{s-0}
	
	This is the second of a series of three papers investigating equivariant analytic torsion for manifolds of $K3^{[2]}$-type with antisymplectic involution.
	We study the invariant $\tau_{M, \mathcal{K}}$ constructed in \cite{I1} in more detail.
	We analyze the singularity of the invariant near the discriminant locus on the moduli space
	and in some special cases of $(M, \mathcal{K})$, we prove the coincidence of $\tau_{M, \mathcal{K}}$ and Yoshikawa's holomorphic torsion invariant.
	Since Yoshikawa's invariant is expressed as the Petersson norm of a certain automorphic form on the moduli space, 
	so is our invariant $\tau_{M, \mathcal{K}}$ for those $(M, \mathcal{K})$.
	These results are useful for calculating the BCOV invariant of the Calabi-Yau fourfold obtained as a crepant resolution of the quotient of a $K3^{[2]}$-type manifold by an antisymplectic involution.
	We will discuss this subject in the third paper (\cite{I3}).
	
	To understand the boundary behavior of $\tau_{M, \mathcal{K}}$,
	we are naturally led to analyze the singular behavior of the Quillen metric and the $L^2$-metric on the determinant of the cohomology of a degenerating family of compact \K manifolds, since $\tau_{M, \mathcal{K}}$ is the product of some analytic torsions and a correction term and since analytic torsion is the ratio of these metrics.
	In \cite{MR1064578}, Bismut-Bost studied the singular behavior of the Quillen metric for degenerations of relative dimension 1.
	In \cite{MR1486991}, Bismut determined the singularity of Quillen metrics for semistable degenerations of certain type.
	In \cite{MR1639560}, \cite{MR2262777}, Yoshikawa obtained a formula for the singularity of Quillen metrics for general one-parameter degenerations of compact \K manifolds.
	On the other hand, in the case of trivial twisting bundle, the singularity of the $L^2$-metric was studied by Schmid (\cite{MR0382272}), Peters (\cite{MR744325}), Eriksson-Freixas i Montplet-Mourougane (\cite{MR3871824}) in the context of variations of Hodge structure.
	In particular, Eriksson-Freixas i Montplet-Mourougane (\cite{MR3871824}, \cite{MR4255041}) described the singular behavior of the $L^2$-metric in terms of the monodromy action on the limit Hodge structure
	and they determined the singularity of the analytic torsion and the BCOV invariant for degenerations of Calabi-Yau manifolds.
	In this paper, we study an equivariant analogue of these results to analyze the singularity of $\tau_{M, \mathcal{K}}$.
	Let us explain our results in more detail.

\medskip

	Recall that an irreducible holomorphic symplectic fourfold is a $K3^{[2]}$-type manifold if it is deformation equivalent to the Hilbert scheme of 2-points of a K3 surface.
	Fix an abstract lattice $L_2$ isomorphic to the Beauville-Bogomolov-Fujiki lattice of a $K3^{[2]}$-type manifold.
	By Joumaah \cite{MR3519981}, the deformation type of a pair consisting of a $K3^{[2]}$-type manifold and an antisymplectic involution is encoded by the pair $(M, \mathcal{K})$,
	where $M$ is an admissible sublattice of $L_2$ and $\mathcal{K}$ is a K\"ahler-type chamber (See \cite[Definitions 6.1 and 8.1]{MR3519981}).
	An antisymplectic involution $\iota : X \to X$ of a $K3^{[2]}$-type manifold $X$ is said to be of type $(M, \mathcal{K})$ if its deformation type is given by $(M, \mathcal{K})$ (See \cite[Definition 10.1]{MR3519981} and Section \ref{sss-2-1-1} for a precise definition).
	The set of isomorphism classes of $K3^{[2]}$-type manifolds with involution of type $(M, \mathcal{K})$ is denoted by $\tilde{\mathcal{M}}_{M, \mathcal{K}}$.
	In \cite{MR3519981}, Joumaah defined a period map $P_{M, \mathcal{K}} : \tilde{\mathcal{M}}_{M, \mathcal{K}} \to \mathcal{M}_{M, \mathcal{K}} \setminus \bar{\mathscr{D}}_{M^{\perp}}$,
	where $\mathcal{M}_{M, \mathcal{K}}$ is an orthogonal modular variety associated to $(M, \mathcal{K})$
	and $\bar{\mathscr{D}}_{M^{\perp}}$ is a divisor on $\mathcal{M}_{M, \mathcal{K}}$ (see Section \ref{sss-2-1-1}).
	He proved that $P_{M, \mathcal{K}}$ is generically injective.	
	
	Let us briefly recall the invariant $\tau(X, \iota)$ of a $K3^{[2]}$-type manifold with antisymplectic involution $(X, \iota)$ constructed in \cite{I1}.
	Let $h_X$ be an $\iota$-invariant Ricci-flat \K metric on $X$ with normalized volume $1$.
	The equivariant analytic torsion of the cotangent bundle $\Omega^1_X$ with respect to $h_X$ and the $\G$-action induced from $\iota$ is denoted by $\tau_{\iota}(\bar{\Omega}^1_X)$.
	Let $X^{\iota}$ be the fixed locus of $\iota : X \to X$.
	By \cite{MR2805992}, $X^{\iota}$ is a possibly disconnected smooth complex surface.
	The analytic torsion of the trivial line bundle $\mathcal{O}_{X^{\iota}}$ with respect to $h_X|_{X^{\iota}}$ is denoted by $\tau(\bar{\mathcal{O}}_{X^{\iota}})$.
	Let $\vol(X^{\iota}, h_X|_{X^{\iota}})$ be the volume of $X^{\iota}$ with respect to $h_X|_{X^{\iota}}$.
	We denote by $\vol_{L^2}\left( H^1(X^{\iota}, \mathbb{Z}), h_X|_{X^{\iota}} \right)$ the covolume of $H^1(X^{\iota}, \mathbb{Z})$ with respect to the $L^2$-metric on $H^1(X^{\iota}, \mathbb{R})$ induced from $h_X|_{X^{\iota}}$.
	We define a real number $\tau(X, \iota)$ by
	\begin{align*}
		\tau(X, \iota)=\tau_{\iota}(\bar{\Omega}_X^1) 
		 \tau(\bar{\mathcal{O}}_{X^{\iota}})^{-2} \vol(X^{\iota}, h_X|_{X^{\iota}})^{-2} \vol_{L^2}(H^1(X^{\iota}, \mathbb{Z}), h_X|_{X^{\iota}}).
	\end{align*}
	By \cite{I1}, $\tau(X, \iota)$ is independent of the choice of a normalized $\iota$-invariant Ricci-flat \K metric $h_X$.
	Hence $\tau(X, \iota)$ is an invariant of $(X, \iota)$ (See Section \ref{sss-2-1-2} for more details about the invariant).	
	
	Write $\tau_{M, \mathcal{K}}$ for the restriction of $\tau$ to $\tilde{\mathcal{M}}_{M, \mathcal{K}}$.
	Since $\tau_{M, \mathcal{K}}(X, \iota)$ depends only on the isomorphism class of $(X, \iota)$
	and since $P_{M, \mathcal{K}}$ is generically injective, $\tau_{M, \mathcal{K}}$ is viewed as a smooth function on $\mathcal{M}_{M, \mathcal{K}} \setminus \bar{\mathscr{D}}_{M^{\perp}}$.
	Let $C \subset \mathcal{M}^*_{M, \mathcal{K}}$ be an irreducible projective curve, where $\mathcal{M}^*_{M, \mathcal{K}}$ is the Baily-Borel compactification of $\mathcal{M}_{M, \mathcal{K}}$.
	We assume that $C$ is not contained in $\bar{\mathscr{D}}_{M^{\perp}} \cup (\mathcal{M}^*_{M, \mathcal{K}} \setminus \mathcal{M}_{M, \mathcal{K}})$.
	Let $p \in C \cap \bar{\mathscr{D}}_{M^{\perp}}$ be a smooth point of $C$.
	Choose a coordinate $(\D, s)$ on $C$ centered at $p$ such that $\D \cap \bar{\mathscr{D}}_{M^{\perp}} = \{p \}$.
	As an application of the embedding formula for equivariant Quillen metrics \cite{MR1316553} 
	and an equivariant analogue of the formula for the singularity of the $L^2$-metrics \cite{MR4255041},
	we obtain the algebraicity of the singularity of $\tau_{M, \mathcal{K}}$ as follows.

	\begin{thm}[Theorem \ref{t3-2-3-1}]\label{t-0-1-2}
		There exists a constant $a \in \mathbb{Q}$ such that 
		$$
			\log \tau_{M, \mathcal{K}}|_C (s) = a \log |s|^2 + O(\log \log |s|^{-1}) \quad (s \to 0).
		$$
	\end{thm}

\medskip

	Let us give some applications of Theorem \ref{t-0-1-2}.
	Let $M_0$ be a primitive hyperbolic 2-elementary sublattice of the $K3$ lattice $L_{K3}$.
	A K3 surface with antisymplectic involution $(Y, \sigma)$ is called a 2-elementary K3 surface of type $M_0$ 
	if there exists an isometry $\alpha : H^2(X, \mathbb{Z}) \to L_{K3}$ such that $\alpha(H^2(X, \mathbb{Z})^{\sigma})=M_0$.
	Let $(Y, \sigma)$ be a 2-elementary K3 surface of type $M_0$.
	Making use of equivariant analytic torsion of $(Y, \sigma)$ and the analytic torsion of $Y^{\sigma}$, Yoshikawa introduced an invariant $\tau_{M_0}(Y, \sigma)$ and determined its structure as an automorphic form on the moduli space.
	Namely, $\tau_{M_0}$ is expressed as the Petersson norm of an automorphic form obtained as the product of a certain Borcherds product and a certain Siegel modular form
	(\cite{MR4177283}, \cite{MR2047658}, \cite{MR2968220}, \cite{MR3039773}).
	
	Let $Y^{[2]}$ be the Hilbert scheme of 2-points of $Y$ and $\sigma^{[2]} : Y^{[2]} \to Y^{[2]}$ be the involution induced from $\sigma$.
	Then $(Y^{[2]}, \sigma^{[2]})$ is a manifold of $K3^{[2]}$-type with antisymplectic involution.
	Let $M_0^{\perp}$ be the orthogonal complement of $M_0$ in $K_{K3}$.
	Comparing the curvature equation and the boundary behavior, we obtain the following comparison theorem for the invariants $\tau$ and $\tau_{M_0}$.
	
	\begin{thm}[Theorem \ref{t2-5-4}]\label{t-0-1-3}
		Suppose that $\rk(M_0) \leqq 17$
		and that $\Delta(M_0^{\perp})$, the set of roots of $M_0^{\perp}$, consists of a unique $O(M_0^{\perp})$-orbit.
		Then there exists a positive constant $C_{M_0}>0$ depending only on $M_0$ such that, for any 2-elementary K3 surface $(Y, \sigma)$ of type $M_0$, the following identity holds:
		$$
			\tau(Y^{[2]}, \sigma^{[2]})=C_{M_0}\tau_{M_0}(Y, \sigma)^{-2(\rk(M_0)-9)}.
		$$
	\end{thm}
		
	In this case, $\tau(Y^{[2]}, \sigma^{[2]})$  is expressed as the Petersson norm of a certain Borcherds product and a certain Siegel modular form. 
		
	Let us consider the case $M_0= \langle 2 \rangle$.
	Since $\Delta(M_0^{\perp})$ consists of two $O(M_0^{\perp})$-orbits, this case is not covered by Theorem \ref{t-0-1-3}.
	A 2-elementary K3 surface $(Y, \sigma)$ of type $\langle 2 \rangle$ is given by the pair consisting of the double cover $Y \to \mathbb{P}^2$ blanched over a smooth sextic $C$ and its covering involution $\sigma : Y \to Y$.
	The fixed locus of $\sigma^{[2]} : Y^{[2]} \to Y^{[2]}$ is isomorphic to $C^{(2)} \sqcup Y/\sigma$.
	Since $Y/\sigma = \mathbb{P}^2$, we have the Mukai flop $f : Y^{[2]} \dashrightarrow \elm_{Y/\sigma}(Y^{[2]})$, where $\elm_{Y/\sigma}(Y^{[2]})$ is a manifold of $K3^{[2]}$-type (\cite{MR1963559}, \cite{ohashi2013non}).
	We define $\elm_{Y/\sigma}(\sigma^{[2]}) = f \circ \sigma \circ f^{-1}$.
	Then $\elm_{Y/\sigma}(\sigma^{[2]})$ is a biregular involution.
	Moreover, $(\elm_{Y/\sigma}(Y^{[2]}), \elm_{Y/\sigma}(\sigma^{[2]}))$ is an irreducible holomorphic symplectic fourfold of $K3^{[2]}$-type with antisymplectic involution.
	By Joumaah \cite{MR3519981}, $(\elm_{Y/\sigma}(Y^{[2]}), \elm_{Y/\sigma}(\sigma^{[2]}))$ is not deformation equivalent to $(Y^{[2]}, \sigma^{[2]})$,
	and any manifold of $K3^{[2]}$-type with antisymplectic involution with admissible lattice $\langle 2 \rangle \oplus \langle -2 \rangle$ is one of these two types.
	
	\begin{thm}[Theorem \ref{t2-5-6}]\label{t-0-1-4}
		There exist positive constants $C_1, C_2 > 0$ such that for any 2-elementary K3 surface $(Y, \sigma)$ of type $M_0=\langle 2 \rangle$, the following identities hold:
		\begin{align*}
			&\tau(Y^{[2]}, \sigma^{[2]})  =C_1\tau_{M_0}(Y, \sigma)^{16}, \\
			&\tau(\elm_{Y/\sigma}(Y^{[2]}), \elm_{Y/\sigma}(\sigma^{[2]})) =C_2\tau_{M_0}(Y, \sigma)^{16}.
		\end{align*}
	\end{thm}
	

\bigskip

	This paper is organized as follows. 
	In Section 1, we recall the notion of equivariant analytic torsion and study its asymptotic behavior for degenerations of projective manifolds with $\G$-action when the singular fiber has simple normal crossing singularities.
	In Section 2, we recall the invariant $\tau$ of $K3^{[2]}$-type manifolds with antisymplectic involution, construct a 1-parameter degeneration of $K3^{[2]}$-type manifolds with antisymplectic involution over a finite branched covering of a given curve on $\mathcal{M}_{M, \mathcal{K}}$ and prove Theorem \ref{t-0-1-2}. 
	In Section 3, we show that $\tau_{M, \mathcal{K}}$ and $\tau_{M_0}$ satisfy the same curvature equation, where, for a 2-elementary K3 surface $(Y, \sigma)$ of type $M_0$, $(M, \mathcal{K})$ is the type of $(Y^{[2]}, \sigma^{[2]})$. 
	In Section 4, we compare $\tau_{M, \mathcal{K}}$ with $\tau_{M_0}$ and prove Theorems \ref{t-0-1-3} and \ref{t-0-1-4}.
	In appendix written by Yoshikawa, the singularity of the $\G$-equivariant Quillen metric is determined.

\bigskip
	
	\textbf{Acknowledgements}.
	 I would like to express my gratitude to my advisor, Professor Ken-Ichi Yoshikawa, for suggesting this problem and for his help and encouragement. 
	 I am also indebted to him for his generous contribution of his work to the appendix.
	 This work was supported by JSPS KAKENHI Grant Number 23KJ1249.

\section{The singularity of equivariant analytic torsion}\label{s-1}

	In this section, we show that the equivariant analytic torsion of the cotangent bundle has an algebraic singularity.
	For this, we study the singularity of the equivariant Quillen metrics and the $L^2$-metrics on the equivariant determinant of the cohomology.
	We follow the key idea of \cite{MR4255041}, \cite{MR2454893} and \cite{MR2262777}.

\subsection{Equivariant analytic torsion}\label{ss-1-1}

	
	Let $X$ be a compact complex manifold of dimension $n$,
	and let $\iota : X \to X$ be a holomorphic involution of $X$.
	Let $\G$ be the group generated by the order 2 element $\iota$.
	In what follows, we consider the $\mu_2$-action on $X$ induced by $\iota$.
	Let $h_X$ be an $\iota$-invariant \K metric on $X$.
	The \K form attached to $h_X$ is defined by
	$$
		\omega_X = \frac{i}{2} \sum_{j,k} h_X \left( \frac{\partial}{\partial z^j}, \frac{\partial}{\partial z^k} \right) dz^j \wedge d\bar{z}^k ,
	$$
	where $(z^1, \dots, z^n)$ is a system of local coordinates on $X$.
	The space of smooth $(p,q)$-forms on $X$ is denoted by $A^{p,q}(X)$.
	
	Let $E$ be a $\G$-equivariant holomorphic vector bundle on $X$,
	and $h_E$ a $\G$-invariant hermitian metric on $E$.
	The space of $E$-valued smooth $(p,q)$-forms on $X$ is denoted by $A^{p,q}(X, E)$ or $A^{p,q}(E)$.
	
	The metrics $h_X$ and $h_E$ induce a $\G$-invariant hermitian metric $h$ on the complex vector bundle $\wedge^{p,q}T^*X \otimes E$.
	The $L^2$-metric on $A^{p,q}(X, E)$ is defined by
	$$
		\langle \alpha, \beta \rangle_{L^2} = \int_X h(\alpha, \beta) \frac{\omega_X^n}{n !}, \quad \alpha, \beta \in A^{p,q}(X, E).
	$$
	The Dolbeault operator of $E$ is denoted by $\bar{\partial}_E : A^{p,q}(X, E) \to A^{p,q+1}(X, E)$,
	and its formal adjoint is denoted by $\bar{\partial}_E^* : A^{p,q}(X, E) \to A^{p,q-1}(X, E)$.
	We define the Laplacian $D_{p,q}^2$ acting on $A^{p,q}(X, E)$ by 
	$$
		D_{p,q}^2 = (\bar{\partial}_E + \bar{\partial}_E^*)^2 : A^{p,q}(X, E) \to A^{p,q}(X, E).
	$$
	We denote the spectrum of $D_{p,q}^2$ by $\sigma(D_{p,q}^2)$, and the eigenspace of $D_{p,q}^2$ associated with an eigenvalue $\lambda \in \sigma(D_{p,q}^2)$ by $E_{p,q}(\lambda)$.
	Note that $\sigma(D_{p,q}^2)$ is a discrete subset contained in $\mathbb{R}_{\geqq 0}$.
	Moreover $E_{p,q}(\lambda)$ is finite dimensional.
	
	\begin{dfn}\label{d-1-1}
		Let $g \in \G$. 
		The spectral zeta function is defined by 
		$$
			\zeta_{p,q,g}(s) = \sum_{\lambda \in \sigma(D_{p,q}^2) \setminus \{ 0 \}} \lambda^{-s} \operatorname{Tr} (g|_{E_{p,q}(\lambda)}) \quad (s \in \mathbb{C}, \operatorname{Re} s >n).
		$$
	\end{dfn}
	
	Note that $\zeta_{p,q,g}(s)$ converges absolutely on the domain $\operatorname{Re} s > n$
	and extends to a meromorphic function on $\mathbb{C}$ which is holomorphic at $s=0$.
	
	\begin{dfn}\label{d-1-2}
		Let $g \in \G$.
		The equivariant analytic torsion of $\overline{E} = (E, h_E)$ on $(X, \omega_X)$ is defined by
		$$
			\tau_g(\overline{E}) = \exp \left\{ - \sum_{q=0}^n (-1)^q q \zeta'_{0,q,g}(0) \right\} .
		$$
	\end{dfn}
	
	If $g=1$, it is the (usual) analytic torsion of $\overline{E}$ and is denoted by $\tau(\overline{E})$ instead of $\tau_1(\overline{E})$.
	
	We denote by $H^q(X, E)_{\pm}$ the $(\pm 1)$-eigenspace of $\iota^* : H^q(X, E) \to H^q(X, E)$.
	We set
	$$
		\lambda_{\pm}(E) = \bigotimes_{q \geqq 0} ( \det H^q(X, E)_{\pm} )^{(-1)^q}. 
	$$ 
	We define the equivariant determinant of the cohomologies of $E$ by
	$$
		\lambda_{\G}(E) = \lambda_{+}(E) \oplus \lambda_{-}(E). 
	$$
	
	By Hodge theory, we can identify $H^q(X, E)$ with the space of $E$-valued harmonic $(0, q)$-forms on $X$.
	Under this identification, the cohomology $H^q(X, E)$ is endowed with the $\G$-invariant hermitian metric induced from the $L^2$-metric on $A^{p,q}(X, E)$.
	It induces the hermitian metric $\| \cdot \|_{\lambda_{\pm}(E), L^2}$ on $\lambda_{\pm}(E)$.
	We define the equivariant metric on $\lambda_{\G}(E)$ by
	$$
		\| \alpha \|_{\lambda_{\G}(E), L^2}(\iota) = \| \alpha_+ \|_{\lambda_{\pm}(E), L^2} \cdot \| \alpha_- \|_{\lambda_{\pm}(E), L^2}^{-1} \quad ( \alpha=(\alpha_+, \alpha_-) \in \lambda_{\G}(E), \alpha_+, \alpha_- \neq 0 ),
	$$
	and call it the equivariant $L^2$-metric on $\lambda_{\G}(E)$.
	We define the equivariant Quillen metric on $\lambda_{\G}(E)$ by
	$$
		\| \alpha \|^2_{\lambda_{\G}(E), Q}(\iota) = \tau_g(\overline{E}) \| \alpha \|^2_{\lambda_{\G}(E), L^2}(\iota) .
	$$

	Let $\mathscr{X}$ and $S$ be complex manifolds of dimension $m+n$ and $m$, respectively.
	Let $\iota : \mathscr{X} \to \mathscr{X}$ be a holomorphic involution.
	Then $\iota$ induces a $\G$-action on $\mathscr{X}$.
	We consider the trivial $\G$-action on $S$.
	
	Let $f : (\mathscr{X}, \iota) \to S$ be a proper surjective $\G$-equivariant holomorphic submersion.
	Suppose that $f$ is locally K\"ahler.
	Namely, for each point $s \in S$ there is an open neighborhood $U$ of $s$ such that $f^{-1}(U)$ is K\"ahler.
	The fiber of $f$ is denoted by $X_s$ $(s \in S)$ or simply $X$.
	Since $f$ is $\G$-equivariant, the involution $\iota$ induces a holomorphic involution on each fiber $X_s$,
	which is denoted by $\iota_s$ or simply $\iota$. 
	
	Let $h_{\xs}$ be an $\iota$-invariant hermitian metric on the relative tangent bundle $T\xs$ which is fiberwise K\"ahler.
	Set $h_s = h_{\xs}|_{X_s}$ $(s \in S)$.
	This is an $\iota_s$-invariant \K metric on $X_s$.
	We denote its \K form by $\omega_s$
	and we set $\omega_{\xs}= \{ \omega_s \}_{s \in S}$.
	
	Let $\overline{E}=(E, h_E)$ be a $\G$-equivariant holomorphic hermitian vector bundle on $\X$.
	We assume that $R^qf_*E$ is a locally free sheaf for all $q \geqq 0$ and we regard it as a holomorphic vector bundle on $S$.
	By Hodge theory, $R^qf_*E$ is equipped with the $\iota$-invariant hermitian metric.
	This is called the $L^2$-metric and is denoted by $h_{L^2}$.
	
	Let $g \in \G$.
	We define a real-valued function on $S$ by
	$$
		\tau_g(\overline{E})(s) = \tau_g(\overline{E}|_{X_s} ) \quad (s \in S).
	$$
	
	Let $E_{\pm}$ be the $(\pm 1)$-eigenbundle of the $\G$-action on $E|_{\mathscr{X}^{\iota}}$,
	and the restriction of $h_E$ to $E_{\pm}$ is denoted by $h_{\pm}$.
	The curvature form of $(E_{\pm}, h_{\pm})$ is denoted by $R_{\pm}$.
	Recall that the equivariant Todd form and the equivariant Chern character form are differential forms on $\mathscr{X}^{\iota}$ defined by
	\begin{align}\label{al-1-A}
		Td_{\iota}(E, h_E) = Td \left(- \frac{R_+}{2\pi i} \right) \det \left( \frac{I}{I+ \exp({+\frac{R_-}{2\pi i}})} \right),
	\end{align}
	and
	\begin{align}\label{al-1-B}
		ch_{\iota}(E, h_E) = ch \left(- \frac{R_+}{2\pi i} \right) - ch \left(- \frac{R_-}{2\pi i} \right) ,
	\end{align}
	respectively.
	If $\alpha$ is a differential form, then $[\alpha]^{(p,q)}$ is the component of $\alpha$ of bidegree $(p,q)$.
	
	The $(\pm 1)$-eigenbundle of $R^qf_*E$ is denoted by $(R^qf_*E)_{\pm}$.
	We set
	$$
		\lambda_{\pm}(E) = \bigotimes_{q \geqq 0} \det (R^qf_*E)_{\pm}.
	$$
	We define the equivariant determinant of the cohomologies of $E$ by
	$$
		\lambda_{\G}(E) = \lambda_{+}(E) \oplus \lambda_{-}(E).
	$$
	It is equipped with the equivariant $L^2$-metric and the equivariant Quillen metric.
	For an open subset $U$ of $S$, a holomorphic section $\sigma = (\sigma_+, \sigma_-)$ is called an admissible section if both $\sigma_+$ and $\sigma_-$ are nowhere vanishing on $U$.

\subsection{K\"ahler extension and the singularity of equivariant Quillen metrics}\label{ss-1-4}

	Let $\mathscr{X}$ be a smooth projective manifold of dimension $n+1$
	and let $\iota : \X \to \X$ be a holomorphic involution.
	In what follows, we consider the $\mu_2$-action on $X$ induced by $\iota$.
	Let $C$ be a smooth projective curve with trivial $\G$-action,
	and let $f : \mathscr{X} \to C$ be a $\G$-equivariant holomorphic surjection.
	Then the $\G$-action preserves the fibers of $f$.
	
	Let $\Sigma_f = \{ x \in \X ; df(x)=0 \}$ be the critical locus of $f$ 
	and let $\Delta_f = f(\Sigma_f)$ be the discriminant locus of $f$.
	The fiber of $f$ at $s \in C$ is denoted by $X_s$.
	We fix a general fiber $X_{\infty}$.
	We set $C^{\circ}=C \setminus \Delta_f$, $\X^{\circ}=f^{-1}(C^{\circ})$, and $f^{\circ}=f|_{\X^{\circ}}$.
	Then $f^{\circ} : \X^{\circ} \to C^{\circ}$ is a family of projective manifolds with $\G$-action.
	
	Let $h_{\X}$ be a $\G$-invariant \K metric on $\X$.
	We assume that the \K form of $h_{\X}$ is integral.
	The relative tangent bundle $T\X/C$ is the $\G$-equivariant subbundle of $T\X|_{\X \setminus \Sigma_f}$ defined by
	$$
		T\X/C = \operatorname{Ker} (f_*|_{\X \setminus \Sigma_f} : T\X|_{\X \setminus \Sigma_f} \to f^*TC|_{\X \setminus \Sigma_f}).
	$$
	We set $h_{\X/C} = h_{\X}|_{T\X/C}$.
	
	Fix $0 \in \Delta_f$ and let $(\varDelta, s)$ be a local coordinate on $C$ centered at $0$ such that $\D \cap \Delta_f = \{ 0 \}$ and that $s(\D)$ is the unit disk in $\mathbb{C}$.
	We set $X = f^{-1}(\D)$ and $\D^* = \D \setminus \{0\}$.
	
	\begin{assumption}\label{as-1-2-1}
		We assume that the singular fiber $X_0$ is simple normal crossing.
		Namely, for each point of $X_0$ there exists a coordinate neighborhood $(U, z_1, \dots, z_{n+1})$ on $\X$ such that $X_0 \cap U$ is given by $z_1 \cdots z_k =0$ and each irreducible component of $X_0$ is smooth. 
	\end{assumption}		

	The $\G$-invariant hermitian metric on $\Omega^1_{\X}$ induced from the metric $h_{\X}$ on $T\X$ is also denoted by $h_{\X}$.
	The $\G$-invariant hermitian metric on $\Omega^1_{\X/C} = \Omega^1_{\X}|_{\X \setminus \Sigma_f} / f^* \Omega^1_C|_{\X \setminus \Sigma_f}$ induced from the metric $h_{\X/C}$ on $T\X/C$ is also denoted by $h_{\X/C}$.
	
	We define a hermitian metric $h_{f^* \Omega^1_C}$ on $f^* \Omega^1_C|_{\X \setminus \Sigma_f}$ by $h_{f^* \Omega^1_C} = h_{\X}|_{f^* \Omega^1_C|_{\X \setminus \Sigma_f}}$.
	Let $h_C$ be a hermitian metric on $\Omega^1_C$ such that $h_C(ds, ds) =1$ on $\D$.
	Then $f^*h_C$ is a hermitian metric on $f^* \Omega^1_C$.
	We identify $f$ with $s \circ f$.
	Then $df (= d(s \circ f))$ is a holomorphic 1-form, and
	\begin{align*}
		f^*h_C(df, df) = f^*( h_C(ds, ds) )=1, \quad
		h_{f^*\Omega^1_C}(df, df) = h_{\X}(df, df) = \| df \|^2.
	\end{align*} 
	Therefore we have
	\begin{align}\label{f-1-4}
		h_{f^*\Omega^1_C} = \| df \|^2 f^*h_C.
	\end{align}
	
	Let $\widetilde{\Omega}^1_{\X/C}$ be the complex of locally free sheaves on $\X$ defined by
	$$
		\widetilde{\Omega}^1_{\X/C} : f^* \Omega^1_C \to \Omega^1_{\X}.
	$$
	Since 
	$$
		0 \to f^* \Omega^1_C|_{\X^{\circ}} \to \Omega^1_{\X^{\circ}} \to \Omega^1_{\X^{\circ}/C^{\circ}} \to 0
	$$
	is a short exact sequence of $\G$-equivariant vector bundles on $\X^{\circ}$, we have the following canonical isomorphism of holomorphic vector bundles of rank $2$ on $C^{\circ}$
	$$
		\lambda_{\G}(\Omega^1_{\X^{\circ}/C^{\circ}}) \cong \lambda_{\G}(\Omega^1_{\X^{\circ}}) \otimes \lambda_{\G}(f^* \Omega^1_C)^{-1}|_{C^{\circ}}.
	$$
	Namely,
	$$
		\lambda_{\pm}(\Omega^1_{\X^{\circ}/C^{\circ}}) \cong \lambda_{\pm}(\Omega^1_{\X^{\circ}}) \otimes \lambda_{\pm}(f^* \Omega^1_C)^{-1}|_{C^{\circ}}.
	$$
	
	\begin{dfn}\label{d-1-5}
		We define the K\"ahler extension of $\lambda_{\G}(\Omega^1_{\X^{\circ}/C^{\circ}})$ by 
		\begin{align}\label{al6-1-2-1}
			\lambda_{\G}(\widetilde{\Omega}^1_{\X/C}) = \lambda_{+}(\widetilde{\Omega}^1_{\X/C}) \oplus \lambda_{-}(\widetilde{\Omega}^1_{\X/C}), \quad \lambda_{\pm}(\widetilde{\Omega}^1_{\X/C}) =  \lambda_{\pm}(\Omega^1_{\X}) \otimes \lambda_{\pm}(f^* \Omega^1_C)^{-1}.
		\end{align}
	\end{dfn}
	
	Clearly, we have the canonical isomorphism of holomorphic vector bundles on $C^{\circ}$
	\begin{align}\label{f-1-6}
		\lambda_{\G}(\widetilde{\Omega}^1_{\X/C})|_{C^{\circ}} \cong \lambda_{\G}(\Omega^1_{\X^{\circ}/C^{\circ}}).
	\end{align}
	
	Let $\Pi : \mathbb{P}(\Omega^1_{\X}) \to \X$ be the projective bundle associated with $\Omega^1_{\X}$ and $\Pi^{\vee} : \mathbb{P}(T\X)^{\vee} \to \X$ be the projective bundle associated with $T\X$,
	where the fiber $ \mathbb{P}(T_x\X)^{\vee}$ is the set of all hyperplanes of $T_x\X$.
	
	The Gauss maps $\nu : \X \setminus \Sigma_f \to \mathbb{P}(\Omega^1_{\X})$ and $\mu : \X \setminus \Sigma_f \to \mathbb{P}(T\X)^{\vee}$ are defined by
	$$
		\nu(x) = [df] = \left[ \sum^n_{i=0} \frac{\partial f}{\partial z_i} (x) dz_i \right], \quad \mu(x)=[T_xX_{f(x)}].
	$$
	Under the canonical isomorphism $\mathbb{P}(\Omega^1_{\X}) \cong \mathbb{P}(T\X)^{\vee}$, we have $\nu = \mu$.
	
	Let $L = \mathcal{O}_{\mathbb{P}(\Omega^1_{\X})}(-1) \subset \Pi^* \Omega^1_{\X}$ be the tautological line bundle on $\mathbb{P}(\Omega^1_{\X})$ 
	and set $Q = \Pi^* \Omega^1_{\X} / L$.
	We have the following short exact sequence of holomorphic vector bundles on $\Pi^* \Omega^1_{\X}$:
	$$
		\mathcal{S} : 0 \to L \to  \Pi^* \Omega^1_{\X}  \to Q \to 0.
	$$
	
	Let $H=\mathcal{O}_{\mathbb{P}(T\X)^{\vee}}(1)$ and let $U$ be the universal hyperplane bundle of $(\Pi^{\vee})^*T\X$.
	We have the dual of $\mathcal{S}$ given by
	$$
		\mathcal{S}^{\vee} : 0 \to U \to (\Pi^{\vee})^*T\X \to H \to 0.
	$$
	
	The vector bundle $ \Pi^* \Omega^1_{\X}$ is endowed with the hermitian metric $\Pi^* h_{\X}$.
	Let $h_L$ and $h_Q$ be its induced and quotient metrics on $L$ and $Q$, respectively.
	The vector bundle $(\Pi^{\vee})^*T\X$ is endowed with the hermitian metric $(\Pi^{\vee})^*h_{\X}$.
	Let $h_U$ and $h_H$ be its induced and quotient metrics on $U$ and $H$, respectively.
	We denote by $\overline{\mathcal{S}}$ the short exact sequence of holomorphic hermitian vector bundles
	$$
		\overline{\mathcal{S}} = (\mathcal{S}, (h_L, \Pi^*h_{\X}, h_Q)) .
	$$  
	
	We have the following isometries of holomorphic hermitian vector bundles over $\X \setminus \Sigma_f$:
	\begin{align}\label{f-1-9}
		\overline{\mathcal{F}} = \nu^* \overline{\mathcal{S}}, \quad (T\X/C, h_{\X/C}) = \mu^*(U, h_U),
	\end{align}
	where $\overline{\mathcal{F}}$ is the short exact sequence of holomorphic hermitian vector bundles over $\X \setminus \Sigma_f$ given by
	$$
		\mathcal{F} : 0 \to f^* \Omega^1_C \to \Omega^1_{\X} \to \Omega^1_{\X/C} \to 0 \quad \text{and} \quad \overline{\mathcal{F}} =(\mathcal{F}, (h_{f^* \Omega^1_C}, h_{\Omega^1_{\X}}, h_{\X/C})).
	$$
	Since $df$ is a nowhere vanishing holomorphic section of $\nu^*L|_{\X \setminus \Sigma_f}$, we have the following equation of $(1,1)$-forms on $\X \setminus \Sigma_f$:
	\begin{align}\label{f-1-10}
		\nu^* c_1(L, h_L) = -dd^c \log \| df \|^2 ,
	\end{align}  
	where $\displaystyle d^c = \frac{\partial -\bar{\partial}}{4 \pi i}$.
 	
	Note that $\Sigma_f$ is a proper analytic subset of $\X$.
	By \cite[Theorem 4.5.3]{1360011145989358208}, the Gauss maps $\nu$ and $\mu$ extend to meromorphic maps $\nu : \X \dashrightarrow \mathbb{P}(\Omega^1_{\X})$ and $\mu : \X \dashrightarrow \mathbb{P}(T\X)^{\vee}$.
	Then there exist a smooth projective manifold $\widetilde{\X}$,
	a normal crossing divisor $E$ on $\widetilde{\X}$, 
	a bimeromorphic morphism $q : \widetilde{\X} \to \X$,
	and two holomorphic maps $\widetilde{\nu} : \widetilde{\X} \to \mathbb{P}(\Omega^1_{\X})$, $\widetilde{\mu} : \widetilde{\X} \to \mathbb{P}(T\X)^{\vee}$
	such that
	\begin{itemize}
		\item $q^{-1}(\Sigma_f) =E$,
		\item the restriction $ q|_{\widetilde{\X} \setminus E} : \widetilde{\X} \setminus E \to \X \setminus \Sigma_f $ is an isomorphism,
		\item on $\widetilde{\X} \setminus E$, $\widetilde{\nu} = \nu \circ q$ and $\widetilde{\mu} = \mu \circ q$.
	\end{itemize}
	
	We set $E_0 = E \cap q^{-1}(X_0)$.
	The fixed locus $\X^{\iota}$ admits the decomposition $\X^{\iota} =  \X^{\iota}_H \sqcup \X^{\iota}_V$,
	where any connected component of $\X^{\iota}_H$ is flat over $C$ and $\X^{\iota}_V \subset f^{-1}(\Delta_f)$.
	Let $\widetilde{\X}^{\iota}_H \subset \widetilde{\X}$ be the proper transform of $\X^{\iota}_H$.
	We define topological invariants of the germs $( f : (\X, X_0) \to (C, 0), \Omega^1_{\X})$ and $( f : (\X, X_0) \to (C, 0), f^*\Omega^1_{C})$ by
	\begin{align*}
		\alpha_{\iota}(X_0, \Omega^1_{\X}) =& \int_{E_0 \cap \widetilde{\X}^{\iota}_H} \widetilde{\nu}^* \left\{ \frac{1-Td(H)^{-1}}{c_1(H)} \right\} q^*\left\{ Td_{\iota}(T\X) ch_{\iota}(\Omega^1_{\X}) \right\} \\
		&\quad - \int_{ \X^{\iota}_V \cap X_0} Td_{\iota}(T\X) ch_{\iota}(\Omega^1_{\X}), \\
		\alpha_{\iota}(X_0, f^*\Omega^1_{C}) =& \int_{E_0 \cap \widetilde{\X}^{\iota}_H} \widetilde{\nu}^* \left\{ \frac{1-Td(H)^{-1}}{c_1(H)} \right\} q^*\left\{ Td_{\iota}(T\X) ch_{\iota}(f^*\Omega^1_{C}) \right\} \\
		&\quad - \int_{ \X^{\iota}_V \cap X_0} Td_{\iota}(T\X) ch_{\iota}(f^*\Omega^1_{C}),
	\end{align*}
	
	The equivariant Quillen metric on $\lambda_{\G}(\Omega^1_{\X^{\circ}/C^{\circ}})$ with respect to $h_{\X/C}$ is denoted by $ \| \cdot \|^2_{ \lambda_{\G}(\Omega^1_{\X^{\circ}/C^{\circ}}), Q, h_{\X/C} }(\iota) $,
	the equivariant Quillen metric on $\lambda_{\G}(\Omega^1_{\X})$ with respect to $h_{\X}$ is denoted by $ \| \cdot \|^2_{ \lambda_{\G}(\Omega^1_{\X}), Q, h_{\X} }(\iota) $,
	the equivariant Quillen metric on $\lambda_{\G}(f^* \Omega^1_{C})|_{C^{\circ}}$ with respect to $h_{f^* \Omega^1_C}$ is denoted by $ \| \cdot \|^2_{ \lambda_{\G}(f^*\Omega^1_{C}), Q, h_{f^* \Omega^1_C} }(\iota) $,
	and the equivariant Quillen metric on $\lambda_{\G}(f^* \Omega^1_{C})$ with respect to $f^*h_{C}$ is denoted by $ \| \cdot \|^2_{ \lambda_{\G}(f^* \Omega^1_{C}), Q, f^* h_{C} }(\iota) $.
	
	\begin{dfn}\label{d-1-7}
		\begin{enumerate}[ label= \rm{(\arabic*)} ]
			\item	We denote by $ \| \cdot \|^2_{ \lambda_{\G}(\widetilde{\Omega}^1_{\X/C}), Q,  h_{\X/C} }(\iota) $ the equivariant Quillen metric on $\lambda_{\G}(\widetilde{\Omega}^1_{\X/C})|_{C^{\circ}}$ 
			induced from $ \| \cdot \|^2_{ \lambda_{\G}(\Omega^1_{\X^{\circ}/C^{\circ}}), Q, h_{\X/C} }(\iota) $ by the isomorphism (\ref{f-1-6}).
			\item Let $ \| \cdot \|^2_{ \lambda_{\G}(\widetilde{\Omega}^1_{\X/C}), Q, h_{f^* \Omega^1_C} }(\iota) $ be the $\G$-equivariant metric on $\lambda_{\G}(\widetilde{\Omega}^1_{\X/C})|_{C^{\circ}}$ defined by
			$$
				\| \cdot \|^2_{ \lambda_{\G}(\widetilde{\Omega}^1_{\X/C}), Q, h_{f^* \Omega^1_C} }(\iota) = \| \cdot \|^2_{ \lambda_{\G}(\Omega^1_{\X}), Q, h_{\X} }(\iota)  \otimes \| \cdot \|^{-2}_{ \lambda_{\G}(f^* \Omega^1_{C}), Q, h_{f^* \Omega^1_C} }(\iota)
			$$
			via the canonical isomorphism (\ref{al6-1-2-1}).
			\item Let $ \| \cdot \|^2_{ \lambda_{\G}(\widetilde{\Omega}^1_{\X/C}), Q, f^*h_{C} }(\iota) $ be the $\G$-equivariant metric on $\lambda_{\G}(\widetilde{\Omega}^1_{\X/C})$ defined by
			$$
				\| \cdot \|^2_{ \lambda_{\G}(\widetilde{\Omega}^1_{\X/C}), Q, f^*h_{C} }(\iota) = \| \cdot \|^2_{ \lambda_{\G}(\Omega^1_{\X}), Q, h_{\X} }(\iota)  \otimes \| \cdot \|^{-2}_{ \lambda_{\G}(f^* \Omega^1_{C}), Q, f^*h_{C} }(\iota).
			$$
			via the canonical isomorphism (\ref{al6-1-2-1}).
		\end{enumerate}
	\end{dfn}
	
	Let $\sigma_{\X} = (\sigma_{\X}^+, \sigma_{\X}^-)$ be an admissible section of $\lambda_{\G}(\Omega^1_{\X})|_{\D}$,
	and let $\sigma_C = (\sigma_C^+, \sigma_C^-)$ be an admissible section of $\lambda_{\G}(f^* \Omega^1_{C})|_{\D}$.
	We define an admissible section $\sigma$ of $\lambda_{\G}(\widetilde{\Omega}^1_{\X/C})|_{\D}$ by $\sigma=\sigma_{\X} \otimes \sigma_C^{-1}$.
	
	For $f, g \in C^{\infty}(\D^*)$, we write $f \equiv g$ if $f-g \in C^{0}(\D)$. 
	
	\begin{prop}\label{p-1-8}
		The following equalities hold:
		\begin{align*}
			\log \| \sigma_{\X} (s)\|^2_{ \lambda_{\G}(\Omega^1_{\X}), Q, h_{\X} }(\iota) &\equiv \alpha_{\iota}(X_0, \Omega^1_{\X}) \log |s|^2, \\
			\log \| \sigma_{C} (s)\|^2_{ \lambda_{\G}( f^*\Omega^1_{C}), Q, f^*h_{C} }(\iota) &\equiv \alpha_{\iota}(X_0, f^*\Omega^1_{C}) \log |s|^2.
		\end{align*}
	\end{prop}
	
	\begin{proof}
		This proposition follows from Theorem \ref{t-ap-2-8} in Appendix.
	\end{proof}
	
	Let $f_*$ be the integration along the fibers of $f : \XX \to C$. 
	
	\begin{prop}\label{p-1-11}
		The following equality of functions on $\D$ holds:
		$$
			\log \left( \frac{\| \cdot \|^2_{ \lambda_{\G}(\widetilde{\Omega}^1_{\X/C}), Q,  h_{\X/C} }(\iota)}{\| \cdot \|^2_{ \lambda_{\G}(\widetilde{\Omega}^1_{\X/C}), Q, h_{f^* \Omega^1_C} }(\iota)} \right) \equiv 0.
		$$
	\end{prop}
	
	\begin{proof}
		Note that $\mathcal{F}$ is acyclic on $\X^{\circ}$.
		By the anomaly formula for equivariant Quillen metrics \cite[Theorem 2.5.]{MR1316553}, we have
		$$
			\log \left( \frac{\| \cdot \|^2_{ \lambda_{\G}(\widetilde{\Omega}^1_{\X/C}), Q,  h_{\X/C} }(\iota)}{ \| \cdot \|^2_{ \lambda_{\G}(\widetilde{\Omega}^1_{\X/C}), Q, h_{f^* \Omega^1_C} }(\iota)} \right)
			= \left[ f_* \left( Td_{\iota}(T\X/C, h_{\X/C}) \widetilde{ch}_{\iota}(\overline{\mathcal{F}}) \right) \right]^{(0,0)},
		$$
		where $\widetilde{ch}_{\iota}(\overline{\mathcal{F}})$ is the Bott-Chern secondary class \cite[e), f)]{MR929146} such that
		$$
			-dd^c\widetilde{ch}_{\iota}(\overline{\mathcal{F}}) = ch_{\iota}(f^* \Omega^1_C, h_{f^* \Omega^1_C}) +ch_{\iota}(\Omega^1_{\X/C}, h_{\X/C}) -ch_{\iota}(\Omega^1_{\X}, h_{\Omega^1_{\X}}) .
		$$
		We set $\widetilde{f} = f \circ q : \widetilde{\X} \to C$.
		The integration along the fibers of $f$ and $\widetilde{f}$ are denoted by $f_*$ and $(\widetilde{f})_*$, respectively.
		For any differential form $\alpha$ on $\widetilde{\X}^{\iota} \setminus q^{-1}(\Sigma_f)$, we have $(\widetilde{f})_*(\alpha) = f_* \left( (q^{-1})^* \alpha \right)$.
		On $\D^*$, we deduce from (\ref{f-1-9}) that
		\begin{align*}
			  f_* \left( Td_{\iota}(T\X/C, h_{\X/C}) \widetilde{ch}_{\iota}(\overline{\mathcal{F}}) \right) 
			&= f_* \left( \mu^* Td_{\iota}(U, h_U) \nu^* \widetilde{ch}_{\iota}(\overline{\mathcal{S}}) \right)  \\
			&= f_* (q^{-1})^* \left( \widetilde{\mu}^* Td_{\iota}(U, h_U) \widetilde{\nu}^* \widetilde{ch}_{\iota}(\overline{\mathcal{S}}) \right) \\
			&= (\widetilde{f})_* \left( \widetilde{\mu}^* Td_{\iota}(U, h_U) \widetilde{\nu}^* \widetilde{ch}_{\iota}(\overline{\mathcal{S}}) \right).
		\end{align*}
		Since $\widetilde{\mu}^* Td_{\iota}(U, h_U) \widetilde{\nu}^* \widetilde{ch}_{\iota}(\overline{\mathcal{S}})$ is a smooth differential form on $\widetilde{\X}^{\iota}$ and since $ \widetilde{f} : \widetilde{\X}^{\iota} \to C $ is a proper holomorphic map, 
		$\left[ (\widetilde{f})_* \left( \widetilde{\mu}^* Td_{\iota}(U, h_U) \widetilde{\nu}^* \widetilde{ch}_{\iota}(\overline{\mathcal{S}}) \right) \right]^{(0,0)}$ is a continuous function on $\D$ by \cite[Th\'eor\`eme 4 bis.]{MR0666639}.
		Therefore the function $\left[ f_* \left( Td_{\iota}(T\X/C, h_{\X/C}) \widetilde{ch}_{\iota}(\overline{\mathcal{F}}) \right) \right]^{(0,0)}$ on $\D^*$ extends to 
		a continuous function on $\D$. 
		This completes the proof.
 	\end{proof}
	
	We set 
	$$
		\beta_\iota(X_0, \widetilde{\Omega}^1_{\X/C}) = \int_{E_0 \cap \widetilde{\X}^{\iota}_H} \widetilde{\mu}^* Td_{\iota}(U, h_U) \frac{\exp(\widetilde{\nu}^* c_1(L, h_L)) -1}{\widetilde{\nu}^* c_1(L, h_L)}.
	$$
	
	\begin{prop}\label{p-1-12}
		The following equality of functions on $\D$ holds:
		$$
			\log \left( \frac{ \| \cdot \|^2_{ \lambda_{\G}(\widetilde{\Omega}^1_{\X/C}), Q, h_{f^* \Omega^1_C} }(\iota) }{ \| \cdot \|^2_{ \lambda_{\G}(\widetilde{\Omega}^1_{\X/C}), Q, f^*h_{C} }(\iota) } \right) \equiv \beta_\iota(X_0, \widetilde{\Omega}^1_{\X/C}) \log |s|^2.
		$$
	\end{prop}
	
	\begin{proof}
		Let $\widetilde{ch}_{\iota}(f^* \Omega^1_C ; f^*h_C, h_{f^* \Omega^1_C})$ and $\widetilde{ch}(f^* \Omega^1_C ; f^*h_C, h_{f^* \Omega^1_C})$ be the Bott-Chern secondary classes \cite[e), f)]{MR929146} such that
		\begin{align*}
			dd^c \widetilde{ch}_{\iota}(f^* \Omega^1_C ; f^*h_C, h_{f^* \Omega^1_C}) &= ch_{\iota}(f^* \Omega^1_C, f^*h_C) -ch_{\iota}(f^* \Omega^1_C , h_{f^* \Omega^1_C}), \\
			dd^c \widetilde{ch}(f^* \Omega^1_C ; f^*h_C, h_{f^* \Omega^1_C}) &= ch(f^* \Omega^1_C, f^*h_C) -ch(f^* \Omega^1_C , h_{f^* \Omega^1_C})
		\end{align*}
		hold on $\XX \setminus \Sigma_f$ and $\X \setminus \Sigma_f$, respectively.
		By the anomaly formula for equivariant Quillen metrics \cite[Theorem 2.5.]{MR1316553}, we have
		\begin{align*}
			\log \left( \frac{ \| \cdot \|^2_{ \lambda_{\G}(\widetilde{\Omega}^1_{\X/C}), Q, h_{f^* \Omega^1_C} }(\iota) }{ \| \cdot \|^2_{ \lambda_{\G}(\widetilde{\Omega}^1_{\X/C}), Q, f^*h_{C} }(\iota) } \right) 
			&=\log \left( \frac{ \| \cdot \|^2_{ \lambda_{\G}(f^* \Omega^1_C), Q, f^*h_C }(\iota) }{ \| \cdot \|^2_{ \lambda_{\G}(f^* \Omega^1_C), Q, h_{f^* \Omega^1_C} }(\iota) } \right) \\
			&=\left[ f_* \left( Td_{\iota}(T\X/C, h_{\X/C}) \widetilde{ch}_{\iota}(f^* \Omega^1_C ; f^*h_C, h_{f^* \Omega^1_C}) \right) \right]^{(0,0)}.
		\end{align*}
		Since $\iota$ acts trivially on $f^* \Omega^1_C|_{ \XX \setminus \Sigma_f }$, we have
		$$
			\widetilde{ch}_{\iota}(f^* \Omega^1_C ; f^*h_C, h_{f^* \Omega^1_C}) = \widetilde{ch}(f^* \Omega^1_C ; f^*h_C, h_{f^* \Omega^1_C})|_{\XX \setminus \Sigma_f}.
		$$
		By \cite[(5.5)]{MR2454893}, we have
		\begin{align}\label{al6-1-2-2}
			\widetilde{ch}(f^* \Omega^1_C ; f^*h_C, h_{f^* \Omega^1_C}) = \frac{\exp(\nu^* c_1(L, h_L)) -1}{\nu^* c_1(L, h_L)} \log \| df \|^2.
		\end{align}
		%
		%
		By (\ref{f-1-9}) and (\ref{al6-1-2-2}), we have
		\begin{align*}
			& \quad \left[ f_* \left( Td_{\iota}(T\X/C, h_{\X/C}) \widetilde{ch}_{\iota}(f^* \Omega^1_C ; f^*h_C, h_{f^* \Omega^1_C}) \right) \right]^{(0,0)} \\
			&=\left[ f_* \left( \mu^* Td_{\iota}(U, h_{U}) \frac{\exp(\nu^* c_1(L, h_L)) -1}{\nu^* c_1(L, h_L)} \log \| df \|^2 \right) \right]^{(0,0)} \\
			&=\left[ \widetilde{f}_* \left( \widetilde{\mu}^* Td_{\iota}(U, h_{U}) \frac{\exp(\widetilde{\nu}^* c_1(L, h_L)) -1}{\widetilde{\nu}^* c_1(L, h_L)} q^* \log \| df \|^2 \right) \right]^{(0,0)} \\
			&\equiv \beta_\iota(X_0, \widetilde{\Omega}^1_{\X/C}) \log |s|^2,
		\end{align*}
		where the last equivalence holds by \cite[Lemma 4.4 and Corollary 4.6]{MR2262777}.
	\end{proof}
	
	We define a topological invariant $\gamma_{\iota}(X_0, \widetilde{\Omega}^1_{\X/C})$ by
	$$
		\gamma_{\iota}(X_0, \widetilde{\Omega}^1_{\X/C}) =  \alpha_{\iota}(X_0, \Omega^1_{\X})  - \alpha_{\iota}(X_0, f^* \Omega^1_{C}) + \beta_\iota(X_0, \widetilde{\Omega}^1_{\X/C}) .
	$$
	By Definition \ref{d-1-7} $(3)$ and by Propositions \ref{p-1-8}, \ref{p-1-11}, and \ref{p-1-12}, we obtain the following result.
	
	\begin{prop}\label{p-1-13}
		The following equality of functions on $\D$ holds:
		$$
			\log \| \sigma (s) \|^2_{ \lambda_{\G}(\widetilde{\Omega}^1_{\X/C}), Q,  h_{\X/C} }(\iota) \equiv \gamma_{\iota}(X_0, \widetilde{\Omega}^1_{\X/C}) \log |s|^2.
		$$
	\end{prop}

\subsection{Logarithmic extension and singularity of equivariant $L^2$-metrics}
	
	Let $\mathcal{V}$ be a flat vector bundle on the punctured disk $\D^*$.
	The Gauss-Manin connection $\nabla$ is defined by
	$$
		\nabla( v \otimes f ) = v \otimes d f,
	$$
	where $v$ is a flat section of $\mathcal{V}$ and $f$ is a holomorphic function.
	The lower Deligne extension $(\mathcal{V}_{\log}, \tilde{\nabla})$ is a unique extension of $(\mathcal{V}, \nabla)$ such that
	\begin{enumerate}[ label= \rm{(\arabic*)} ]
		\item	$\mathcal{V}_{\log}$ is a holomorphic vector bundle on $\D$ with $\mathcal{V}_{\log}|_{\D^*} = \mathcal{V}$,
		\item $\tilde{\nabla} : \mathcal{V}_{\log} \to \mathcal{V}_{\log} \otimes \Omega^1_{\D}(\log [0])$ is a logarithmic connection with $\tilde{\nabla}|_{\D^*} =\nabla$,
		\item If $\alpha$ is an eigenvalue of the residue map $\operatorname{res}_0(\nabla)$, then $\alpha \in \mathbb{Q} \cap \left[ 0, 1\right)$.
	\end{enumerate}
	Here the residue map $\operatorname{res}_0(\nabla) : \mathcal{V}_{\log}(0) \to \mathcal{V}_{\log}(0)$ is induced by
	$$
		(1 \otimes R_0) \circ \tilde{\nabla} :  \mathcal{V}_{\log} \to \mathcal{V}_{\log} \otimes \Omega^1_{\D}(\log [0]) \to \mathcal{V}_{\log},
	$$
	where $\mathcal{V}_{\log}(0) = \mathcal{V}_{\log} / I_0 \mathcal{V}_{\log}$ with $I_0 \subset \mathcal{O}_{\D}$ being the ideal sheaf of $0 \in \D$ and $R_0 :  \Omega^1_{\D}(\log [0]) \to \mathbb{C}$ is the Poincar\'e residue map given by $R_0(f \frac{ds}{s}) = f(0)$.
	
	Following Eriksson-Freixas i Montplet-Mourougane \cite[\S 2]{MR4255041}, let us briefly recall the explicit construction of $\mathcal{V}_{\log}$.
	Let $q : \mathbb{H} \to \D^*$ be the universal covering defined by $q(\tau) = \exp(2 \pi i \tau)$, where $\mathbb{H}$ is the upper half plane.
	Let $\nabla : \mathcal{V} \to \mathcal{V} \otimes \Omega^1_{\D^*}$ be the Gauss-Manin connection defined by
	$$
		\nabla( v \otimes f ) = v \otimes d f,
	$$
	where $v$ is a flat section of $\mathcal{V}$ and $f$ is a holomorphic function.
	We define the space $V_{\infty}$ of flat sections of $\mathcal{V}$ by
	$$
		V_{\infty} = \operatorname{Ker} ( q^*\nabla : \Gamma( \mathbb{H}, q^*\mathcal{V}) \to \Gamma( \mathbb{H}, q^*\mathcal{V} \otimes \Omega^1_{\mathbb{H}}) ).
	$$
	The monodromy transformation $T \in \operatorname{End} V_{\infty} $ is defined by $T \sigma (\tau) = \sigma(\tau+1)$.
	By the monodromy theorem \cite[(6.1)]{MR0382272}, $T$ is quasi-unipotent.
	By the Jordan decomposition, we have $T=T_sT_u=T_uT_s$,
	where $T_s$ is a semi-simple endomorphism and $T_u$ is a unipotent endomorphism.
	The logarithm of $T_u$ is well-defined.
	We choose the lower branch of logarithm that takes values in $\mathbb{R} + 2 \pi i \left( -1, 0 \right]$
	and denote it by $^l\log$.
	Then we can define $^l\log T_s$ and we set 
	$$
		\frac{1}{2 \pi i}\log T = \frac{1}{2 \pi i} {}^l\log T_s + \frac{1}{2 \pi i}\log T_u .
	$$
	
	For $e \in V_{\infty}$, we define $\widetilde{e} \in \Gamma( \mathbb{H}, q^*\mathcal{V})$ by
	$$
		\widetilde{e} (\tau) = \exp \left( -\tau \log T \right) e(\tau).
	$$
	Since $\widetilde{e} (\tau +1) =\widetilde{e} (\tau) $, we may regard $\widetilde{e}$ as the section of $\mathcal{V}$ over $\D^*$.
	Therefore we have an injective $\mathbb{C}$-linear map of vector spaces
	$$
		\phi : V_{\infty} \to \Gamma(\D^*, \mathcal{V}), \quad e \mapsto \widetilde{e}(q),
	$$   
	which induces an isomorphism of vector bundles on $\D^*$
	\begin{align}\label{al6-1-3-5}
		V_{\infty} \otimes \mathcal{O}_{\D^*} \cong \mathcal{V}.
	\end{align}
	We define $\mathcal{V}_{\log}$ by
	$$
		\mathcal{V}_{\log} = \phi(V_{\infty}) \otimes \mathcal{O}_{\D}.
	$$
	By taking the fiber of $\mathcal{V}_{\log}$ at $0$, we have an isomorphism of vector spaces
	\begin{align}\label{al-1-5-1}
		\psi : V_{\infty} \to \mathcal{V}_{\log}(0), \quad e \to \widetilde{e}(0).
	\end{align}

	Consider the normal crossing degeneration $f : X=f^{-1}(\D) \to \D$.
	We recall Steenbrink's construction of the lower extension of the Hodge bundle induced from $f : X^{\circ} =f^{-1}(\D^*) \to \D^*$.
	For more details, we refer to \cite[\S 2]{MR0429885}, \cite[\S 2]{MR485870} and \cite[\S 2]{MR4255041}.
	
	We define locally free sheaves on $X =f^{-1}(\D)$ by
	\begin{align*}
		\Omega^1_{X/ \D}(\log) = \Omega^1_X(\log X_0) / f^* \Omega^1_{\D}(\log [0]), \quad \Omega^p_{X/ \D}(\log) = \wedge^p\Omega^1_{X/ \D}(\log).
	\end{align*}
	These sheaves are equipped with the relative exterior differential $d$ and $(\Omega^{\bullet}_{X/ \D}(\log), d)$ is called the logarithmic de Rham complex of $f$.
	
	Let $\widetilde{\mathcal{V}}_{\log}=R^kf_* \Omega^{\bullet}_{X/ \D}(\log)$ be the $k$-th direct image sheaf of the complex of vector bundles $(\Omega^{\bullet}_{X/ \D}(\log), d)$. 
	By \cite[(2.18) Theorem]{MR0429885}, $\widetilde{\mathcal{V}}_{\log}=R^kf_* \Omega^{\bullet}_{X/ \D}(\log)$ is a locally free sheaf.
	By \cite[(2.2) Proposition]{MR0429885}, we have $\widetilde{\mathcal{V}}_{\log}|_{\D^*} = \mathcal{V}$.
	Hence $\widetilde{\mathcal{V}}_{\log}$ is a holomorphic vector bundle on $\D$ and is an extension of the flat vector bundle $\mathcal{V}= R^kf_*\mathbb{C} \otimes \mathcal{O}_{\D^*}$.
	By \cite[(2.19), (2.20) Propositions]{MR0429885}, there exists a logarithmic connection $\tilde{\nabla} : \widetilde{\mathcal{V}}_{\log} \to \widetilde{\mathcal{V}}_{\log} \otimes \Omega^1_{\D}(\log [0])$ such that $\tilde{\nabla}|_{\D^*} =\nabla$,
	and such that if $\alpha$ is an eigenvalue of the residue map $\operatorname{res}_0(\nabla)$, then $\alpha \in \mathbb{Q} \cap \left[ 0, 1\right)$.
	By the uniqueness of the Deligne extension, $\widetilde{\mathcal{V}}_{\log}$ is the lower Deligne extension of $\mathcal{V}$.

	In what follows, we identify $\widetilde{\mathcal{V}}_{\log}=R^kf_* \Omega^{\bullet}_{X/ \D}(\log)$ with $\mathcal{V}_{\log}$.
	Let $j : \D^* \hookrightarrow \D$ be the inclusion and let $\mathcal{F}^p\mathcal{V}$ be the Hodge filtration bundle of $\mathcal{V}$.
	We define the Hodge filtration bundle $\mathcal{F}^p_{\log}$ of $\mathcal{V}_{\log}$ by
	$$
		\mathcal{F}^p_{\log} = j_*(\mathcal{F}^p\mathcal{V}) \cap \mathcal{V}_{\log} \subset j_*\mathcal{V}.
	$$
	Namely, $\mathcal{F}^p_{\log}$ consists of sections of $\mathcal{F}^p\mathcal{V}$ which extends to sections of $\mathcal{V}_{\log}$.
	
	By \cite[(2.11) Theorem]{MR485870}, $R^{q}f_*  \Omega^p_{X/ \D}(\log)$ is a locally free sheaf.
	Since the spectral sequence 
		\begin{align*}
			E_1^{p,q} = R^{q}f_*  \Omega^p_{X/ \D}(\log) \Rightarrow R^kf_* \Omega^{\bullet}_{X/ \D}(\log)
		\end{align*}
	attached to the b\^ete filtration $\Omega^{\bullet \geqq p}_{X/ \D}(\log)$ degenerates at $E_1$ (cf. \cite[p. 130]{MR756849}), we have
	\begin{align}\label{al7-1-3-1}
		\mathcal{F}^p_{\log} = \operatorname{Im} (R^kf_* \Omega^{\bullet \geqq p}_{X/ \D}(\log) \to R^kf_* \Omega^{\bullet}_{X/ \D}(\log))
	\end{align}
	and
	$$
		\mathcal{F}^p_{\log} / \mathcal{F}^{p+1}_{\log} = R^{k-p}f_*  \Omega^p_{X/ \D}(\log).
	$$
	We have the following short exact sequence of coherent sheaves on $\D$
	$$
		0 \to \mathcal{F}^{p+1}_{\log} \to \mathcal{F}^p_{\log} \to R^{k-p}f_*  \Omega^p_{X/ \D}(\log) \to 0.
	$$
	By the local freeness of $R^{q}f_*  \Omega^p_{X/ \D}(\log)$, $\mathcal{F}^p_{\log}$ is a locally free subsheaf of $\mathcal{V}_{\log}$.

	Recall that $X^{\circ} =f^{-1}(\D^*)$ is equipped with the $\G$-action induced from the involution $\iota$
	and that $\D^*$ is equipped with trivial $\G$-action.
	Since $f : X^{\circ}  \to \D^*$ is $\G$-equivariant, the $\G$-action preserves the fibers of $f$.
	Therefore the $\G$-action induces the $\G$-equivariant structure on the flat vector bundle $\mathcal{V} = R^kf_*\mathbb{C} \otimes \mathcal{O}_{\D^*}$.
	
	By the definition of the Gauss-Manin connection $\nabla$, it is a $\G$-equivariant map and the $\G$-action on $\Gamma( \mathbb{H}, q^*\mathcal{V})$ preserves $V_{\infty}$. 
	Since the identification $(q^* \mathcal{V})_{\tau} \cong \mathcal{V}_{q(\tau)} \cong \mathcal{V}_{q(\tau +1)} \cong (q^* \mathcal{V})_{\tau + 1}$ is $\G$-equivariant, the monodromy operator $T$ is also $\G$-equivariant.
	By its construction, $\phi$ is $\G$-equivariant and $\mathcal{V}_{\log} = \phi(V_{\infty}) \otimes \mathcal{O}_{\D}$ is endowed with the $\G$-action induced from the $\G$-action on $V_{\infty}$.

	\begin{lem}\label{l6-1-3-3}
		Let $E, F$ be $\G$-equivariant holomorphic vector bundles on $\D$
		and let $\varphi : E \to F$ be an isomorphism of holomorphic vector bundles.
		If $\varphi|_{E|_{\D^*}} : E|_{\D^*} \to F|_{\D^*}$ is $\G$-equivariant,
		then $\varphi : E \to F$ is an isomorphism of $\G$-equivariant holomorphic vector bundles on $\D$. 
	\end{lem}
	
	\begin{proof}
		Let $r$ be the rank of $E$.
		The involutions of $E$ and $F$ induced from the $\G$-action are denoted by $i : E \to E$ and $j : F \to F$, respectively.
		Since $E$ and $F$ are holomorphic vector bundles on $\D$, $E$ and $F$ are trivial vector bundles 
		and there are global holomorphic frames $e_1, \dots, e_r$ and $f_1, \dots, f_r$ of $E$ and $F$, respectively.
		We define the $r \times r$-matrix $A = (a_{pq})$ by
		$$
			\varphi(i(e_p)) - j( \varphi(e_p)) = \sum_q a_{pq} f_q.
		$$
		Then $a_{pq}$ is a holomorphic function on $\D$ for each $p,q$.
		Since $\varphi$ is $\G$-equivariant over $\D^*$,
		we have $a_{pq}|_{\D^*} =0$.
		Since $a_{pq}$ is continuous, we have $a_{pq}=0$ on $\D$.
		Therefore for each $p,q$,
		$$
			\varphi(i(e_p)) = j( \varphi(e_p)).
		$$
		Thus $\varphi$ is a $\G$-equivariant isomorphism.
		This completes the proof.
	\end{proof}
	
	\begin{lem}\label{l6-1-3-4}
		  The $\G$-action on $\mathcal{V}_{\log}$ is the extension of the $\G$-action on $\mathcal{V}$ and preserves the Hodge filtration $\mathcal{F}^p_{\log}$.
	\end{lem}
	
	\begin{proof}
		Since $\phi$ is a $\G$-equivariant homomorphism, (\ref{al6-1-3-5}) is a $\G$-equivariant isomorphism.
		Therefore, $\mathcal{V}_{\log}|_{\D^*} = \phi(V_{\infty}) \otimes \mathcal{O}_{\D^*} \cong \mathcal{V}$ is a $\G$-equivariant isomorphism
		and the $\G$-action on $\mathcal{V}_{\log}$ is the extension of the $\G$-action on $\mathcal{V}$.
		Consider the $\G$-action of $R^kf_* \Omega^{\bullet}_{X/ \D}(\log)$ induced from the $\G$-action of $\Omega^{\bullet}_{X/ \D}(\log)$.
		Since the $\G$-action on $R^kf_* \Omega^{\bullet}_{X/ \D}(\log)$ is the extension of the $\G$-action on $\mathcal{V}=R^kf_*\mathbb{C} \otimes \mathcal{O}_{\D^*}$ via the identification $R^kf_* \Omega^{\bullet}_{X^{\circ}/ \D^*}(\log) \cong R^kf_*\mathbb{C} \otimes \mathcal{O}_{\D^*}$,
		it follows from Lemma \ref{l6-1-3-3} that $\mathcal{V}_{\log} \cong R^kf_* \Omega^{\bullet}_{X/ \D}(\log)$ is a $\G$-equivariant isomorphism.
		Since the $\G$-action on $R^kf_* \Omega^{\bullet}_{X/ \D}(\log)$ preserves subbundle $\operatorname{Im} (R^kf_* \Omega^{\bullet \geqq p}_{X/ \D}(\log) \to R^kf_* \Omega^{\bullet}_{X/ \D}(\log))$,
		the corresponding subbundle $\mathcal{F}^p_{\log}$ of $\mathcal{V}_{\log}$ is a $\G$-equivariant subbundle via the isomorphism (\ref{al7-1-3-1}).
		This completes the proof.
	\end{proof}

	By Lemma \ref{l6-1-3-4}, $\mathcal{V}_{\log}$, $\mathcal{F}^p_{\log}$, and $R^{k-p}f_*  \Omega^p_{X/ \D}(\log)$ are $\G$-equivariant vector bundles on $\D$.
	Their $(\pm 1)$-eigenbundles are denoted by $\mathcal{V}_{\log, \pm}$, $\mathcal{F}^p_{\log, \pm}$, and $R^{k-p}f_*  \Omega^p_{X/ \D}(\log)_{\pm}$, respectively.
	These are holomorphic vector bundles on $\D$.

	Consider a semistable reduction
	$$
		\xymatrix{
	                &Y  \ar[d]_{g} \ar[r]^{r} & X  \ar[d]^{f}     \\   	
	                &\D' \ar[r]_{\rho}   & \D ,   
		}
	$$
	where $\D'$ is the unit disk, $t$ is a coordinate on $\D'$, $g : Y \to \D$ is a semistable degeneration and $\rho(t)=t^l$.
	Let $s$ be the coordinate on $\D$ with $s =t^l$.
	Since $g^{-1}(\D'^*)$ is endowed with the $\G$-action induced from the $\G$-action on $X$,
	$R^kg_*\mathbb{C} \otimes \mathcal{O}_{\D'^*}$ is a $\G$-equivariant flat vector bundle on $\D'^*$.
	In the same manner as above, its lower Deligne extension $\mathcal{U} = R^kg_*\Omega^{\bullet}_{Y/\D'}(\log)$ is a $\G$-equivariant vector bundle on $\D'$
	and it has a $\G$-equivariant Hodge filtration bundle $^{\rho}{\mathcal{F}}^p$ such that
	$$
		^{\rho}{\mathcal{F}}^p / ^{\rho}{\mathcal{F}}^{p+1} \cong R^{k-p}g_* \Omega^{p}_{Y/ \D'}(\log)
	$$
	is a $\G$-equivariant isomorphism.
	
	Let $h_X$ be the $\G$-invariant \K metric on $X$ such that $h_X = h_{\X}|_X$ and let $\omega_X$ be the associated \K form. 
	Set $Y^{\circ} = g^{-1}(\D^*)$.
	Since the restriction $r|_{Y^{\circ}} : Y^{\circ} \to X^{\circ}$ is a $\G$-equivariant finite \'etale cover,
	\begin{align}\label{al6-1-3-8}
		\omega_{Y^{\circ}} :=( r^* \omega_X )|_{Y^{\circ}}
	\end{align} 
	is a $\G$-invariant \K form on $Y^{\circ}$.
	The associated \K metric on $Y^{\circ}$ is denoted by $h_{Y^{\circ}}$.
	
	We denote by $H^q(\Omega^p_{X_{\infty}})_{\pm}$ the $(\pm 1)$-eigenspace of the $\G$-action on $H^q(\Omega^p_{X_{\infty}})$.
	We set $h^{p,q}_{\pm} = \dim H^q(\Omega^p_{X_{\infty}})_{\pm}$ and $h^p_{\pm} = \sum_{i \geqq p} h^{i,k-i}_{\pm} $.
	The pullback of differential forms from $X$ to $Y$ induces $\G$-equivariant injections of locally free sheaves of equal rank
	$$
		\rho^*\mathcal{V}_{\log} \hookrightarrow \mathcal{U} , \quad \rho^*\mathcal{F}^p_{\log} \hookrightarrow {}^{\rho}{\mathcal{F}}^p, \quad \rho^*R^{q}f_*  \Omega^p_{X/ \D}(\log) \hookrightarrow R^{q}g_* \Omega^{p}_{Y/ \D'}(\log).
	$$
	The quotients are torsion sheaves and we can write
	$$
		\frac{^{\rho}{\mathcal{F}}^p_{\pm}}{\rho^*\mathcal{F}^p_{\log, \pm}} \cong \bigoplus^{h^p_{\pm}}_{j=1} \frac{\mathcal{O}_{\D',0}}{t^{a^p_{j, \pm}} \mathcal{O}_{\D',0}}, 
		\qquad \frac{R^{q}g_* \Omega^{p}_{Y/ \D'}(\log)_{\pm}}{\rho^*R^{q}f_*  \Omega^p_{X/ \D}(\log)_{\pm}} \cong  \bigoplus^{h^{p,q}_{\pm}}_{j=1} \frac{\mathcal{O}_{\D',0}}{t^{b^{p,q}_{j, \pm}}\mathcal{O}_{\D',0}} 
	$$
	for some integers $a^p_{j, \pm}$ and $b^{p,q}_{j, \pm}$.
	By \cite[Lemma 2.4]{MR4255041}, $0 \leqq a^p_{j, \pm} \leqq l-1$ and $0 \leqq b^{p,q}_{j, \pm} \leqq l-1$. 
	Moreover $a^p_{j, \pm}/l$, $b^{p,q}_{j, \pm}/l$ are independent of the choice of semistable reduction.
	We set 
	\begin{align}\label{al6-1-3-9}
		\alpha^{p,q}_{j, \pm} = b^{p,q}_{j, \pm}/l \quad \text{ and } \quad \alpha^{p,q}_{\pm} = \sum_j \alpha^{p,q}_{j, \pm}.
	\end{align}
	
	The space of flat sections of $\mathcal{U}|_{\D'^*}$ is denoted by $U_{\infty}$.
	There is a $\G$-equivariant isomorphism between $V_{\infty}$ and $U_{\infty}$ defined by
	$$
		\nu : V_{\infty} \to U_{\infty} \qquad \nu(e)(\tau) = e(l \tau).
	$$
	In the same way as (\ref{al-1-5-1}), we have the $\G$-equivariant isomorphism
	$$
		^{\rho}\psi : U_{\infty} \overset{\sim}{\to} \mathcal{U}(0),
	$$
	and the Hodge filtration $^{\rho}\mathcal{F}^p$ induces a filtration $F^p_{\infty, \pm}$ on $V_{\infty, \pm}$ through the isomorphism $^{\rho}\psi \circ \nu$,
	which is called Steenbrink's limiting Hodge filtration.
	
	By \cite[Corollary 2.8]{MR4255041}, $\exp( -2\pi i \alpha^{p,q}_{j, \pm} )$ forms the eigenvalues of the $T_s$-action on $\operatorname{Gr}^p_{ F_{\infty, \pm} }H^k(X_{ \infty })_{ \pm }$ and 
	\begin{align}\label{al2-1-5-1}
		\alpha^{p,q}_{\pm} = -\frac{1}{2 \pi i} \operatorname{Tr}(^l\log T_s|_{\operatorname{Gr}^p_{F_{\infty, \pm}}H^{p+q}(X_{\infty})_{\pm}}).
	\end{align}
	
	\begin{lem}\label{l6-1-3-6}
		For every $p, q \geqq 0$, there exists a nowhere vanishing holomorphic section $\theta^{p,q}_{\pm}$ of $\det R^{q}f_*  \Omega^p_{X/ \D^*}(\log)_{\pm}$ 
		such that $t^{-\sum_j b^{p,q}_{j, \pm}} \rho^* \theta^{p,q}_{\pm}$ is a nowhere vanishing holomorphic section of $\det R^{q}g_*  \Omega^p_{Y/ \D'^*}(\log)_{\pm}$.
	\end{lem}
	
	\begin{proof}
	We set
	\begin{align*}
		k= h^{p+1}_+ =\operatorname{rank} \mathcal{F}^{p+1}_{\log, +}, &\qquad \ell=h^{p}_+ = \operatorname{rank} \mathcal{F}^p_{\log, +}, \\
		m=h^{p+1}_- =\operatorname{rank} \mathcal{F}^{p+1}_{\log,-}, &\qquad n=h^{p}_- =\operatorname{rank} \mathcal{F}^p_{\log, -}.
	\end{align*}
	Choose a local frame of $\mathcal{F}^p_{\log, +}$ and $\mathcal{F}^p_{\log, -}$ around $0$ and denote them by
	$$
		\sigma_{1,+}, \dots, \sigma_{\ell, +} \quad \text{and} \quad \sigma_{1, -}, \dots, \sigma_{n, -}
	$$
	such that $\sigma_{1,+}, \dots, \sigma_{k, +}$ forms a local frame of $\mathcal{F}^{p+1}_{\log, +}$,
	such that $\sigma_{1, -}, \dots, \sigma_{m, -}$ forms a local frame of $\mathcal{F}^{p+1}_{\log, -}$,
	such that $\sigma_{k+1,+}, \dots, \sigma_{\ell, +}$ projects onto a local frame of $\mathcal{F}^{p}_{\log, +} / \mathcal{F}^{p+1}_{\log, +} \cong R^{q}f_*  \Omega^p_{X/ \D}(\log)_{+}$,
	and such that $\sigma_{m+1, -}, \dots, \sigma_{n, -}$ projects onto a local frame of $\mathcal{F}^{p}_{\log, -} / \mathcal{F}^{p+1}_{\log, -} \cong R^{q}f_*  \Omega^p_{X/ \D}(\log)_{-}$.
	Applying \cite[Theorem 2.6 and the proof of Corollary 2.8]{MR4255041} to the local frame $\sigma_{1,+}, \dots, \sigma_{\ell, +}, \sigma_{1, -}, \dots, \sigma_{n, -}$ of $\mathcal{F}^p_{\log}$, 
	there is another frame 
	$$
		\theta_{1}, \dots, \theta_{\ell}, \theta_{\ell+1}, \dots, \theta_{\ell+n}
	$$
	with the following properties:
	\begin{enumerate}[ label= \rm{(\arabic*)} ]
		\item	$\bigoplus^k_{i=1} \mathcal{O}_{\D} \theta_i = \bigoplus^k_{i=1} \mathcal{O}_{\D} \sigma_{i, +} = \mathcal{F}^{p+1}_{\log, +} \quad \text{and} \quad \bigoplus^{\ell}_{i=1} \mathcal{O}_{\D} \theta_i = \bigoplus^{\ell}_{i=1} \mathcal{O}_{\D} \sigma_{i, +} = \mathcal{F}^{p}_{\log, +}$,
		\item there are integers $0 \leqq a_i \leqq l-1$ such that $^{\rho}\theta_i = t^{-a_i} \rho^* \theta_i$ forms a local frame of $^{\rho}\mathcal{F}^p =  {^{\rho}\mathcal{F}^p_+} \oplus  {^{\rho}\mathcal{F}^p_-}$,
		\item $\sum^{\ell}_{i=k+1} a_i = \sum_j b^{p,q}_{j, +}$.
	\end{enumerate}
	By $(1)$ and $(2)$, $\theta_{k+1}, \dots, \theta_{\ell}$ projects onto a local frame of the bundle $R^{q}f_*  \Omega^p_{X/ \D}(\log)_{+}$ and $t^{-a_{k+1}} \rho^*\theta_{k+1}, \dots, t^{-a_{\ell}} \rho^*\theta_{\ell}$ projects onto a local frame of the bundle $R^{q}g_* \Omega^{p}_{Y/ \D'}(\log)_+$.
	Therefore we obtain a nowhere vanishing holomorphic section $\theta^{p,q}_{+}$ of $\det R^{q}f_*  \Omega^p_{X/ \D}(\log)_{+}$ 
	such that $t^{-\sum_j b^{p,q}_{j, +}} \rho^* \theta^{p,q}_{+}$ is a nowhere vanishing holomorphic section of $\det R^{q}g_*  \Omega^p_{Y/ \D'}(\log)_{+}$.
	
	Similarly, applying \cite[Theorem 2.6.]{MR4255041} to the local frame $\sigma_{1,-}, \dots, \sigma_{n, -}, \sigma_{1, +}, \dots, \sigma_{\ell, +}$ of $\mathcal{F}^p_{\log}$, 
	we obtain a nowhere vanishing holomorphic section $\theta^{p,q}_{-}$ of $\det R^{q}f_*  \Omega^p_{X/ \D}(\log)_{-}$ 
	such that $t^{-\sum_j b^{p,q}_{j, -}} \rho^* \theta^{p,q}_{-}$ is a nowhere vanishing holomorphic section of $\det R^{q}g_*  \Omega^p_{Y/ \D'}(\log)_{-}$.
	\end{proof}
	 
	Denote the $L^2$-norms on $\det R^{q}f_*  \Omega^p_{X/ \D}(\log)_{\pm}|_{\D^*}$ and $\det R^{q}g_* \Omega^{p}_{Y/ \D'}(\log)_{\pm}|_{\D'^*}$ with respect to $h_X|_{X^{\circ}}$ and $h_{Y^{\circ}}$ by $ \| \cdot \|^2_{L^2}$.
	
	\begin{lem}\label{l6-1-3-7}
		The following identity holds:
		\begin{align*}
			\rho^* \log \| \theta^{p,q}_{\pm} \|^2_{L^2} -\log \| t^{-\sum_j b^{p,q}_{j, \pm}} \rho^* \theta^{p,q}_{\pm} \|^2_{L^2} = \alpha^{p,q}_{\pm} l \log |t|^2.
		\end{align*}
	\end{lem}
	
	\begin{proof}
		By (\ref{al6-1-3-8}), $ \rho^*R^{q}f_*  \Omega^p_{X/ \D}(\log)|_{\D'^*} \cong R^{q}g_* \Omega^{p}_{Y/ \D'}(\log)|_{\D'^*}$ is an isometry with respect to the $L^2$-metrics.
		Therefore, we have
		$$
			\log \| t^{-\sum_j b^{p,q}_{j, \pm}} \rho^* \theta^{p,q}_{\pm} \|^2_{L^2}  = -\sum_j b^{p,q}_{j, \pm} \log |t|^2 +\rho^* \log \| \theta^{p,q}_{\pm} \|^2_{L^2}.
		$$
		By (\ref{al6-1-3-9}), we obtain the desired result.
	\end{proof}
	
	\begin{prop}\label{p-1-14}
		Let $\sigma = (\sigma_+, \sigma_-)$ be an admissible section of $\lambda_{\G}( \Omega^p_{\X/C}( \log ))$ on $\D$
		and let $ \| \cdot \|^2_{\lambda_{\G}( \Omega^p_{\X/C}( \log )), L^2}(\iota)$ be the $L^2$-metric on $\lambda_{\G}( \Omega^p_{\X/C}( \log ))$.
		Then the following identity holds:
		$$
			\log  \| \sigma (s) \|^2_{\lambda_{\G}( \Omega^p_{\X/C}( \log )), L^2}(\iota) = \sum_q (-1)^q(\alpha^{p,q}_{+} - \alpha^{p,q}_{-}) \log |s|^2 +O(\log \log |s|^{-1})
		$$ 
	\end{prop}
	
	\begin{proof}
		Let $\sigma' = (\sigma'_+, \sigma'_-)$ be an admissible section of $\lambda_{\G}( \Omega^p_{Y/\D'}( \log ))$ on $\D$.
		Since $g : Y \to \D'$ is a semistable degeneration, we get by \cite[(6.6)]{MR0382272}, \cite[Proposition 2.2.1.]{MR744325} 
		\begin{align*}
			\log  \| \sigma'_{\pm} (t) \|^2_{L^2} = O(\log \log |t|^{-1}) \quad (t \to 0).
		\end{align*}
		Since we can take $\sigma'_{\pm} = t^{-\sum_j b^{p,q}_{j, \pm}} \rho^* \theta^{p,q}_{\pm}$, this implies that
		\begin{align}\label{A-al-6-1-2-1}
			\log  \| t^{-\sum_j b^{p,q}_{j, \pm}} \rho^* \theta^{p,q}_{\pm} \|^2_{L^2} = O(\log \log |t|^{-1}) \quad (t \to 0).
		\end{align}
		Since we can take $\sigma_{\pm} = \otimes_{q} (\theta^{p,q}_{\pm})^{(-1)^q}$, we get by Lemma \ref{l6-1-3-7} and (\ref{A-al-6-1-2-1})
		\begin{align}\label{A-al-6-1-2-2}
		\begin{aligned}
			(\rho^* \log \| \sigma_{\pm} \|_{L^2}^2)(t) &= \sum_q (-1)^q \log \| t^{-\sum_j b^{p,q}_{j, \pm}} \rho^* \theta^{p,q}_{\pm} \|_{L^2}^2 +\sum_q (-1)^q \alpha^{p,q}_{\pm} l \log |t|^2 \\
			&= \sum_q (-1)^q \alpha^{p,q}_{\pm} \log |t^l|^2 +O(\log \log |t|^{-1}).
		\end{aligned}
		\end{align}
		Since $s=t^l=\rho(t)$, we get the result by (\ref{A-al-6-1-2-2}).
		
	\end{proof}

\subsection{Comparison between the K\"ahler extension and the logarithmic extension}

	In this subsection, following Eriksson-Freixas i Montplet-Mourougane \cite{MR4255041}, we compare the K\"ahler extension with the logarithmic extension.
	
	We write $X_0 = \sum_i m_i D_i$, where $D_i$ is a reduced irreducible component of $X_0$.
	Since $X_0$ is simply normal crossing, $D_i$ is smooth.
	We set $D(1) = \sqcup_i D_i$.
	Let $a : D(1) \to X$ be the natural morphism.
	
	For a $\G$-equivariant vector bundle $E$ over a complex manifold $M$, we set
	$$
		\chi_{\pm}(E) = \sum_{q \geqq 0} (-1)^q \dim H^q(M, E)_{\pm}.
	$$
	We define integers $\mu_{\pm}$ by
	$$
		\mu_{\pm} = \chi_{\pm} (\mathcal{O}_{X_{\infty}}) - \sum_i \chi_{\pm} (\mathcal{O}_{D_i}).
	$$
	Throughout this subsection, tensor products of line bundles are denoted additively for lighter notation.
	
	\begin{prop}\label{p-1-15}
		The following canonical isomorphisms hold:
		$$
			\lambda_{\pm}(\widetilde{\Omega}^1_{X/\D}) = \lambda_{\pm}( \Omega^1_{\X/C}( \log ) ) + \mu_{\pm} \mathcal{O}_{\D}([0]).
		$$
	\end{prop}
	
	\begin{proof}
		We follow \cite[Proposition 3.7]{MR4255041}.
		By the projection formula, we have
		$$
			\det \left( R^q f_* f^* \Omega^1_{\D} \right)_{\pm} = \operatorname{rank} \left( R^q f_*  \mathcal{O}_X \right)_{\pm} \cdot \Omega^1_{\D} + \det \left( R^q f_*  \mathcal{O}_X \right)_{\pm}.
		$$
		Therefore
		\begin{align*}
			\lambda_{\pm}(\widetilde{\Omega}^1_{X/\D}) = -\chi_{\pm} (\mathcal{O}_{X_{\infty}}) \Omega^1_{\D} + \lambda_{\pm}( \Omega^1_X ) -\lambda_{\pm}( \mathcal{O}_X ).
		\end{align*}
		In the same manner, we can see that
		\begin{align*}
			\lambda_{\pm}( \Omega^1_{X/\D}( \log ) ) = -\chi_{\pm} (\mathcal{O}_{X_{\infty}}) \Omega^1_{\D}( \log [0]) + \lambda_{\pm}( \Omega^1_X( \log X_0 ) ) -\lambda_{\pm}( \mathcal{O}_X ).
		\end{align*}
		Therefore we have
		$$
			\lambda_{\pm}( \Omega^1_{X/\D}( \log ) ) - 	\lambda_{\pm}(\widetilde{\Omega}^1_{X/\D})
			=  -\chi_{\pm} (\mathcal{O}_{X_{\infty}}) \mathcal{O}_{\D}([0]) + \lambda_{\pm}( \Omega^1_X( \log X_0 ) ) - \lambda_{\pm}( \Omega^1_X ).
		$$
		By the short exact sequence
		$$
			0 \to \Omega^1_X \to \Omega^1_X( \log X_0 ) \to a_* \mathcal{O}_{D(1)} \to 0,
		$$
		we have
		\begin{align*}
			 \lambda_{\pm}( \Omega^1_X( \log X_0 ) ) - \lambda_{\pm}( \Omega^1_X ) = \lambda_{\pm}( a_* \mathcal{O}_{D(1)} ) 
			 = \chi_{\pm} (\mathcal{O}_{D(1)}) \mathcal{O}_{\D}([0]).
		\end{align*}
		Thus we obtain the desired results.
	\end{proof}
	
	We set
	\begin{align}\label{al5-1-4-1}
		c(X_0, \Omega^1_{\X/C}) = \gamma_{\iota}(X_0, \widetilde{\Omega}^1_{\X/C}) +(\mu_+ -\mu_-) -\sum_q (-1)^q(\alpha^{1,q}_{+} - \alpha^{1,q}_{-}).
	\end{align}
	
	\begin{thm}\label{ta-1-2}
		The following identity holds:
		$$
			\log \tau_{\iota} ( \bar{\Omega}^1_{\xs} ) = c(X_0, \Omega^1_{\X/C}) \log |s|^2 +O( \log \log |s|^{-1}) \quad (s \to 0).
		$$
	\end{thm}

	\begin{proof}
		Let $ \| \cdot \|^2_{ \lambda_{\G}(\widetilde{\Omega}^1_{\X/C}), Q,  h_{\X/C} }(\iota) $ be the equivariant Quillen metric on the K\"ahler extension $\lambda_{\G}(\widetilde{\Omega}^1_{\X/C})$ defined in Definition \ref{d-1-7}
		and let $\sigma =( \sigma_+, \sigma_-)$ is an admissible section of $\lambda_{\G}(\widetilde{\Omega}^1_{\X/C})$.
		By Proposition \ref{p-1-13}, we have
		\begin{align}\label{al5-1-4-2}
			\log \| \sigma \|^2_{ \lambda_{\G}(\widetilde{\Omega}^1_{\X/C}), Q,  h_{\X/C} }(\iota) \equiv \gamma_{\iota}(X_0, \widetilde{\Omega}^1_{\X/C}) \log |s|^2.
		\end{align}
		Let 
		$$
			\| \cdot \|^2_{ \lambda_{\G}(\Omega^1_{\X/C}(\log)), Q,  h_{\X/C} }(\iota) := \tau_{\iota}( \overline{\Omega}^1_{\X/C} ) \| \cdot \|^2_{ \lambda_{\G}(\Omega^1_{\X/C}(\log)), L^2,  h_{\X/C} }(\iota) 
		$$
		be the equivariant Quillen metric on the logarithmic extension $\lambda_{\G}(\Omega^1_{\X/C}(\log))$ with respect to the \K metric $h_{\X/C}$
		and let $\sigma' =( \sigma'_+, \sigma'_-)$ be an admissible section of $\lambda_{\G}(\Omega^1_{\X/C}(\log))$.
		By Proposition \ref{p-1-15}, there exist holomorphic functions $f_{\pm}$ on $\D$ such that
		\begin{align}\label{al5-1-4-3}
			f_{\pm}(0) \neq 0 \quad \text{and} \quad \sigma_{\pm} = s^{-\mu_{\pm}} f_{\pm} \sigma'_{\pm}
		\end{align}
		By (\ref{al5-1-4-2}) and (\ref{al5-1-4-3}), we have
		\begin{align}\label{al5-1-4-4}
		\begin{aligned}
			&\quad \log \| \sigma' \|^2_{ \lambda_{\G}(\Omega^1_{\X/C}(\log)), Q,  h_{\X/C} }(\iota) \\
			&= \log \| \sigma \|^2_{ \lambda_{\G}(\widetilde{\Omega}^1_{\X/C}), Q,  h_{\X/C} }(\iota) +(\mu_+ -\mu_-) \log |s|^2 +O(1)\\
			&= \left( \gamma_{\iota}(X_0, \widetilde{\Omega}^1_{\X/C}) +(\mu_+ -\mu_-) \right) \log |s|^2 +O(1).
		\end{aligned}
		\end {align}
		On the other hand, by Proposition \ref{p-1-14}, we have
		\begin{align}\label{al5-1-4-5}
			\log  \| \sigma' \|^2_{\lambda_{\G}( \Omega^1_{\X/C}( \log )), L^2}(\iota) = \sum_{q \geqq 0} (-1)^q(\alpha^{1,q}_{+} - \alpha^{1,q}_{-}) \log |s|^2 +O(\log \log |s|^{-1})
		\end{align}
		By (\ref{al5-1-4-4}) and (\ref{al5-1-4-5}), we have
		\begin{align}\label{al5-1-4-6}
		\begin{aligned}
			&\quad \log \tau_{\iota}( \overline{\Omega}^1_{\X/C} ) \\
			&= \log \| \sigma' \|^2_{ \lambda_{\G}(\Omega^1_{\X/C}(\log)), Q,  h_{\X/C} }(\iota) -\log \| \sigma' \|^2_{ \lambda_{\G}(\Omega^1_{\X/C}(\log)), L^2,  h_{\X/C} }(\iota)  \\
			&= \left( \gamma_{\iota}(X_0, \widetilde{\Omega}^1_{\X/C}) +(\mu_+ -\mu_-) - \sum_{q \geqq 0} (-1)^q(\alpha^{1,q}_{+} - \alpha^{1,q}_{-}) \right) \log |s|^2 +O(\log \log |s|^{-1}).
		\end{aligned}
		\end {align}
		By (\ref{al5-1-4-1}), the proof is completed.
	\end{proof}

\section{The algebraicity of the singularity of $\tau_{M, \mathcal{K}}$}

	In this section, we recall the invariant $\tau_{M, \mathcal{K}}$ for manifolds of $K3^{[2]}$-type with antisymplectic involution
	and prove that $\tau_{M, \mathcal{K}}$ has an algebraic singularity.
	The main idea of the proof is due to \cite{MR2454893} and \cite{MR2047658}.

\subsection{The invariant $\tau_{M, \mathcal{K}}$ for manifolds of $K3^{[2]}$-type with antisymplectic involution}\label{ss-2-1}
	
\subsubsection{Manifolds of $K3^{[2]}$-type and antisymplectic involutions}\label{sss-2-1-1}
	
	A simply-connected compact \K manifold $X$ is an irreducible holomorphic symplectic manifold if there exists an everywhere non-degenerate holomorphic 2-form $\eta$ such that $H^0(X, \Omega_X^2)$ is generated by $\eta$.
	By \cite[Proposition 23.14 and Remark 23.15]{MR1963559}, there exist a unique primitive integral quadric form $q_X$ on $H^2(X, \mathbb{Z})$ of signature $(3, b_2(X)-3)$ 
	and a positive rational number $c_X \in \mathbb{Q}_{\geqq 0}$ such that $q_X(\alpha)^n = c_X \int_X \alpha^{2n}$ for any $\alpha \in H^2(X, \mathbb{Z})$. 
	If $b_2(X)=6$, we also require that $q_X(\omega)>0$ for any \K class $\omega$. 
	The quadric form $q_X$ is called the Beauville-Bogomolov-Fujiki form.
	Let $( \cdot , \cdot ) : H^2(X, \mathbb{Z}) \times H^2(X, \mathbb{Z}) \to \mathbb{Z}$ be the integral bilinear form corresponding to $q_X$.
	An irreducible holomorphic symplectic manifold $X$ is a manifold of $K3^{[2]}$-type if $X$ is deformation equivariant to the Hilbert scheme of length $2$ zero-dimensional subschemes of a K3 surface.
	
	We denote by $U$ the even unimodular lattice of rank $2$ of signature $(1,1)$ with Gram matrix  
	$ \begin{pmatrix}
	0 & 1 \\
	1 & 0 \\
	\end{pmatrix}$,
	and denote by $E_8$ the {\it negative} definite even unimodular lattice of rank $8$ associated with the Dynkin diagram $E_8$.
	We set 
	$$
		L_{K3} = E_8^{\oplus 2} \oplus U^{\oplus 3},
	$$
	and
	$$
		L_2 = L_{K3} \oplus \mathbb{Z} e,
	$$
	where $e^2 =-2$ and $(e, L_{K3})=0$.
	If $X$ is a manifold of $K3^{[2]}$-type,
	then the lattice $H^2( X, \mathbb{Z} )$ is isomorphic to $L_2$.
	
	A holomorphic involution $\iota : X \to X$ satisfying $\iota^* \eta =-\eta$ is called antisymplectic.
	In this paper, \textit{involutions on an irreducible holomorphic symplectic manifold are always assumed to be antisymplectic}.
	
	Let $X_1, X_2$ be irreducible holomorphic symplectic manifolds.
	Recall that a parallel-transport operator $f : H^2(X_1, \mathbb{Z}) \to H^2(X_2, \mathbb{Z})$ is an isomorphism such that
	there exist a family $p : \mathcal{X} \to B$ of irreducible holomorphic symplectic manifolds over a possibly reducible analytic base $B$, 
	two points $b_1, b_2 \in B$, and a continuous path $\gamma : [0, 1] \to B$ with $\gamma(0)=b_1, \gamma(1)=b_2$ such that $p^{-1}(b_i) \cong X_i$ $(i=1,2)$ and that the parallel-transport in the local system $R^2p_*\mathbb{Z}$ induces $f : H^2(X_1, \mathbb{Z}) \to H^2(X_2, \mathbb{Z})$.

	Let $X$ be an irreducible holomorphic symplectic manifold and let $\Lambda$ be a lattice isomorphic to $H^2( X, \mathbb{Z} )$.
	A parallel-transport operator $g : H^2(X, \mathbb{Z}) \to H^2(X, \mathbb{Z})$ is called a monodromy operator.
	The subgroup $\mon(X)$ of the isometry group $O(H^2(X, \mathbb{Z}))$ consisting of all monodromy operators of $X$ is called the monodromy group.
	For an isometry $\alpha : H^2( X, \mathbb{Z} ) \to \Lambda$, we set $\mon( \Lambda ) = \alpha \circ \mon(X) \circ \alpha^{-1}$.
	If $X$ is a K3 surface or a manifold of $K3^{[2]}$-type,
	it follows from \cite[Theorem 9.1]{MR2964480} that the group $\mon(\Lambda)$ is a normal subgroup of $O(\Lambda)$ and is independent of the choice of $(X,\alpha)$.
	
	Let $O^+(\Lambda)$ be the subgroup of $O(\Lambda)$ consisting of the isometries of real spinor norm $+1$.
	If $X$ is a K3 surface or a manifold of $K3^{[2]}$-type,
	it follows from \cite[Lemma 9.2]{MR2964480} and \cite[Theorem A]{MR0849050} that $O^+(\Lambda) = \mon(\Lambda)$.
	We set $\tilde{\mathscr{C}}_{\Lambda}=\{ x \in \Lambda_{\mathbb{R}} ; x^2>0 \}$.
	Since $\operatorname{sign}(\Lambda) = (3, \operatorname{rank} \Lambda -3)$, we have $H^2( \tilde{\mathscr{C}}_{\Lambda}, \mathbb{Z} ) \cong \mathbb{Z}$ by \cite[Lemma 4.1]{MR2964480}.
	A generator of $H^2( \tilde{\mathscr{C}}_{\Lambda}, \mathbb{Z} )$ is called an orientation class.
	By \cite[\S 4]{MR2964480}, any isometry of real spinor norm $+1$ preserves an orientation class. 
	
	A pair $(X, \alpha)$ is called a marked manifold of $K3^{[2]}$-type if $X$ is a manifold of $K3^{[2]}$-type and $\alpha : H^2(X, \mathbb{Z}) \to L_2$ is an isometry.
	Let $\mathfrak{M}_{L_2}$ be the moduli space of marked manifolds of $K3^{[2]}$-type constructed in \cite[Definition 25.4]{MR1963559}.
	We fix a connected component $\mathfrak{M}_{L_2}^{\circ}$ of $\mathfrak{M}_{L_2}$.
	
	For a hyperbolic sublattice $M$ of $L_2$ and an involution $\iota_M \in \mon(L_2)$,
	the pair $(M, \iota_M)$ is an admissible sublattice of $L_2$ if the invariant sublattice $(L_2)^{\iota_M}$ of $\iota_M$ is equal to $M$.
	By \cite[Figures 1 and 2]{MR3542771}, there are $150$ types of admissible sublattice of $L_2$.
	Let $(X, \iota)$ be a manifold of $K3^{[2]}$-type with involution.
	Then $\iota$ is of type $M$ if there exists an isometry $\alpha : H^2(X, \mathbb{Z}) \to L_2$ such that $(X,\alpha) \in \mathfrak{M}_{L_2}^{\circ}$ and $\alpha \circ \iota^* = \iota_M \circ \alpha$.  
	We set
	$$
		\Delta(M) = \left\{ \delta \in M ; \delta^2 =-2, \text{or } \delta^2=-10,  (\delta, L_2)=2\mathbb{Z} \right\} \quad \text{and} \quad \tilde{\mathscr{C}}_M = \{ x \in M_{\mathbb{R}} ; x^2>0 \}.
	$$ 
	For $\delta \in \Delta(M)$, we define $\delta^{\perp}=\{ x \in \tilde{\mathscr{C}}_M : (x, \delta)=0 \}$.
	Then $\bigcup_{\delta \in \Delta(M)} \delta^{\perp}$ is a closed subset of $\tilde{\mathscr{C}}_M$
	and a connected component of $\tilde{\mathscr{C}}_M \setminus \bigcup_{\delta \in \Delta(M)} \delta^{\perp}$ is called a K\"ahler-type chamber of $M$.
	The set of all K\"ahler-type chambers of $M$ is denoted by $\ktm$. 
	We set
	$$
		\Gamma(M) = \{ \sigma \in \mon(L_2) ; \iota_M \circ \sigma = \sigma \circ \iota_M \} \text{ and } \Gamma_M = \{ \sigma|_{M} \in O(M) ; \sigma \in \Gamma(M) \}.
	$$
	Then $\Gamma_M$ acts on $\ktm$.
	
	Let $(M, \iota_M)$ be an admissible sublattice of $L_2$ and let $\mathcal{K} \in \ktm$.
	Let $(X, \iota)$ be a manifold of $K3^{[2]}$-type with involution. 
	A marking $\alpha : H^2(X, \mathbb{Z}) \to L_2$ of $(X, \iota)$ is said to be admissible for $(M, \mathcal{K})$ if $(X,\alpha) \in \mathfrak{M}_{L_2}^{\circ}$, $\alpha \circ \iota^* = \iota_M \circ \alpha$, and $\alpha(\mathcal{K}_X^{\iota}) \subset \mathcal{K}$,
	where $\mathcal{K}_X^{\iota} \subset H^2(X, \mathbb{R})^{\iota}$ is the set of $\iota$-invariant \K classes of $X$. 
	Moreover $\iota$ is called an involution of type $(M, \mathcal{K})$ if there exists an admissible marking for $(M, \mathcal{K})$.
	We denote by $\tilde{\mathcal{M}}_{M, \mathcal{K}}$ the set of all isomorphism classes of $K3^{[2]}$-type manifolds $(X, \iota)$ with involution of type $(M, \mathcal{K})$.
	By \cite[Theorem 9.11]{MR3519981}, any two isomorphism classes of $\tilde{\mathcal{M}}_{M, \mathcal{K}}$ are deformation equivalent.
	
	Fix $h \in \mathcal{K} \cap M$.
	Then $h^{\perp} = h^{\perp} \cap L_2$ is a lattice of signature $(2, 20)$ containing $M^{\perp}$, and $\Omega_{h^{\perp}}$ consists of two connected components.
	By \cite[(4.1)]{MR2964480}, the connected component $\mathfrak{M}_{L_2}^{\circ}$ determines the connected component $\Omega_{h^{\perp}}^+$ of 
	$$
		\Omega_{h^{\perp}} = \{ [\eta] \in \mathbb{P}( (h^{\perp})_{\mathbb{C}}) ; (\eta, \eta)=0, (\eta, \bar{\eta}) >0 \}.
	$$
	Namely, for any $(X, \alpha) \in \mathfrak{M}_{L_2}^{\circ}$ and $p = [\sigma] \in \Omega_{h^{\perp}}^+$, the orientation of $\tilde{\mathscr{C}}_{X} = \{ x \in H^2(X, \mathbb{R}) ; q_X(x)>0 \}$ determined by the real and the imaginary part of  the holomorphic 2-form on $X$ and a \K class of $X$ is compatible with 
	the orientation of $\tilde{\mathscr{C}_{L_2}} = \{ x \in L_{2, \mathbb{R}} ; x^2>0 \}$ determined by the real 3-dimensional vector space $W := \operatorname{Span}_{\mathbb{R}} \{ \operatorname{Re} \sigma, \operatorname{Im} \sigma, h \}$ associated to the basis $\{ \operatorname{Re} \sigma, \operatorname{Im} \sigma, h \}$ via the isomorphism $\tilde{\mathscr{C}}_X \cong \tilde{\mathscr{C}_{L_2}}$ induced from the marking $\alpha$.
	
	Set
	$$
		\Omega_{M^{\perp}} = \{ [\eta] \in \mathbb{P}(M^{\perp}_{\mathbb{C}}) ; (\eta, \eta)=0, (\eta, \bar{\eta}) >0 \}.
	$$
	Let $\Omega^+_{M^{\perp}}$ be the connected component of $\Omega_{M^{\perp}}$
	satisfying $\Omega^+_{M^{\perp}} \subset \Omega^+_{h^{\perp}}$.
	Let $(X, \iota) \in \tilde{\mathcal{M}}_{M, \mathcal{K}}$ and choose an admissible marking $\alpha : H^2(X, \mathbb{Z}) \to L_2$ for $(M, \mathcal{K})$.
	Since $\alpha( H^{2,0}(X) ) \subset M^{\perp}_{\mathbb{R}} \subset (h^{\perp})_{\mathbb{R}}$ and $(X, \alpha) \in \mathfrak{M}_{L_2}^{\circ}$, we have
	$$
		\alpha( H^{2,0}(X) ) \in \Omega_{M^{\perp}} \cap \Omega^+_{h^{\perp}} =\Omega^+_{M^{\perp}}.
	$$ 
	
	By definition, every $(X, \iota) \in \tilde{\mathcal{M}}_{M, \mathcal{K}}$ carries an admissible marking for $(M, \mathcal{K})$.
	Moreover, if $\alpha, \beta : H^2(X, \mathbb{Z}) \to L_2$ are two admissible markings for $(M, \mathcal{K})$, 
	then we have $\alpha \circ \beta^{-1} \in \Gamma(\mathcal{K})$,
	where we define
	$$
		\Gamma(\mathcal{K}) =\{ \sigma \in \mon(L_2) ; \sigma \circ \iota_M = \iota_M \circ \sigma  \text{ and } \sigma(\mathcal{K})= \mathcal{K} \} =\{ \sigma \in \Gamma(M) ; \sigma(\mathcal{K})=\mathcal{K} \}.
	$$
	
	Let
	$$
		\Gamma_{M^{\perp}, \mathcal{K}} = \{ \sigma|_{M^{\perp}} \in O(M^{\perp}) ; \sigma \in \Gamma(\mathcal{K}) \}.
	$$ 
	By \cite[Proposition 10.2]{MR3519981}, $\Gamma_{M^{\perp}, \mathcal{K}}$ is a finite index subgroup of $O^+(M^{\perp})$.
	Therefore we obtain an orthogonal modular variety
	$$
		\mathcal{M}_{M, \mathcal{K}} = \Omega^+_{M^{\perp}}/\Gamma_{M^{\perp}, \mathcal{K}} .
	$$
	
	We define the period map $P_{M, \mathcal{K}} : \tilde{\mathcal{M}}_{M, \mathcal{K}} \to \mathcal{M}_{M, \mathcal{K}}$ by
	$$
		P_{M, \mathcal{K}}(X, \iota) = [\alpha(H^{2,0}(X))],
	$$
	where $\alpha : H^2(X, \mathbb{Z}) \to L_2$ is an admissible marking for $(M, \mathcal{K})$.
	
	Let 
	$$
		\Delta(M^{\perp}) =\{ \delta \in M^{\perp} ; \delta^2=-2, \text{or } \delta^2=-10, (\delta, L_2)=2\mathbb{Z} \},
	$$
	and set 
	$$
		\mathscr{D}_{M^{\perp}} = \bigcup_{\delta \in \Delta(M^{\perp})} H_{\delta} \subset \Omega_{M^{\perp}},
	$$
	where
	\begin{align}\label{al6-2-1-1}
		H_{\delta} = \{ x \in \Omega_{M^{\perp}} ; (x, \delta)=0 \}.
	\end{align}
	By \cite[Lemma 7.7]{MR3519981}, $\mathscr{D}_{M^{\perp}}$ is locally finite in $\Omega_{M^{\perp}}$
	and is viewed as a reduced divisor on $\Omega_{M^{\perp}}$.
	Set $\bar{\mathscr{D}}_{M^{\perp}}= \mathscr{D}_{M^{\perp}}/\Gamma_{M^{\perp}, \mathcal{K}}$.
	Then $\bar{\mathscr{D}}_{M^{\perp}}$ is a reduced divisor on $\mathcal{M}_{M, \mathcal{K}}$.
	We set 
	$$
		\mathcal{M}^{\circ}_{M, \mathcal{K}} = \mathcal{M}_{M, \mathcal{K}} \setminus \bar{\mathscr{D}}_{M^{\perp}} \quad \text{  and  } \quad \Omega_{M^{\perp}}^{\circ} = \Omega^+_{M^{\perp}} \setminus \mathscr{D}_{M^{\perp}}.
	$$
	By \cite[Lemma 9.5 and Proposition 9.9]{MR3519981}, the image of the period map $P_{M, \mathcal{K}} : \tilde{\mathcal{M}}_{M, \mathcal{K}} \to \mathcal{M}_{M, \mathcal{K}}$ is $\mathcal{M}^{\circ}_{M, \mathcal{K}}$.

	Let $M_0$ be a primitive hyperbolic 2-elementary sublattice of $L_{K3}$.
	Since $L_{K3}$ is unimodular and since $M_0$ is 2-elementary, the involution
	$$
		M_0 \oplus M_0^{\perp} \to M_0 \oplus M_0^{\perp}, \quad (m,n) \mapsto (m,-n) 
	$$
	extends uniquely to an involution $\iota_{M_0} \in O(L_{K3})$ by \cite[Corollary 1.5.2]{Nikulin1980IntegralSB}.
	
	Let $Y$ be a K3 surface and $\sigma : Y \to Y$ be an antisymplectic involution on $Y$.
	Set
	$$
		H^2(Y, \mathbb{Z})^{\sigma} = \left\{ x \in H^2(Y, \mathbb{Z}) ; \sigma^*x = x \right\}.
	$$
	Let $\alpha : H^2(Y, \mathbb{Z}) \to L_{K3}$ be an isometry.
	We call the pair $(Y, \sigma)$ a 2-elementary K3 surface of type $M_0$ if the restriction of $\alpha$ is an isometry from $H^2(Y, \mathbb{Z})^{\sigma}$ to $M_0$.
	
	Since $\sign(M_0^{\perp}) =(2, \rk M_0^{\perp}-2)$, $\Omega_{M_0^{\perp}}$ consists of two connected components.
	We fix a connected component $\Omega^+_{M_0^{\perp}}$ of $\Omega_{M_0^{\perp}}$.
	We obtain the orthogonal modular variety 
	$$
		\mathcal{M}_{M_0}=\Omega^+_{M_0^{\perp}}/O^+(M_0^{\perp})
	$$
	of dimension $20-\rk(M_0)$.
	
	We set
	$
		\Delta(M_0^{\perp}) =\{ d \in M_0^{\perp} ; d^2=-2 \},
	$
	and
	$$
		\mathscr{D}_{M_0^{\perp}} = \bigcup_{d \in \Delta(M_0^{\perp})} H_d \subset \Omega^+_{M_0^{\perp}}.
	$$
	Here $H_d$ is the divisor on $\Omega^+_{M_0^{\perp}}$ defined in the same way as in (\ref{al6-2-1-1}).
	By \cite[Proposition 1.9.]{MR2047658}, $\mathscr{D}_{M_0^{\perp}}$ is locally finite and is viewed as a reduced divisor on $\Omega^+_{M_0^{\perp}}$.
	Set 
	$$
		\bar{\mathscr{D}}_{M_0^{\perp}} = \mathscr{D}_{M_0^{\perp}} / O^+(M_0^{\perp}) \quad \text{ and } \quad \mathcal{M}^{\circ}_{M_0} = \mathcal{M}_{M_0} \setminus \bar{\mathscr{D}}_{M_0^{\perp}}.
	$$
	By \cite[Theorem 1.8.]{MR2047658}, $\mathcal{M}^{\circ}_{M_0}$ is a coarse moduli space of 2-elementary K3 surfaces of type $M_0$. 
	
	Recall that $L_2 = L_{K3} \oplus \mathbb{Z}e$.
	We consider the case
	$$
		M= \widetilde{M}_0 =M_0 \oplus \mathbb{Z}e .
	$$
	In this case, we define an involution $\iota_{\widetilde{M}_0} : L_2 \to L_2$ by
	$$
		 \iota_{\widetilde{M}_0}(x_0 + ae)=\iota_{M_0}(x_0) +ae \quad (x_0 \in L_{K3},  a \in \mathbb{Z}).
	$$
	Then $(\widetilde{M}_0, \iota_{\widetilde{M}_0})$ is an admissible sublattice of $L_2$.
	
	Let $\mathcal{K} \in \operatorname{KT}(\widetilde{M}_0)$.
	A hyperplane $H$ of $\widetilde{M}_{0,\mathbb{R}}$ is a face of $\mathcal{K}$ if $H \cap \partial \mathcal{K}$ contains an open subset of $H$.
	
	\begin{dfn}\label{d2-3-1-1}
		A K\"ahler-type chamber $\mathcal{K} \in \operatorname{KT}(\widetilde{M}_0)$ is natural if the hyperplane $M_{0, \mathbb{R}}$ is a face of $\mathcal{K}$.
	\end{dfn}
	
	Since $M_0^{\perp_{L_{K3}}}=\widetilde{M}_0^{\perp_{L_2}}$, we may identify $O(M_0^{\perp})=O(\widetilde{M}_0^{\perp})$ and $\Omega^+_{M_0^{\perp}} = \Omega^+_{\widetilde{M}_0}$.
	
	\begin{thm}\label{t-3-1-1}
		If $\mathcal{K} \in \operatorname{KT}(\widetilde{M}_0)$ is natural, then $\Gamma_{\widetilde{M}_0^{\perp}, \mathcal{K}} = O^+(M_0^{\perp})$.
		In particular, the identity map $\Omega_{M_0^{\perp}} \to \Omega_{\widetilde{M}_0^{\perp}}$ induces an isomorphism of orthogonal modular varieties $\phi : \mathcal{M}_{M_0} \cong \mathcal{M}_{\widetilde{M}_0, \mathcal{K}}$.
	\end{thm}
	
	\begin{proof}
		See \cite[Theorem 2.28 and Corollary 2.29]{I1}.
	\end{proof}
	
	Let $\tilde{\mathcal{M}}_{M_0}$ be the set of isomorphism classes of 2-elementary K3 surfaces of type $M_0$,
	and $\bar{\pi}_{M_0} : \tilde{\mathcal{M}}_{M_0} \to \mathcal{M}_{M_0}$ be the period map defined by
	$$
		\bar{\pi}_{M_0}(Y, \sigma) =  [\alpha(H^{2,0}(Y))] \quad ((Y, \sigma) \in \tilde{\mathcal{M}}_{M_0}),
	$$
	where $\alpha : H^2(Y, \mathbb{Z}) \to L_{K3}$ is a marking of $(Y, \sigma)$.
	Recall that $\tilde{\mathcal{M}}_{\widetilde{M}_0, \mathcal{K}}$ is the set of isomorphism classes of $K3^{[2]}$-type manifolds with involution of type $(\widetilde{M}_0, \mathcal{K})$,
	and $P_{\widetilde{M}_0, \mathcal{K}} : \tilde{\mathcal{M}}_{\widetilde{M}_0, \mathcal{K}} \to \mathcal{M}_{\widetilde{M}_0, \mathcal{K}}$ is the period map.
	We have a natural map $\Phi : \tilde{\mathcal{M}}_{M_0} \to \tilde{\mathcal{M}}_{\widetilde{M}_0, \mathcal{K}}$ defined by
	$$
		\Phi(Y, \sigma) = (Y^{[2]}, \sigma^{[2]})  \quad ((Y, \sigma) \in \tilde{\mathcal{M}}_{M_0}).
	$$
	Then the following diagram commutes:
	$$
		\xymatrix{
			\tilde{\mathcal{M}}_{M_0} \ar[r]^{\Phi}  \ar[d]_{\bar{\pi}_{M_0}} & \tilde{\mathcal{M}}_{\widetilde{M}_0, \mathcal{K}} \ar[d]^{P_{\widetilde{M}_0, \mathcal{K}}}   \\
			 \mathcal{M}_{M_0} \ar[r]_{\cong}^{\phi} & \mathcal{M}_{\widetilde{M}_0, \mathcal{K}}  
		}
	$$
	
	\begin{exa}\label{e-2-35}
		We define a sublattice $M_0$ of $L_{K3}$ by $M_0 =\mathbb{Z}h$, where $h \in L_{K3}$ satisfies $h^2=2$.
		It is a primitive hyperbolic 2-elementary sublattice of $L_{K3}$.
		We set $\widetilde{M}_0=M_0 \oplus \mathbb{Z} e$.
		Then $\widetilde{M}_0$ is an admissible sublattice of $L_2$.
		By \cite[Example 9.12]{MR3519981}, we have
		$$
			\Delta(\widetilde{M}_0) = \pm \{e, 2h+3e, 2h-3e \} \quad \text{and} \quad \operatorname{KT}(\widetilde{M}_0) /\Gamma_{\widetilde{M}_0} =\{ [\mathcal{K}], [\mathcal{K}'] \},
		$$
		where $\mathcal{K}$ and $\mathcal{K}'$ are K\"ahler-type chambers defined by $\mathcal{K}= \mathbb{R}_{> 0} h +\mathbb{R}_{> 0}(3h+2e)$ and $\mathcal{K}' =\mathbb{R}_{> 0}(3h+2e) +\mathbb{R}_{> 0}(h+e)$.
		Then $\mathcal{K}$ is a natural chamber.
		By Theorem \ref{t-3-1-1}, $\Gamma_{\widetilde{M}_0^{\perp}, \mathcal{K}} = O^+(M_0^{\perp})$ and the identity map $\Omega_{M_0^{\perp}} \to \Omega_{\widetilde{M}_0^{\perp}}$ gives an isomorphism of orthogonal modular varieties $\mathcal{M}_{M_0} \cong \mathcal{M}_{\widetilde{M}_0, \mathcal{K}}$.
		
		Moreover a general point of $\mathcal{M}_{\widetilde{M}_0, \mathcal{K}}$ is the period of $(Y^{[2]}, \sigma^{[2]})$, where $(Y, \sigma)$ is a 2-elementary K3 surface of type $M_0$.
		A 2-elementary K3 surface $(Y, \sigma)$ of type $M_0$ is given by the double cover $Y \to \mathbb{P}^2$ branched over a smooth sextic and the covering involution $\sigma$.
		
		Consider Ohashi-Wandel's involution $\elm_{Y/\sigma}(\sigma^{[2]})$ of the Mukai flop $\elm_{Y/\sigma}(Y^{[2]})$ of $Y^{[2]}$ along $Y/\sigma = \mathbb{P}^2$ (\cite{ohashi2013non}, \cite{MR3519981}, and \cite[Example 2.8]{I1}).
		By \cite[Example 9.12]{MR3519981}, $(\elm_{Y/\sigma}(Y^{[2]}), \elm_{Y/\sigma}(\sigma^{[2]}))$ is a manifold of $K3^{[2]}$-type with involution of type $(\widetilde{M}_0, \mathcal{K}')$.
	\end{exa}

	For the non-natural K\"ahler-type chamber $[\mathcal{K}']$, we still have the following.
	
	\begin{thm}\label{p-2-2}
		With the same notation as in Example \ref{e-2-35}, we have $\Gamma_{\widetilde{M}_0^{\perp}, \mathcal{K}'} = O^+(M_0^{\perp})$.
		In particular, the identity map $\Omega_{M_0^{\perp}} \to \Omega_{\widetilde{M}_0^{\perp}}$ induces an isomorphism of orthogonal modular varieties $\psi : \mathcal{M}_{M_0} \cong \mathcal{M}_{\widetilde{M}_0, \mathcal{K}'}$.
	\end{thm}
	
	\begin{proof}
		See \cite[Proposition 2.31]{I1}.
	\end{proof}
	
	We have a natural map $\Psi : \tilde{\mathcal{M}}_{M_0} \to \tilde{\mathcal{M}}_{\widetilde{M}_0, \mathcal{K}'}$ defined by
	$$
		\Psi(Y, \sigma) = (\elm_{Y/\sigma}(Y^{[2]}), \elm_{Y/\sigma}(\sigma^{[2]}))  \quad ((Y, \sigma) \in \tilde{\mathcal{M}}_{M_0}).
	$$
	By \cite[Proposition 25.14]{MR1963559}, a birational transformation of an irreducible holomorphic symplectic manifold preserves its period.
	Therefore the following diagram commutes:
	$$
		\xymatrix{
			\tilde{\mathcal{M}}_{M_0} \ar[r]^{\Psi}  \ar[d]_{\bar{\pi}_{M_0}} & \tilde{\mathcal{M}}_{\widetilde{M}_0, \mathcal{K}'} \ar[d]^{P_{\widetilde{M}_0, \mathcal{K}'}}   \\
			 \mathcal{M}_{M_0} \ar[r]_{\cong}^{\psi} & \mathcal{M}_{\widetilde{M}_0, \mathcal{K}'} . 
		}
	$$

\subsubsection{The invariant $\tau_{M, \mathcal{K}}$ and its properties.}\label{sss-2-1-2}
	
	We fix an admissible sublattice $M$ of $L_2$ and a K\"ahler-type chamber $\mathcal{K} \in \ktm$.
	Let $(X, \iota) \in \tilde{\mathcal{M}}_{M, \mathcal{K}}$ and let $h_X$ be an $\G$-invariant \K metric on $X$. 
	Let $\omega_X$ be the $\G$-invariant \K form attached to $h_X$.
	
	Let $\tau_{\iota}(\bar{\Omega}_X^1)$ be the equivariant analytic torsion of the cotangent bundle $\bar{\Omega}_X^1 =(\Omega_X^1, h_X)$ endowed with the hermitian metric induced from $h_X$.
	
	The fixed locus of $\iota : X \to X$ is denoted by $X^{\iota}$.
	By \cite[Theorem 1]{MR2805992}, $X^{\iota}$ is a possibly disconnected compact complex surface.
	Let $X^{\iota} = \sqcup_i Z_i$ be the decomposition into the connected components.
	Let $\tau(\bar{\mathcal{O}}_{X^{\iota}})$ be the analytic torsion of the trivial bundle $\bar{\mathcal{O}}_{X^{\iota}}$ with respect to the canonical metric.
	Here $\tau(\bar{\mathcal{O}}_{X^{\iota}})$ is given by 
	$$
		\tau(\bar{\mathcal{O}}_{X^{\iota}}) = \prod_i \tau(\bar{\mathcal{O}}_{Z_i}),
	$$
	where $\tau(\bar{\mathcal{O}}_{Z_i})$ is the analytic torsion of the trivial bundle $\bar{\mathcal{O}}_{Z_i}$ with respect to the canonical metric.
	
	The volume of $(X, \omega_{X})$ is defined by
	$$
		\vol(X, \omega_{X}) = \int_X \frac{\omega_{X}^4}{4!}.
	$$
	Set $\omega_{X^{\iota}} = \omega_X|_{X^{\iota}}$.
	This is a \K form on $X^{\iota}$ attached to $h_{X^{\iota}} =h_X|_{X^{\iota}}$.
	We define the volume of $(X^{\iota}, \omega_{X^{\iota}})$ by
	$$
		\vol(X^{\iota}, \omega_{X^{\iota}}) =\prod_i \vol(Z_i, \omega_X|_{Z_i}) = \prod_i \int_{Z_i} \frac{(\omega_X|_{Z_i})^2}{2!}.
	$$
	The covolume of the lattice $\operatorname{Im}(H^1(X^{\iota}, \mathbb{Z}) \to H^1(X^{\iota}, \mathbb{R}))$ with respect to the $L^2$-metric induced from $h_X$ is denoted by $\vol_{L^2}(H^1(X^{\iota}, \mathbb{Z}), \omega_{X^{\iota}})$ .
	Namely, 
	$$
		\vol_{L^2}(H^1(X^{\iota}, \mathbb{Z}), \omega_{X^{\iota}}) = \det(\langle e_i, e_j \rangle_{L^2}),
	$$
	where $e_1, \dots ,e_{b_1(X^{\iota})}$ is an integral basis of $\operatorname{Im}(H^1(X^{\iota}, \mathbb{Z}) \to H^1(X^{\iota}, \mathbb{R}))$.
		
	We define a real-valued function $\varphi$ on $X$ by
	$$
		\varphi = \frac{\omega_{X}^4/4!}{\eta^2 \wedge \bar{\eta}^2} \frac{\|\eta^2\|_{L^2}^2}{\vol(X, \omega_{X})},
	$$
	where $\eta$ is a holomorphic symplectic 2-form on $X$.
	Obviously, $\varphi$ is independent of the choice of $\eta$.
	We define a positive number $A(X, \iota, h_X) \in \mathbb{R}_{>0}$ by
	\begin{align}\label{al2-3-1-2}
		A(X, \iota, h_X) =\exp \left[ \frac{1}{48} \int_{X^{\iota}} (\log \varphi) \varOmega \right],
	\end{align}
	where $\varOmega$ is a characteristic form on $X^{\iota}$ defined by
	\begin{align}\label{al2-3-1-3}
		\varOmega = c_1(TX^{\iota}, h_{X^{\iota}})^2 -8c_2(TX^{\iota}, h_{X^{\iota}}) -c_1(TX, h_X)|^2_{X^{\iota}} +3c_2(TX, h_X)|_{X^{\iota}}.
	\end{align}
	Here we denote by $c_i(TX, h_X)$, $c_i(TX^{\iota}, h_{X^{\iota}})$ the $i$-th Chern forms of the holomorphic hermitian vector bundles $(TX, h_X)$, $(TX^{\iota}, h_{X^{\iota}})$, respectively.
	Note that if $h_X$ is Ricci-flat, then we have $\varphi =1$ and $A(X, \iota, h_X) =1$.
	
	We set 
	$$
		t=\operatorname{Tr}(\iota^*|_{H^{1,1}(X)}).
	$$
	By \cite[Theorem 2]{MR2805992}, $t$ is an odd number with $-19 \leqq t \leqq 21$.
	By the definition of the admissible sublattice $(M, \iota_M)$, we have $t= \operatorname{Tr}(\iota_M)+2$.
	Therefore $t$ depends only on $(M, \iota_M)$ and is a constant function on $\tilde{\mathcal{M}}_{M, \mathcal{K}}$.
	
	\begin{dfn}\label{d-3-1}
		We define a real number $\tau_{M, \mathcal{K}}(X, \iota)$ by
		\begin{align*}
			\tau_{M, \mathcal{K}}(X, \iota)=\tau_{\iota}(\bar{\Omega}_X^1) &\vol(X, \omega_{X})^{\frac{(t-1)(t-7)}{16}} A(X, \iota, h_X) \\
			&\cdot \tau(\bar{\mathcal{O}}_{X^{\iota}})^{-2} \vol(X^{\iota}, \omega_{X^{\iota}})^{-2} \vol_{L^2}(H^1(X^{\iota}, \mathbb{Z}), \omega_{X^{\iota}}).
		\end{align*}
	\end{dfn}
	
	Let $f : (\X, \iota) \to S$ be a family of $K3^{[2]}$-type manifolds with involution of type $(M, \mathcal{K})$
	and let $h_{\xs}$ be an $\iota$-invariant fiberwise K\"ahler metric on $T\xs$.
	Let $h_{\xss}$ be the induced metric on $T\xss$.
	The $\iota$-invariant hermitian metric on $\Omega^1_{\xs}$ induced from $h_{\xs}$ is also denoted by $h_{\xs}$.
	
	We define the characteristic form $\omega_{H^{\cdot}(\xss)} \in A^{1,1}(S)$ by
	\begin{align}\label{al6-2-1-1}
		\omega_{H^{\cdot}(\xss)} = c_1(f_*\Omega^1_{\xss}, h_{L^2}) -c_1(R^1f_*\mathcal{O}_{\mathscr{X}^{\iota}}, h_{L^2}) -2c_1(f_*K_{\xss}, h_{L^2}).
	\end{align}
	By \cite[Lemma 3.13]{I1}, $\omega_{H^{\cdot}(\xss)}$ is independent of the choice of the $\G$-invariant fiberwise \K metric $h_{\xs}$.
	
	\begin{thm}\label{p-3-4}
		We define a real-valued function $\tau_{M, \mathcal{K}, \xs}$ on $S$ by 
		$$
			\tau_{M, \mathcal{K}, \xs}(s) =\tau_{M, \mathcal{K}}(X_s, \iota_s) \quad (s \in S).
		$$
		Then $\tau_{M, \mathcal{K}, \xs}$ is smooth and satisfies
		$$
			-dd^c \log \tau_{M, \mathcal{K}, \xs} = \frac{(t+1)(t+7)}{16} c_1(f_*K_{\xs}, h_{L^2}) +\omega_{H^{\cdot}(\xss)}.
		$$
	\end{thm}
	
	\begin{proof}
		See \cite[Theorem 3.12]{I1}.
	\end{proof}
	
	\begin{thm}\label{t-0-1}
		For $(X, \iota) \in \tilde{\mathcal{M}}_{M, \mathcal{K}}$, the real number $\tau_{M, \mathcal{K}}(X, \iota)$ is independent of the choice of an $\iota$-invariant \K form.
		In particular, $\tau_{M, \mathcal{K}}(X, \iota)$ is an invariant of $(X, \iota)$.
	\end{thm}
	
	\begin{proof}
		See \cite[Theorem 3.14]{I1}.
	\end{proof}
	
	By \cite[Lemma 3.15]{I1}, $\tau_{M, \mathcal{K}}$ is viewed as a smooth real-valued function on $\mathcal{M}^{\circ}_{M, \mathcal{K}}$.
	Namely,
	$$
		\tau_{M,\mathcal{K}}(p) = \tau_{M,\mathcal{K}}(X, \iota) \quad ((X, \iota) \in P_{M, \mathcal{K}}^{-1}(p) )
	$$
	is independent of the choice of $(X, \iota) \in P_{M, \mathcal{K}}^{-1}(p)$.
	
	Let $\omega_{\mathcal{M}_{M, \mathcal{K}}}$ be the orbifold \K form on $\mathcal{M}_{M, \mathcal{K}}$ induced from the \K form of the Bergman metric on the period domain $\Omega_{M^{\perp}}$.
	
	In \cite[Lemma 3.16]{I1}, there exists a smooth $(1,1)$-form $\sigma_{M, \mathcal{K}}$ on $\mathcal{M}^{\circ}_{M, \mathcal{K}}$ such that for any $(X, \iota) \in \tilde{\mathcal{M}}_{M, \mathcal{K}}$ we have
	$$
		P_{M, \mathcal{K}}^*\sigma_{M, \mathcal{K}} = c_1(\pi_*\Omega^1_{\XX / \operatorname{Def}(X, \iota)}, h_{L^2}) -c_1(R^1\pi_*\mathcal{O}_{\mathscr{X}^{\iota}}, h_{L^2}) -2c_1(\pi_*K_{\XX / \operatorname{Def}(X, \iota)}, h_{L^2}),
	$$ 
	where $P_{M, \mathcal{K}} : \operatorname{Def}(X, \iota) \to \mathcal{M}_{M, \mathcal{K}}$ is the period map of the Kuranishi family $\pi : (\X, \iota) \to \operatorname{Def}(X, \iota)$ of $(X, \iota)$.
	
	\begin{thm}\label{t-0-A}
		The following equation of differential forms on $\mathcal{M}^{\circ}_{M, \mathcal{K}}$ holds:
		$$
			-dd^c \log \tau_{M, \mathcal{K}} = \frac{(t+1)(t+7)}{8} \omega_{\mathcal{M}_{M, \mathcal{K}}} +\sigma_{ M, \mathcal{K} }.
		$$
	\end{thm}
	
	\begin{proof}
		See \cite[Theorem 3.17]{I1}.
	\end{proof}

\subsection{Asymptotic behavior of $\tau_{M, \mathcal{K}}$}
	
	In this subsection, we keep the notation of the previous section \S1.2.
	In addition, we assume that the fiber $(X_s, \iota_s)$ of $f : (\X, \iota) \to C$ at any point $s \in C^{\circ}=C \setminus \Delta_f$ is a manifold of $K3^{[2]}$-type with involution of type $(M, \mathcal{K})$.

	Fix $0 \in \Delta_f$.
	Let $(\D, s)$ be a coordinate on $C$ centered at $0$ such that $s(\D)$ is the unit disk and that $\D \cap \Delta_f = \{ 0 \}$.
	The restriction of $f$ to $X=f^{-1}(\D)$ is denoted by $f : X \to \D$.
	By Assumption \ref{as-1-2-1}, the singular fiber $X_0 = f^{-1}(0)$ is simple normal crossing.
	
	Let $h_X := h_{\X}|_X$ be the $\G$-invariant \K metric on $X$ induced from the $\G$-invariant \K metric $h_{\X}$ on $\X$ and its associated K\"ahler form is denoted by $\omega_X$. 
	The fiberwise K\"ahler metric on $TX/\D |_{\D^*}$ induced from $h_X$ is denoted by $h_{X/\D}$
	and we set $h_s = h_{X/\D}|_{X_s}$ and $\omega_s = \omega_X|_{X_s}$ for $s \in \D$.
	We set $\omega_{X/\D} = \{ \omega_s \}_{s \in \D}$.
	
	Recall that the \K form of $h_{\X}$ is assumed to be integral. 
	Then $\omega_s$ and $\omega_s|_{X_s^{\iota}}$ are integral K\"ahler forms for $s \in \D^*$.
	By \cite[Proposition 4.2]{MR4255041}, the volumes $\vol(X_s, \omega_s)$, $\vol(X_s^{\iota}, \omega_s|_{X_s^{\iota}})$ and the covolume $\vol_{L^2}(H^1(X_s^{\iota}, \mathbb{Z}), \omega_s|_{X_s^{\iota}})$ are rational numbers.
	Therefore these are constant functions on $\D$ and are denoted by $\vol(X_{\infty}, \omega_{\X}|_{X_{\infty}})$, $\vol(X_{\infty}^{\iota},  \omega_{\X}|_{X_{\infty}^{\iota}})$ and $\vol_{L^2}(H^1(X_{\infty}^{\iota}, \mathbb{Z}), \omega_{\X}|_{X_{\infty}^{\iota}})$, respectively.
	
	We define a smooth function on $\D^*$ by
	$$
		A(X/\D, h_{X/\D})(s) =A(X_s, \iota_s, h_s),
	$$
	where $A(X_s, \iota_s, h_s)$ is a positive number defined in (\ref{al2-3-1-2}).
	
	Consider the relative canonical bundle $K_{X/\D} = K_X \otimes f^*K_{\D}^{-1}$.
	Since $f_*K_{X/\D}$ is a torsion free sheaf on $\D$, it is an invertible sheaf
	and we may regard it as a holomorphic line bundle on $\D$.
	
	\begin{lem}\label{l2-3-2}
		There exists a holomorphic form $\xi \in H^0(X, K_X)$ such that $\operatorname{div}(\xi) \subset X_0$.
	\end {lem}
	
	\begin{proof}
		Since $f_*K_X$ is a holomorphic line bundle on $\D$, it is a trivial line bundle and there exists a nowhere-vanishing holomorphic section $\xi \in H^0(\D, f_*K_X) =H^0(X, K_X)$ such that $f_*K_X =\mathcal{O}_{\D} \cdot \xi$.
		The fiber of $f_*K_X$ at $s \in \D^*$ is $H^0(X_s, K_X|_{X_s}) = \mathbb{C} \xi|_{X_s}$.
		Since $K_X|_{X_s} \cong K_{X_s} \cong \mathcal{O}_{X_s}$, $\xi|_{X_s}$ is a nowhere vanishing holomorphic section on $X_s$.
		Therefore we have $\operatorname{div}(\xi) \subset X_0$.
	\end{proof}
	
	By Lemma \ref{l2-3-2}, we have a holomorphic form $\xi \in H^0(X, K_X)$ such that $\operatorname{div}(\xi) \subset X_0$.
	We define a section $\eta_{X/\D} \in H^0(X, K_{X/\D})$ by 
	$$
		\eta_{X/\D} =\xi \otimes (f^* ds)^{-1}.
	$$
	For $s \in \D^*$, we may identify $\eta_{X/\D}|_{X_s}$ with the Poincar\'e residue
	$$
		\eta_s = \operatorname{Res}_{X_s} \frac{\xi}{f-s} \in H^0(X_s, K_{X_s}).
	$$
	Then we have $\xi|_{X_s} = \eta_s \wedge df$.
	We regard $\eta_{X/\D}$ as a family of holomorphic 4-forms 
	and we regard it as a section of $H^0(\D, f_*K_{X/\D})$.
	
	By \cite[Theorem~4.4]{MR4255041}, we have
	\begin{align}\label{f2-3-3}
		\log \| \eta_s \|^2_{L^2} = \alpha^{4,0} \log |s|^2 + O(\log \log |s|^{-1}), 
	\end{align}
	where $\alpha^{4,0}$ is a rational number defined by (cf. (\ref{al2-1-5-1}))
	$$
		\alpha^{4,0}= \alpha_+^{4,0} + \alpha_-^{4,0}, \quad \alpha^{4,0}_{\pm} = -\frac{1}{2 \pi i} \operatorname{Tr}(^l\log T_s|_{\operatorname{Gr}^4_{F_{\infty, \pm}}H^{4}(X_{\infty})_{\pm}}).
	$$
	
	Consider the projective bundle $\Pi' : \mathbb{P}(T \XX_H)^{\vee} \to \XX_H$.
	Let $\mu' : \XX_H \setminus \Sigma_f \to \mathbb{P}(T \XX_H)^{\vee}$ be the Gauss map defined by $\mu'(x) = [T_x X^{\iota}_{f(x)}]$.
	Let $A \in \mathbb{P}(T \XX_H)^{\vee}$ and set $x = \Pi'(A)$.
	Since the $(+1)$-eigenspace of the $\G$-action on $T_x\X$ is identified with $T_x \XX_H$ and since the $(-1)$-eigenspace is identified with $N_{\XX_H / \X, x}$, $A \oplus N_{\XX_H/\X,x}$ is a hyperplane of $T_x\X = T_x \XX_H \oplus N_{\XX_H / \X , x}$.
	We define an injective holomorphic map $\phi : \mathbb{P}(T \XX_H)^{\vee} \to \mathbb{P}(T \X)^{\vee}$ by 
	$$
		\phi([A]) = [A \oplus N_{\XX_H/\X,x}]. 
	$$
	Therefore $\mathbb{P}(T \XX_H)^{\vee}$ is regarded as a closed submanifold of $\mathbb{P}(T \X)^{\vee}$ via $\phi$.
	
	Recall that the Gauss maps $\nu$ and $\mu$ extend to meromorphic maps $\nu : \X \dashrightarrow \mathbb{P}(\Omega^1_{\X})$ and $\mu : \X \dashrightarrow \mathbb{P}(T\X)^{\vee}$.
	Then there exist a smooth projective manifold $\widetilde{\X}$,
	a normal crossing divisor $E$ on $\widetilde{\X}$, 
	a bimeromorphic morphism $q : \widetilde{\X} \to \X$,
	and two holomorphic maps $\widetilde{\nu} : \widetilde{\X} \to \mathbb{P}(\Omega^1_{\X})$, $\widetilde{\mu} : \widetilde{\X} \to \mathbb{P}(T\X)^{\vee}$
	such that
	\begin{itemize}
		\item $q^{-1}(\Sigma_f) =E$,
		\item the restriction $ q|_{\widetilde{\X} \setminus E} : \widetilde{\X} \setminus E \to \X \setminus \Sigma_f $ is an isomorphism,
		\item on $\widetilde{\X} \setminus E$, $\widetilde{\nu} = \nu \circ q$ and $\widetilde{\mu} = \mu \circ q$.
	\end{itemize}
	
	The strict transform of $\XX_H$ by $q : \widetilde{\X} \to \X$ is denoted by $\widetilde{\XX_H}$.
	The restriction of $q : \widetilde{\X} \to \X$ is also denoted by $q : \widetilde{\XX_H} \to \XX_H$.
	Since the following diagram commutes 
	$$
		\xymatrix{
		                &\widetilde{\XX_H} \setminus E \ar[r]^{q}_{\cong} \ar@{^{(}->}[d] & \XX_H \setminus \Sigma_f  \ar[r]^{\mu'} \ar@{^{(}->}[d] &  \mathbb{P}(T \XX_H)^{\vee}  \ar@{^{(}->}[d]^{\phi}  \\
		                & \widetilde{\X} \setminus E \ar[r]^{q}_{\cong} & \X \setminus \Sigma_f \ar[r]_{\mu} & \mathbb{P}(T \X)^{\vee}  ,					   
		}
	$$
	we have $\widetilde{\mu}(\widetilde{\XX_H}) \subset \mathbb{P}(T \XX_H)^{\vee}$. 
	Therefore there exists a holomorphic map $\widetilde{\mu}' : \widetilde{\XX_H} \to \mathbb{P}(T \XX_H)^{\vee} $ which is an extension of $\mu' \circ q : \widetilde{\XX_H} \setminus E  \to \mathbb{P}(T \XX_H)^{\vee}$.
	
	Let $U$ and $U'$ be the universal hyperplane bundle over $ \mathbb{P}(T \X)^{\vee} $ and $ \mathbb{P}(T \XX_H)^{\vee} $, respectively.
	These are equipped with hermitian metrics $h_U$ and $h_{U'}$ respectively as in \S 1.2.
	We have an isometry of the holomorphic hermitian vector bundle over $\XX_H \setminus \Sigma_f $:
	$$
		(T\XX_H/C, h_{\XX_H/C}) = (\mu')^*(U', h_{U'}).
	$$
	We define a $\partial$ and $\overline{\partial}$-closed smooth form $\widetilde{\varOmega}$ on $\widetilde{\XX_H} $ by
	$$
		\widetilde{\varOmega} = \widetilde{\mu'}^*c_1(U', h_{U'})^2  -8\widetilde{\mu'}^*c_2(U', h_{U'}) -\widetilde{\mu}^*c_1(U,h_U)|_{\widetilde{\XX_H}}^2 +3\widetilde{\mu}^*c_2(U,h_U)|_{\widetilde{\XX_H}}.
	$$
	On $\widetilde{\XX_H} \setminus E$, we have 
	$$
		q^* \varOmega = \widetilde{\varOmega}.
	$$
	
	Recall that $E_0$ was defined by $E_0 = E \cap q^{-1}(X_0)$.
	We define a rational number $\delta(X_0, \Omega^1_{X/\D})$ by
	\begin{align}\label{al4-2-2-1}
		\delta(X_0, \Omega^1_{X/\D}) = \frac{1}{48} \int_{E_0 \cap \widetilde{\XX_H}} \widetilde{\varOmega} -\frac{1}{48} \int_{q^* \operatorname{div}(\xi) \cap \widetilde{\XX_H}} \widetilde{\varOmega} -\frac{1}{16} \alpha^{4,0} (t^2+7).
	\end{align}
	
	\begin{lem}\label{l2-3-4}
		The following identity holds:
		$$
			\log A(X/ \D)(s) = \delta(X_0, \Omega^1_{X/\D}) \log |s|^2 +O( \log \log |s|^{-1}) \quad (s \to 0).
		$$
	\end{lem}
	
	\begin{proof}
		Note that $\omega_X^5/5!$ is the volume form on $X$ and $\omega_{X/\D}^4/4! |_{X_s}$ is the volume form on $X_s$ for each $s \in \D$.
		We have
		$$
			\frac{\omega_X^5}{5!} = \frac{\omega_{X/\D}^4}{4!} \wedge \frac{i df \wedge \overline{df} }{ \| df \|^2 }.
		$$
		Then we have
		\begin{align*}
			\frac{ \eta_{X/\D} \wedge \overline{\eta}_{X/\D} }{ \omega_{X/\D}^4/4! } = \frac{ \eta_{X/\D} \wedge \overline{\eta}_{X/\D} \wedge df \wedge \overline{df} }{ \omega_{X/\D}^4/4! \wedge df \wedge \overline{df} }
			= \frac{i}{ \| df \|^2 } \frac{ \xi \wedge \overline{\xi} }{ \omega_X^5/5! } = \frac{ \| \xi \|^2 }{ \| df \|^2 }.
		\end{align*}
		We set $\widetilde{f} = f \circ q : \widetilde{\XX_H} \to \D$.
		By \cite[Lemma 3.5]{I1}, we have
		\begin{align}\label{al4-2-2-2}
		\begin{aligned}
			&\quad 48 \log A(X/\D, h_{X/\D}) \\
			&= f_*\left[ \log \left\{ \frac{\omega_{X/\D}^4/4!}{\eta_{X/\D} \wedge \overline{\eta}_{X/\D} } \frac{ \| \eta_{X/\D} \|^2_{L^2} }{ \vol(X_{\infty}, \omega_{\X}|_{X_{\infty}}) } \right\} \varOmega \right] \\
			&= f_*\left[ \log \frac{\| df \|^2}{ \| \xi \|^2 } \varOmega \right] -3(t^2+7) \log \frac{ \| \eta_{X/\D} \|^2_{L^2} }{ \vol(X_{\infty}, \omega_{\X}|_{X_{\infty}}) } \\
			&= \widetilde{f}_* \{ \log q^* \| df \|^2 \widetilde{\varOmega} \} -3(t^2+7) \log \| \eta_{X/\D} \|^2_{L^2} \\ 
			&\qquad -\widetilde{f}_* \{ \log q^*\| \xi \|^2 \widetilde{\varOmega} \} +3(t^2+7) \log \vol(X_{\infty}, \omega_{\X}|_{X_{\infty}}).
		\end{aligned}
		\end{align}
		By \cite[Lemma~4.4 and Corollary~4.6]{MR2262777}, we have
		\begin{align}\label{al4-2-2-3}
			\widetilde{f}_* \{ \log q^*\| df \|^2 \widetilde{\varOmega} \} = \left( \int_{E_0 \cap \widetilde{\XX_H}} \widetilde{\varOmega} \right) \log |s|^2 +O(1)
		\end{align}
		and
		\begin{align}\label{al4-2-2-4}
			 \widetilde{f}_* \{ \log q^*\| \xi \|^2 \widetilde{\varOmega} \} = \left(  \int_{q^* \operatorname{div}(\xi) \cap \widetilde{\XX_H} } \widetilde{\varOmega} \right) \log |s|^2 +O(1).
		\end{align}
		By the formula (\ref{f2-3-3}), (\ref{al4-2-2-1}), (\ref{al4-2-2-2}), (\ref{al4-2-2-3}) and (\ref{al4-2-2-4}), the proof is completed.
	\end{proof}
	
	Consider the flat family $f : \X^{\iota}_H \to C$.
	Let us calculate the singularity of the analytic torsion $\tau(\overline{ \mathcal{O} }_{\X^{\iota}_H})$ of $\overline{ \mathcal{O} }_{\X^{\iota}_H} $, where $\mathcal{O}_{\X^{\iota}_H}$ is equipped with the canonical metric.
	Since $\mathcal{O}_{\X^{\iota}_H}$ is a holomorphic line bundle on $\X^{\iota}_H$, the determinant of the cohomology $\lambda( \mathcal{O}_{\X^{\iota}_H})$ is a holomorphic line bundle on $C$. 
	We set
	\begin{align}\label{al6-2-2-1}
		\hat{c}(X^{\iota}_0, \mathcal{O}_{X^{\iota}_0}) = \alpha( X^{\iota}_0, \mathcal{O}_{X^{\iota}_0} ) -\sum_{q=0}^2 (-1)^q \alpha^{0,q},
	\end{align}
	where we define
	\begin{align}\label{al6-2-2-2}
		 \alpha( X^{\iota}_0, \mathcal{O}_{X^{\iota}_0} ) = \int_{ E_0 \cap \widetilde{\XX_H} } \widetilde{\mu'}^* \left\{ Td(U') \frac{ Td( H ) -1}{ c_1(H) } \right\} q^* ch(\mathcal{O}_{\X^{\iota}_H}),
	\end{align}
	and
	\begin{align}\label{al6-2-2-3}
		\alpha^{0,q} = -\frac{1}{2 \pi i} \operatorname{Tr}(^l\log T_s|_{\operatorname{Gr}^0_{F_{\infty}}H^{q}(X^{\iota}_{\infty})}).
	\end{align}
	Here $T$ is the monodromy operator and $F_{\infty}$ is the Steenbrink's limiting Hodge filtration for the family $f : \XX \to C$.
	
	\begin{lem}\label{l3-2-2-2}
		The following identity holds:
		$$
			\log \tau(\overline{ \mathcal{O} }_{\X^{\iota}_H}) = \hat{c}(X^{\iota}_0, \mathcal{O}_{X^{\iota}_0}) \log |s|^2 +O(\log\log |s|^{-1}) \quad (s \to 0).
		$$
	\end{lem}
	
	\begin{proof}
		Let $\| \cdot \|^2_{Q, \lambda( \mathcal{O}_{\X^{\iota}_H})}$ and $\| \cdot \|^2_{L^2, \lambda( \mathcal{O}_{\X^{\iota}_H})}$ be the Quillen and $L^2$-metrics on $\lambda( \mathcal{O}_{\X^{\iota}_H})$, respectively.
		Since $\lambda( \mathcal{O}_{\X^{\iota}_H})|_{\D}$ is a trivial line bundle on $\D$, there exists a nowhere-vanishing holomorphic section $\sigma$ of $\lambda( \mathcal{O}_{\X^{\iota}_H})|_{\D}$.
		By \cite[Theorem 6.1]{MR2262777} (See also \cite{MR1486991}), we have
		\begin{align}\label{al4-2-2-5}
			\log \| \sigma \|^2_{Q, \lambda( \mathcal{O}_{\X^{\iota}_H})} = \alpha( X^{\iota}_0, \mathcal{O}_{X^{\iota}_0} ) \log |s|^2 +O(1) \quad (s \to 0).
		\end{align}
		By \cite[Theorem 4.4.(1)]{MR4255041}, we have
		\begin{align}\label{al4-2-2-6}
			\log \| \sigma \|^2_{L^2, \lambda( \mathcal{O}_{\X^{\iota}_H})} = \left( \sum_{q=0}^2 (-1)^q \alpha^{0,q} \right) \log |s|^2 +O( \log \log |s|^{-1}) \quad (s \to 0).
		\end{align}
		By (\ref{al4-2-2-5}) and (\ref{al4-2-2-6}), we obtain the desired result.
	\end{proof}
	
	We set
	\begin{align}\label{al6-2-2-4}
		\kappa(X_0, \Omega^1_{X/\D}, \mathcal{O}_{X^{\iota}}) = c(X_0, \Omega^1_{\X/C}) +\delta(X_0, \Omega^1_{X/\D}) -2\hat{c}(X^{\iota}_0, \mathcal{O}_{X^{\iota}_0})
	\end{align}
	
	\begin{thm}\label{t2-3-5}
		The following identity holds:
		$$
			\log \tau_{M, \mathcal{K}, \xs} = \kappa(X_0, \Omega^1_{X/\D}, \mathcal{O}_{X^{\iota}}) \log |s|^2 +O( \log \log |s|^{-1} ) \quad (s \to 0).
		$$
	\end{thm}
	
	\begin{proof}
		As we remarked above, the volumes $\vol(X_{\infty}, \omega_{\X}|_{X_{\infty}})$, $\vol(X_{\infty}^{\iota},  \omega_{\X}|_{X_{\infty}^{\iota}})$ and the covolume $\vol_{L^2}(H^1(X_{\infty}^{\iota}, \mathbb{Z}), \omega_{\X}|_{X_{\infty}^{\iota}})$ are constant functions on $\D$.
		The result follows from Theorem \ref{ta-1-2}, Definition \ref{d-3-1}, Lemmas \ref{l2-3-4} and \ref{l3-2-2-2}.
	\end{proof}

\subsection{Construction of a 1-parameter degeneration over a finite branched covering of a curve on the moduli space}\label{ss-2-3}
	
	The Baily-Borel compactification of the orthogonal modular variety $\mathcal{M}_{M, \mathcal{K}} = \Omega^+_{M^{\perp}}/\Gamma_{M^{\perp}, \mathcal{K}}$ is denoted by $\mathcal{M}^*_{M, \mathcal{K}}$.
	Let $\Pi_{M^{\perp}, \mathcal{K}} : \Omega^+_{M^{\perp}} \to \Omega^+_{M^{\perp}}/\Gamma_{M^{\perp}, \mathcal{K}}$ be the natural projection.
	
	\begin{assumption}\label{as4-2-3-1}
		Let $C \subset \mathcal{M}^*_{M, \mathcal{K}}$ be an irreducible projective curve.
		We assume that $C$ is not contained in $\bar{\mathscr{D}}_{M^{\perp}} \cup (\mathcal{M}^*_{M, \mathcal{K}} \setminus \mathcal{M}_{M, \mathcal{K}})$.
	\end{assumption}
	
	\begin{prop}\label{p2-3-1}
		There exist a smooth projective curve $B$,
		a morphism $\varphi: B \to \mathcal{M}^*_{M, \mathcal{K}}$,
		a smooth projective variety $\mathscr{X}$ of complex dimension $5$,
		an involution $\iota : \mathscr{X} \to \mathscr{X}$,
		and a surjective morphism $f : \mathscr{X} \to B$ such that
		\begin{enumerate}[ label= \rm{(\arabic*)} ]
			\item $\varphi(B)=C$.
			\item $f \circ \iota =f$.
			\item There exists a Zariski open subset $B^{\circ} \neq \emptyset$ of $B$ such that $\varphi(B^{\circ}) \subset \mathcal{M}^{\circ}_{M, \mathcal{K}}$
				and for each $b \in B^{\circ}$, the fiber $(X_b, \iota_b)$ is a manifold of $K3^{[2]}$-type with involution of type $(M, \mathcal{K})$ whose period is $\varphi(b)$. 
			\item $f : \mathscr{X} \to B$ is a normal crossing degeneration. 
				Namely, its singular fibers are simple normal crossing divisor on $\X$. 
		\end{enumerate}
	\end{prop}
	
	\begin{proof}
		We follow \cite[Theorem 2.8]{MR2047658} and \cite[Theorem 3.1.(1)]{MR2968220}.\par
		\noindent
		\textbf{Step 1 } 
		By Assumption \ref{p2-3-1}, there exists a point $x_0 \in C \cap \mathcal{M}^{\circ}_{M, \mathcal{K}}$ such that $x_0 \notin \operatorname{Sing} C$.
		Let $\tilde{x}_0 \in \Pi_{M^{\perp}, \mathcal{K}}^{-1} (x)$ be a point.
		The stabilizer of $\Gamma_{M^{\perp}, \mathcal{K}}$ at $\tilde{x}_0$ is denoted by $\Gamma_{M^{\perp}, \mathcal{K}, \tilde{x}_0}$.
		Since $\Gamma_{M^{\perp}, \mathcal{K}}$ acts properly discontinuously on $\Omega^+_{M^{\perp}}$, 
		there exists an open neighborhood $\widetilde{W} \subset \Omega^+_{M^{\perp}} \setminus \mathscr{D}_{M^{\perp}}$ of $\tilde{x}_0$ such that
		$$
			\{ \gamma \in \Gamma_{M^{\perp}, \mathcal{K}} ; \gamma ( \widetilde{W} ) \cap \widetilde{W} \neq \emptyset \} = \Gamma_{M^{\perp}, \mathcal{K}, \tilde{x}_0 }.
		$$
		Note that $\Gamma_{M^{\perp}, \mathcal{K}, \tilde{x}_0}$ is a finite group.
		By shrinking $\widetilde{W}$, we assume that $\widetilde{W}$ is $\Gamma_{M^{\perp}, \mathcal{K}, \tilde{x}_0}$-invariant.
		Then the restriction $\Pi_{M^{\perp}, \mathcal{K}} : \widetilde{W} \to \Pi_{M^{\perp}, \mathcal{K}}(\widetilde{W})$ can be identified with the quotient map $\widetilde{W} \to \widetilde{W}/ \Gamma_{M^{\perp}, \mathcal{K}, \tilde{x}_0}$ by the finite group $\Gamma_{M^{\perp}, \mathcal{K}, \tilde{x}_0}$.
		Therefore there exists a smooth analytic curve $U$ of $\Pi_{M^{\perp}, \mathcal{K}}^{-1}(C) \cap \widetilde{W} (\subset \Omega^+_{M^{\perp}})$ such that $\Pi_{M^{\perp}, \mathcal{K}}(U)$ is a smooth analytic curve of $C \cap \Pi_{M^{\perp}, \mathcal{K}}(\widetilde{W})$ 
		and the restriction $\Pi_{M^{\perp}, \mathcal{K}} : U \to \Pi_{M^{\perp}, \mathcal{K}}(U)$ is an isomorphism.
		
		Let $\tilde{x} \in U \subset \Omega^+_{M^{\perp}} \setminus \mathscr{D}_{M^{\perp}}$ and set $x = \Pi_{M^{\perp}, \mathcal{K}}(\tilde{x}) \in \mathcal{M}^{\circ}_{M, \mathcal{K}}$.
		By \cite[Lemma 9.5 and Proposition 9.9]{MR3519981}, the image of the period map $P_{M, \mathcal{K}} : \tilde{\mathcal{M}}_{M, \mathcal{K}} \to \mathcal{M}_{M, \mathcal{K}}$ is $\mathcal{M}^{\circ}_{M, \mathcal{K}}$.
		There exists $(X, \iota) \in \tilde{\mathcal{M}}_{M, \mathcal{K}}$ such that $P_{M, \mathcal{K}}(X, \iota) =x$.
		Let $\pi : (\X, \iota) \to \operatorname{Def}(X, \iota)$ be the Kuranishi family of $(X, \iota)$.
		By the local Torelli theorem, we may regard $\operatorname{Def}(X, \iota)$ as an open neighborhood of $\tilde{x}$ in $\Omega^+_{M^{\perp}}$.
		By shrinking $U$, we assume $U \subset \operatorname{Def}(X, \iota)$.
		Let $\tilde{\gamma} : U \to \operatorname{Def}(X, \iota)$ be the inclusion.
		Then there exists a family $P_{\mathcal{Z}} : (\mathcal{Z}, \iota_{\mathcal{Z}}) \to U$ of $K3^{[2]}$-type manifolds with involution of type $(M, \mathcal{K})$ with period $\tilde{\gamma} : U \to \operatorname{Def}(X, \iota) (\subset \Omega^+_{M^{\perp}})$.
		
		By \cite[Lemma 3.11]{I1}, we may assume that $P_{\mathcal{Z}} : \mathcal{Z} \to U$ is a projective morphism.
		Since $\mathcal{Z}$ has a $\mu_2$-equivariant relatively ample line bundle, we may assume that there exist a sufficiently large positive integer $N>0$ and an involution $\iota_{\mathbb{P}^N} : \mathbb{P}^N \to \mathbb{P}^N$ such that
		$\mathcal{Z} \subset \mathbb{P}^N \times U$, $P_{\mathcal{Z}} = \operatorname{pr}_2|{\mathcal{Z}}$, and $\iota_{\mathcal{Z}} = \left( \iota_{\mathbb{P}^N} \times \operatorname{id}_{U} \right)|_{\mathcal{Z}}$,
		where $ \operatorname{pr}_2 : \mathbb{P}^N \times U \to U$ is the second projection.
		
		The Hilbert polynomial of $Z_t = P_{\mathcal{Z}}^{-1}(t)$ with respect to $\mathcal{O}_{\mathbb{P}^N}(1)$ is denoted by $P(m)$.
		Let $\hilb$ be the Hilbert scheme parametrizing the closed subschemes of $\mathbb{P}^N$ with Hilbert polynomial $P$.
		This is a projective variety and there is the universal projective flat family $\pi : \mathcal{X} \to \hilb$.
		The point of $\hilb$ corresponding to a subscheme $Z$ of $\mathbb{P}^N$ with Hilbert polynomial $P$ is denoted by $[Z]$.
		We define a morphism $c: U \to \hilb$ by $c(t)= [Z_t]$.
		The family $P_{\mathcal{Z}}: \mathcal{Z} \to U$ is given by the pullback of $\pi : \mathcal{X} \to \hilb$ by $c: U \to \hilb$.
		
		The involution $\iota_{\mathbb{P}^N} : \mathbb{P}^N \to \mathbb{P}^N$ induces an involution $I : \hilb \to \hilb$ given by $I([Z])=[\iota_{\mathbb{P}^N}(Z)]$.
		We denote the fixed locus of $I$ by $\hilbi$, and denote the restriction of $\pi$ by $\pi^I : \mathcal{X}^I = \pi^{-1}(\hilbi) \to \hilbi$.
		There exists an involution $\mathcal{I} : \mathcal{X}^I \to \mathcal{X}^{I}$ such that $\pi^I \circ \mathcal{I} = \pi^I$ and $\mathcal{I} = \left( \iota_{\mathbb{P}^N} \times \operatorname{id}_{\hilb} \right)|_{\mathcal{X}^I}$.
		Since $\iota_{\mathcal{Z}} = \left( \iota_{\mathbb{P}^N} \times \operatorname{id}_{U} \right)|_{\mathcal{Z}}$, we have $c(U) \subset \hilbi$.
		Let $T$ be the Zariski closure of $c(U)$ in $\hilbi$.
		This is an irreducible reduced projective variety and $c(U) \cap \left( T \setminus \operatorname{Sing}T \right) \neq \emptyset$.
		The family $P_{\mathcal{Z}}: (\mathcal{Z}, \iota_{\mathcal{Z}}) \to U$ is the pullback of $\pi^I|_T : \left( \mathcal{X}^I|_T, \mathcal{I}|_T \right) \to T$ by $c:U \to T$.\par
		\noindent
		\textbf{Step 2 }
		Set $D=\{ [V] \in T ; \operatorname{Sing}V \neq \emptyset \}$ and $T^{\circ} = T \setminus \left( \operatorname{Sing}T \cup D \right)$.
		Then the restriction $\pi^I|_{T^{\circ}} : \left( \mathcal{X}^I|_{T^{\circ}}, \mathcal{I}|_{T^{\circ}} \right) \to T^{\circ}$ is a family of $K3^{[2]}$-type manifolds with antisymplectic involution of type $(M, \mathcal{K})$.
		We denote its period map by $\mathcal{P}_{T^{\circ}} : T^{\circ} \to \mathcal{M}^{\circ}_{M, \mathcal{K}}$.
		Since $c(U) \cap \left( T \setminus \operatorname{Sing}T \right) \neq \emptyset$ and $c(U) \subset T$, we may assume that $c(U) \subset T^{\circ}$.
		Since the period map for $P_{\mathcal{Z}} : (\mathcal{Z}, \iota_{\mathcal{Z}}) \to U$ is $ \gamma := \Pi_{M^{\perp}, \mathcal{K}} \circ \tilde{\gamma} : U \to \mathcal{M}^{\circ}_{M, \mathcal{K}}$,
		we have $\mathcal{P}_{T^{\circ}}(c(t))= \gamma(t)$ for any $t \in U$.
		
		\begin{claim}\label{claim2-3-1}
			The period map  $\mathcal{P}_{T^{\circ}} : T^{\circ} \to \mathcal{M}^{\circ}_{M, \mathcal{K}}$ extends to a meromorphic map $\mathcal{P}_T : T \dashrightarrow \mathcal{M}^*_{M, \mathcal{K}}$.
			Moreover, the Zariski closure of the graph of $\mathcal{P}_{T^{\circ}}$ in $T \times \mathcal{M}^*_{M, \mathcal{K}}$ is contained in $T \times C$. 
		\end{claim}
		
		\begin{proof}
			Let $\mu : \tilde{T} \to T$ be a resolution of $T$ with the following properties:
			$\tilde{T}$ is a projective variety, $\mu$ is a birational morphism, and there is a normal crossing divisor $E$ on $\tilde{T}$ such that $\mu|_{\tilde{T} \setminus E} : \tilde{T} \setminus E \to T^{\circ}$ is an isomorphism.
			Note that $\mathcal{P}_{T^{\circ}} : T^{\circ} \to \mathcal{M}^{\circ}_{M, \mathcal{K}}$ is locally liftable.
			Namely, for each point $t \in T^{\circ}$, there is an open neighborhood $V \subset T^{\circ}$, and a holomorphic map $\tilde{\mathcal{P}}:V \to \Omega^+_{M^{\perp}}$ such that $\Pi_{M, \mathcal{K}} \circ  \tilde{\mathcal{P}}= \mathcal{P}_{T^{\circ}}|_V$.
			Therefore $\mathcal{P}_{T^{\circ}} \circ \mu|_{\tilde{T} \setminus E} : \tilde{T} \setminus E \to \mathcal{M}^{\circ}_{M, \mathcal{K}}$ is also locally liftable.
			By \cite[3.7. Theorem]{MR338456}, $\mathcal{P}_{T^{\circ}} \circ \mu|_{\tilde{T} \setminus E}$ extends to a holomorphic map $\mathcal{P}_{\tilde{T}} : \tilde{T} \to \mathcal{M}^*_{M, \mathcal{K}}$.
			Set $\mathcal{P}_T = \mathcal{P}_{\tilde{T}} \circ \mu^{-1} : T \dashrightarrow \mathcal{M}^*_{M, \mathcal{K}}$.
			This is a meromorphic map and satisfies $\mathcal{P}_T|_{T^{\circ}} = \mathcal{P}_{T^{\circ}}$.
			
			It suffices to check that $\mathcal{P}_{T^{\circ}}(T^{\circ}) \subset C$.
			Let $T'$ be the Zariski closure of $(\mathcal{P}_{T^{\circ}})^{-1}(C)$ in $T$.
			This is a closed subvariety of $T$.
			Since $\mathcal{P}_{T^{\circ}} \left( c(U) \right) =\gamma(U) \subset C$, we have $c(U) \subset T'$.
			By the definition of $T$, we have $T \subset T'$. 
			Since $T' \subset T$ by definition, we get $T=T'$.
			Thus we obtain $\mathcal{P}_{T^{\circ}}(T^{\circ}) \subset C$.
		\end{proof}
		
		Let $\Gamma$ be the Zariski closure of the graph of $\mathcal{P}_{T^{\circ}}$ in $T \times \mathcal{M}^*_{M, \mathcal{K}}$.
		By the above claim, $\Gamma$ is a closed subvariety of $T \times C$.
		Set $\Gamma^{\circ} = \Gamma \cap \left( T^{\circ} \times \mathcal{M}^{\circ}_{M, \mathcal{K}} \right)$. 
		For each $t \in U$, we have $\left( c(t), \gamma(t) \right) \in \Gamma^{\circ}$.
		Therefore, $\Gamma^{\circ}$ is a non-empty Zariski open subset of $\Gamma$.
		Let $p : \Gamma \to T$ and $q : \Gamma \to C$ be the natural projections.
		Since $p|_{\Gamma^{\circ}} : \Gamma^{\circ} \cong T^{\circ}$ is isomorphic, we have $\Gamma^{\circ} \subset \Gamma \setminus \operatorname{Sing}\Gamma$. \par
		Since $\mathcal{P}_{T^{\circ}}(c(t))= \gamma(t) \in C \hspace{2pt} (t \in U)$, the image of $q : \Gamma \to C$ contains an open subset of $C$.
		Therefore, $q : \Gamma \to C$ is dominant.
		Fix a projective embedding $\Gamma \subset \mathbb{P}^{N'}$.
		There is a hyperplane $H_1 \subset \mathbb{P}^{N'}$ such that $q|_{\Gamma \cap H_1} : \Gamma \cap H_1 \to C$ is dominant and $\dim \left( \Gamma \cap H_1 \right) = \dim \Gamma -1$.
		Repeating this process, we have hyperplanes $H_1, \dots , H_r \subset \mathbb{P}^{N'}$ such that $q|_{\Gamma \cap H_1 \cap \dots \cap H_r } : \Gamma \cap H_1 \cap \dots \cap H_r \to C$ is dominant and $\Gamma \cap H_1 \cap \dots \cap H_r$ is a curve.
		Let $B'$ be an irreducible component of $\Gamma \cap H_1 \cap \dots \cap H_r$ such that $q|_{B'} : B' \to C$ is dominant.
		We denote by $B$ the normalization of $B'$, and denote by $h$ the composition of the normalization morphism $B \to B'$ and the embedding $B' \to \Gamma$.
		Set $\varphi = q \circ h : B \to C$.
		Then $B$ is a smooth projective curve and $\varphi : B \to C$ is a surjective morphism.
		
		Consider the family of $K3^{[2]}$-type manifolds $\pi_B : \mathcal{X}^I|_T \times_T B \to B$ induced from $\pi^I|_T : \left( \mathcal{X}^I|_T, \mathcal{I}|_T \right) \to T$ 
		and the involution $\mathcal{I} \times \operatorname{id}_B : \mathcal{X}^I|_T \times_T B \to \mathcal{X}^I|_T \times_T B$ induced from $\mathcal{I} : \mathcal{X}^I \to \mathcal{X}^{I}$ by $p \circ h :B \to T$.
		Its period map is given by $\mathcal{P}_T \circ p \circ h = q \circ h = \varphi$.
		Set $B^{\circ} =B \cap h^{-1}(\Gamma^{\circ})$.
		Then $B^{\circ}$ is a non-empty Zariski open subset of $B$.\par
		\noindent
		\textbf{Step 3 }
		By \cite[Theorem 13.4]{MR1440306}, there exist a resolution of singularities $\Pi : \mathcal{W} \to \mathcal{X}^I|_T \times_T B$, and a holomorphic involution $\theta : \mathcal{W} \to \mathcal{W}$ such that
		$\mathcal{W}$ is a smooth projective manifold and such that $\Pi : \mathcal{W} \to \mathcal{X}^I|_T \times_T B$ is $\G$-equivariant with respect to the $\G$-actions induced from $\theta$ and $\mathcal{I} \times \operatorname{id}_B$.
		We set $\widetilde{\pi} = \pi_B \circ \Pi : (\mathcal{W}, \theta) \to B$.
		
		By \cite[Theorem 3.35]{MR2289519}, there exits a blowup sequence functor $\mathcal{BP}$ defined on all triples $(X, I, E)$, where $X$ is a smooth scheme of finite type over $\mathbb{C}$,
		$I \subset \mathcal{O}_X$ is an ideal sheaf that is not zero on any irreducible component of $X$, $E$ is a simple normal crossing divisor on $X$,
		satisfying the following properties:
		\begin{enumerate}[ label= \rm{(\arabic*)} ]
			\item In the blowup sequence $\mathcal{BP}(X, I, E)$ =
			$$
				\Pi : X_r \xrightarrow{\pi_{r-1}} X_{r-1} \xrightarrow{\pi_{r-2}} \dots \xrightarrow{\pi_{1}} X_1 \xrightarrow{\pi_{0}} X_0 =X,
			$$
			all centers of blowups are smooth.
			\item The pullback $\Pi^*I \subset \mathcal{O}_{X_r}$ is the ideal sheaf of a simple normal crossing divisor.
			\item $\mathcal{BP}$ commutes with smooth morphism. In particular, for any automorphism $\tau : X \to X$,
			$$
				\mathcal{BP}( X, \tau^*I, \tau^{-1}(E) ) = \tau^* \mathcal{BP}( X, I, E ).
			$$
		\end{enumerate}  
		Here $\tau^* \mathcal{BP}( X, I, E )$ is induced from the pullback 
		$$
			\xymatrix{
				 \Pi' : \hspace{-35pt} &X'_r \ar[r]^{\tau_{r-1}^* \pi_{r-1}} \ar[d]^{\tau_r} &X'_{r-1} \ar[r]^{\tau_{r-2}^* \pi_{r-2}} \ar[d]^{\tau_{r-1}} &\dots \ar[r]^{\tau_{1}^* \pi_{1}} &X'_1 \ar[r]^{\tau^* \pi_{0}} \ar[d]^{\tau_1}  &X_0 =X \ar[d]^{\tau}   \\
				 \Pi : \hspace{-35pt} &X_r \ar[r]^{\pi_{r-1}} &X_{r-1} \ar[r]^{\pi_{r-2}} &\dots \ar[r]^{\pi_{1}} &X_1 \ar[r]^{\pi_{0}} &X_0 =X    
			}
		$$
		by deleting $\tau_i^* \pi_i$ whose center is empty and reindexing the resulting blowup sequence.
		Let $\hat{W}$ be the union of the singular fibers of $\widetilde{\pi}$ and let $I_{\hat{W}}$ be its ideal sheaf.
		Therefore the blowup sequence $\mathcal{BP}( \mathcal{W}, I_{\hat{W}}, \emptyset)$ gives a smooth projective variety $\X$, a projective birational morphism $\Pi' : \X \to \mathcal{W}$, a simple normal crossing divisor $\hat{X} \subset \X$ with $\hat{X} = (\Pi')^*\hat{W}$.
		Moreover, $\mathcal{BP}( \mathcal{W}, I_{\hat{W}}, \emptyset) = \mathcal{BP}( \mathcal{W}, \theta^* I_{\hat{W}}, \emptyset) = \theta^* \mathcal{BP}( \mathcal{W}, I_{\hat{W}}, \emptyset)$ induces an involution $\iota : \X \to \X$ such that the following diagram commutes:
		$$
			\xymatrix{
				\X \ar[r]^{\iota}  \ar[d]_{\Pi'} & \X \ar[d]^{\Pi'}   \\
				\mathcal{W} \ar[r]^{\theta} & \mathcal{W} .   
			}
		$$
		Thus $\Pi' : (\X, \iota) \to (\mathcal{W}, \theta)$ is $\G$-equivariant and the singular fibers of $f := \widetilde{\pi} \circ \Pi'$ are simple normal crossing
		and the data $(B, \phi, \X, \iota, f)$ satisfies the conditions (1), (2), (3), (4) in the statement.
		This completes the proof.
	\end{proof}

	\begin{thm}\label{t3-2-3-1}
		Let $C \subset \mathcal{M}^*_{M, \mathcal{K}}$ be an irreducible projective curve satisfying Assumption \ref{as4-2-3-1}.
		Let $p \in C \cap \bar{\mathscr{D}}_{M^{\perp}}$ be a smooth point of $C$.
		Choose a coordinate $(\D, s)$ on $C$ centered at $p$ such that $\D \cap \bar{\mathscr{D}}_{M^{\perp}} = \{p \}$ and $s(\D)$ is the unit disk.
		There exists a constant $a \in \mathbb{Q}$ such that
		$$
			\tau_{M, \mathcal{K}} (s) = a \log |s|^2 + O(\log \log |s|^{-1}) \quad (s \to 0).
		$$
	\end{thm}
	
	\begin{proof}
		Applying Theorem \ref{t2-3-5} to the flat family $f : (\mathscr{X}, \iota) \to B$ obtained in Proposition \ref{p2-3-1}, we obtain the desired result.
	\end{proof}

\section{Natural involutions}\label{s-n}

	Let $M_0$ be a primitive hyperbolic 2-elementary sublattice of $L_{K3} = U^{\oplus 3} \oplus E_8^{\oplus 2}$.
	Set $r= \rk M_0$, $l= \dim_{\mathbb{Z}_2} (M_0^{\vee}/M_0)$.
	Moreover when $M_0 \ncong U(2) \oplus E_8(2)$, we set
	\begin{align}\label{f-n-1}
		g = 11- \frac{r+l}{2} , \qquad
		k = \frac{r-l}{2}
	\end{align}
	
	Recall that $L_2 = L_{K3} \oplus \mathbb{Z}e$, where $e^2=-2$ and $(e, L_{K3})=0$.
	Set $\widetilde{M}_0=M_0 \oplus \mathbb{Z}e$.
	Then $\widetilde{M}_0$ is an admissible sublattice of $L_2$.
	Let $\mathcal{K} \in \operatorname{KT}(\widetilde{M}_0)$ be a natural K\"ahler-type chamber.
	Then $(X, \iota)=(Y^{[2]}, \sigma^{[2]})$ is a $K3^{[2]}$-type manifold with involution of type $(\widetilde{M}_0, \mathcal{K})$.
	
	Let $f : (\mathscr{X}, \iota) \to S$ be a family of $K3^{[2]}$-type manifolds with involution of type $(\widetilde{M}_0, \mathcal{K})$,
	and let $g :  (\mathscr{Y}, \sigma) \to S$ be a family of 2-elementary K3 surfaces of type $M_0$ such that $X_s = Y_s^{[2]}$ and $ \iota_s = \sigma_s^{[2]}$ for each $s \in S$.
	Let $h_{\xs}$ be an $\iota$-invariant fiberwise K\"ahler metric on $T\xs$ and $h_{\mathscr{Y} /S}$ be a $\sigma$-invariant fiberwise K\"ahler metric on $T\mathscr{Y} /S$.
	Recall that $t=\operatorname{Tr}(\iota^*|_{H^{1,1}(X)})$.
	Recall that the differential form $\omega_{H^{\cdot}(\xss)}$ was defined by $(\ref{al6-2-1-1})$.
	The aim of this section is to show that
	$$
		\omega_{H^{\cdot}(\xss)} = (t+1) c_1(g_* \Omega^1_{\mathscr{Y}^{\sigma}/S}, h_{L^2}),
	$$
	and to prove that the function $\tau_{\widetilde{M}_0, \mathcal{K}, \xs}$ satisfies the same $\partial \bar{\partial}$-equation as the equation satisfied by the invariant $\tau_{M_0}$ of 2-elementary K3 surfaces of type $M_0$. 
	
	If $M_0=E_8(2) \oplus U(2)$, then the fixed locus $X^{\iota}$ of the fiber $(X, \iota)$ is the Enriques surface $Y/\sigma$.
	Therefore we have $\omega_{H^{\cdot}(\mathfrak{X^{\iota}}/S)}=c_1(g_* \Omega^1_{\mathscr{Y}^{\sigma}/S}, h_{L^2})=0$ and this is the trivial case.
	In the rest of this section, we assume that $M_0 \ncong E_8(2) \oplus U(2)$.

\subsection{The Hodge form for natural involutions}
	
	Recall that the direct image sheaves $g_* \Omega^1_{\mathscr{Y}^{\sigma}/S}$, $R^1g_*\mathcal{O}_{\mathscr{Y}^{\sigma}}$, $f_*\Omega^1_{\xss}$, and $R^1f_*\mathcal{O}_{\mathscr{X}^{\iota}}$ are locally free $\mathcal{O}_S$-modules and we may identify them with the corresponding holomorphic vector bundles over $S$.
	By the Hodge decomposition, we have
	$$
		\overline{H^0(Y_s^{\sigma}, \Omega^1_{Y_s^{\sigma}})} = H^1(Y_s^{\sigma}, \mathcal{O}_{Y_s^{\sigma}}) \quad {\it and} \quad 
		\overline{H^0(X_s^{\iota}, \Omega^1_{X_s^{\iota}})} = H^1(X_s^{\iota}, \mathcal{O}_{X_s^{\iota}})
	$$
	for each $s \in S$
	and they induce $\mathbb{C}$-antilinear homomorphisms of complex vector bundles
	$$
		c : g_* \Omega^1_{\mathscr{Y}^{\sigma}/S} \to R^1g_*\mathcal{O}_{\mathscr{Y}^{\sigma}} \quad {\it and} \quad c : f_*\Omega^1_{\xss} \to R^1f_*\mathcal{O}_{\mathscr{X}^{\iota}}.
	$$

	\begin{lem}\label{l-n-1}
		There exist isomorphisms of holomorphic vector bundles on $S$
		$$
			\phi : (g_* \Omega^1_{\mathscr{Y}^{\sigma}/S})^{\oplus k+1} \to f_*\Omega^1_{\xss} 
			\qquad {\it and} \qquad \psi : (R^1g_*\mathcal{O}_{\mathscr{Y}^{\sigma}})^{\oplus k+1} \to R^1f_*\mathcal{O}_{\mathscr{X}^{\iota}} 
		$$
		such that the following diagram commutes:
		$$
			\xymatrix{
				(g_* \Omega^1_{\mathscr{Y}^{\sigma}/S})^{\oplus k+1} \ar[r]^{\phi}  \ar[d]^{c} & f_*\Omega^1_{\xss} \ar[d]^c   \\
				(R^1g_*\mathcal{O}_{\mathscr{Y}^{\sigma}})^{\oplus k+1} \ar[r]^{\psi} & R^1f_*\mathcal{O}_{\mathscr{X}^{\iota}} .  
			}
		$$
	\end{lem}
	
	\begin{proof}
		Let $(Y, \sigma)$ be a 2-elementary K3 surface of type $M_0$ and set $(X, \iota)=(Y^{[2]}, \sigma^{[2]})$.
		By \cite[Theorem~4.2.2.]{MR633160}, the fixed locus of a 2-elementary K3 surface $(Y, \sigma)$ of type $M_0$ is
		\begin{equation}\label{al-n-3}
			Y^{\sigma} = \left\{
				\begin{aligned}
				\begin{split}
					&\emptyset \quad &(M_0 \cong E_8(2) \oplus U(2)) ,\\
					&E \sqcup F \quad &(M_0 \cong E_8(2) \oplus U) ,\\
					&C \sqcup E_1 \sqcup E_2 \sqcup \dots \sqcup E_k \quad &(otherwise), 
				\end{split} \end{aligned} \right. 
		\end{equation}
		where $E,F$ are elliptic curves on $Y$, $C$ is a curve of genus $g$ on $Y$, and $E_i$ $(i=1, \dots , k)$ are $(-2)$-curves on $Y$.
		The fixed locus $X^{\iota}$ of $(X, \iota)$ is given by
		\begin{align}\label{al-n-2}
			X^{\iota} = \left\{
				\begin{aligned}
				\begin{split}
					&Y/\sigma  \qquad\qquad\qquad\qquad\qquad\qquad\qquad\qquad\qquad (M_0 \cong E_8(2) \oplus U(2)), \\
					&E^{(2)} \sqcup F^{(2)} \sqcup (E \times F) \sqcup Y/\sigma  \qquad\qquad\qquad\qquad (M_0 \cong E_8(2) \oplus U), \\
					&C^{(2)} \sqcup (\sqcup_i E_i^{(2)}) \sqcup (\sqcup_i C \times E_i) \sqcup (\sqcup_{i < j} E_i \times E_j) \sqcup Y/\sigma \quad(otherwise),
				\end{split} \end{aligned} \right. 
		\end{align}
		where $E^{(2)}$, $F^{(2)}$, $C^{(2)}$ and ${E_i}^{(2)}$ are the symmetric product of the curves $E$, $F$, $C$ and $E_i$, respectively.
		Note that $E^{(2)}$, $F^{(2)}$ are ruled surfaces and that $Y/\sigma$, $E_i^{(2)}$, $E_i \times E_j$ are rational surfaces.
		
		If $M_0 \ncong E_8(2) \oplus U, E_8(2) \oplus U(2)$, then we have by (\ref{al-n-3}), (\ref{al-n-2})
		\begin{align}\label{al6-3-2-3}
		\begin{aligned}
			H^0(Y_s^{\sigma}, \Omega^1_{Y_s^{\sigma}}) &= H^0(C, \Omega^1_{C}), \\ 
			H^0(X_s^{\iota}, \Omega^1_{X_s^{\iota}}) &= H^0(C^{(2)}, \Omega^1_{C^{(2)}}) \oplus \oplus_{i} H^0(C \times E_i, \Omega^1_{C \times E_i}).
		\end{aligned}
		\end{align}
		Let $p, q : C \times C \to C$ be the first and second projection.
		For a holomorphic 1-form $\alpha$ on $C$, 
		$$
			p^* \alpha + q^* \alpha
		$$
		induces a holomorphic 1-form $\beta$ on $C^{(2)}$.
		By \cite[(6.3)]{MR151460}, the following homomorphism is an isomorphism
		\begin{align}\label{al6-3-2-4}
			H^0(C, \Omega^1_{C}) \cong H^0(C^{(2)}, \Omega^1_{C^{(2)}}), \quad \alpha \mapsto \beta.
		\end{align}
		By the K\"unneth formula, we have
		\begin{align}\label{al6-3-2-5}
			H^0(C \times E_i, \Omega^1_{C \times E_i}) \cong H^0(C, \Omega^1_{C}).
		\end{align}
		By (\ref{al6-3-2-3}), (\ref{al6-3-2-4}) and (\ref{al6-3-2-5}), we have the canonical isomorphism
		\begin{align}\label{al6-3-2-6}
		\begin{aligned}
			H^0(X_s^{\iota}, \Omega^1_{X_s^{\iota}}) &= H^0(C^{(2)}, \Omega^1_{C^{(2)}}) \oplus \oplus_{i} H^0(C \times E_i, \Omega^1_{C \times E_i}) \\
			&\cong H^0(C, \Omega^1_{C})) \oplus_i H^0(C, \Omega^1_{C}) \\
			&= H^0(Y_s^{\sigma}, \Omega^1_{Y_s^{\sigma}})^{\oplus k+1}. 
		\end{aligned}
		\end{align}
		
		If $M_0=E_8(2) \oplus U$, then we have by (\ref{al-n-3}), (\ref{al-n-2})
		\begin{align*}
			H^0(Y^{\sigma}, \Omega^1_{Y^{\sigma}}) &= H^0(E, \Omega^1_{E}) \oplus H^0(F, \Omega^1_{F}), \\
			H^0(X^{\iota}, \Omega^1_{X^{\iota}}) &= H^0(E^{(2)}, \Omega^1_{E^{(2)}}) \oplus H^0(F^{(2)}, \Omega^1_{F^{(2)}}) \oplus H^0(E \times F, \Omega^1_{E \times F}).
		\end{align*}
		By (\ref{al6-3-2-4}) and by the K\"unneth formula, we have the canonical isomorphism
		\begin{align}\label{al6-3-2-7}
		\begin{aligned}
			H^0(X_s^{\iota}, \Omega^1_{X_s^{\iota}}) &= H^0(E^{(2)}, \Omega^1_{E^{(2)}}) \oplus H^0(F^{(2)}, \Omega^1_{F^{(2)}}) \oplus H^0(E \times F, \Omega^1_{E \times F}) \\
			&\cong H^0(E, \Omega^1_{E}) \oplus H^0(F, \Omega^1_{F}) \oplus H^0(E, \Omega^1_{E}) \oplus H^0(F, \Omega^1_{F})\\ 
			&= H^0(Y_s^{\sigma}, \Omega^1_{Y_s^{\sigma}})^{\oplus 2}. 
		\end{aligned}
		\end{align}
		
		By (\ref{al6-3-2-6}), (\ref{al6-3-2-7}), we have an isomorphism of holomorphic vector bundles
		$$
			\phi : (g_* \Omega^1_{\mathscr{Y}^{\sigma}/S})^{\oplus k+1} \cong f_*\Omega^1_{\xss}.
		$$
		Similarly, we have an isomorphism of holomorphic vector bundles
		$$
			\psi : (R^1g_*\mathcal{O}_{\mathscr{Y}^{\sigma}})^{\oplus k+1} \cong R^1f_*\mathcal{O}_{\mathscr{X}^{\iota}}.
		$$
		By their constructions, the following diagram commutes:
		$$
			\xymatrix{
				H^0(Y_s^{\sigma}, \Omega^1_{Y_s^{\sigma}})^{\oplus k+1} \ar[r]^{\phi_s}  \ar[d]^{c} & H^0(X_s^{\iota}, \Omega^1_{X_s^{\iota}}) \ar[d]^c   \\
				{H^1(Y_s^{\sigma}, \mathcal{O}_{Y_s^{\sigma}})}^{\oplus k+1} \ar[r]^{\psi_s} & H^1(X_s^{\iota}, \mathcal{O}_{X_s^{\iota}})   .
			}
		$$
		This completes the proof.
	\end{proof}
	
	Let $\alpha_1, \dots, \alpha_m$ be a local holomorphic frame of $(g_* \Omega^1_{\mathscr{Y}^{\sigma}/S})^{\oplus k+1}$,
	and let $\beta_1, \dots, \beta_m$ be a local holomorphic frame of $(R^1g_* \mathcal{O}_{\mathscr{Y}^{\sigma}})^{\oplus k+1}$.
	We define real-valued functions $\tilde{c}_1(\phi)$ and $\tilde{c}_1(\psi)$ by
	\begin{align}\label{al4-3-2-1}
		\tilde{c}_1(\phi) =-\log \frac{\det (\langle \phi( \alpha_i), \phi( \alpha_j ) \rangle_{L^2}) }{\det (\langle  \alpha_i,  \alpha_j  \rangle_{L^2})}, \quad
		\tilde{c}_1(\psi) =-\log \frac{\det (\langle \psi( \beta_i), \psi( \beta_j ) \rangle_{L^2}) }{\det (\langle  \beta_i,  \beta_j  \rangle_{L^2})}
	\end{align}
	Note that $\tilde{c}_1(\phi)$ and $\tilde{c}_1(\psi)$ are independent of the choice of holomorphic frames and they are globally defined on $S$. 
	
	\begin{lem}\label{l-n-2}
		We have $\tilde{c}_1(\phi) = \tilde{c}_1(\psi)$.
	\end{lem}
	
	\begin{proof}
		Note that $\bar{\alpha}_1, \dots, \bar{\alpha}_m$ is smooth frame of $R^1g_* \mathcal{O}_{\mathscr{Y}^{\sigma}}^{\oplus k+1}$ which is not holomorphic,
		and that $\psi(\bar{\alpha}_1)=\overline{\phi(\alpha_1)}, \dots, \psi(\bar{\alpha}_m)=\overline{\phi(\alpha_m)}$ is smooth frame of $R^1f_*\mathcal{O}_{\mathscr{X}^{\iota}}^{\oplus k+1}$ which is not holomorphic.
		Let $C =(c_{ij})$ be a $m \times m$-matrix defined by $\bar{\alpha}_i = \sum_j c_{ij} \beta_j$. 
		Then we have $\overline{\phi(\alpha_i)}=\psi(\bar{\alpha}_i)=\psi(\sum_j c_{ij} \beta_j)=\sum_j c_{ij} \psi( \beta_j)$,
		and we have
		\begin{align}\label{al4-3-2-2}
		\begin{aligned}
			 \frac{\det (\langle \overline{\phi( \alpha_i)}, \overline{\phi( \alpha_j )} \rangle_{L^2}) }{\det (\langle  \bar{\alpha_i}, \bar{ \alpha_j}  \rangle_{L^2})}
			&=  \frac{\det \left( C \cdot (\langle \psi( \beta_i), \psi( \beta_j ) \rangle_{L^2}) \cdot {}^t\overline{C} \right)}{\det \left( C \cdot (\langle  \beta_i,  \beta_j  \rangle_{L^2}) \cdot {}^t\overline{C} \right)} \\
			&=  \frac{\det (\langle \psi( \beta_i), \psi( \beta_j ) \rangle_{L^2}) }{\det (\langle  \beta_i,  \beta_j  \rangle_{L^2})}.
		\end{aligned}
		\end{align}
		Let $\gamma$ and $\delta$ be smooth $(1,0)$-forms.
		Then $\delta$ is primitive, and by \cite[Proposition~6.29]{MR2451566}, we have
		\begin{align}\label{al4-3-2-3}
			\langle  \bar{\gamma}, \bar{ \delta}  \rangle_{L^2} = \overline{ \langle  \gamma, \delta  \rangle_{L^2} } =\langle  \delta, \gamma  \rangle_{L^2} .
		\end{align}
		The same equality holds where $\gamma$ and $\delta$ are smooth $(0,1)$-forms.
		By (\ref{al4-3-2-1}), (\ref{al4-3-2-2}) and (\ref{al4-3-2-3}), we have
		\begin{align*}
			\tilde{c}_1(\phi) &=-\log \frac{\det (\langle \phi( \alpha_i), \phi( \alpha_j ) \rangle_{L^2}) }{\det (\langle  \alpha_i,  \alpha_j  \rangle_{L^2})} \\
			&=-\log \left\{ \frac{\det (\langle \phi( \alpha_i), \phi( \alpha_j ) \rangle_{L^2}) }{\det (\langle \overline{\phi( \alpha_i)}, \overline{\phi( \alpha_j )} \rangle_{L^2})}  \frac{\det (\langle \overline{\phi( \alpha_i)}, \overline{\phi( \alpha_j )} \rangle_{L^2})}{\det (\langle  \bar{\alpha_i}, \bar{ \alpha_j}  \rangle_{L^2})} \frac{\det (\langle  \bar{\alpha_i}, \bar{ \alpha_j}  \rangle_{L^2})}{\det (\langle  \alpha_i,  \alpha_j  \rangle_{L^2})} \right\} \\
			&=-\log \left\{ \frac{\det (\langle \phi( \alpha_i), \phi( \alpha_j ) \rangle_{L^2}) }{ \det {}^t (\langle \phi( \alpha_i), \phi( \alpha_j ) \rangle_{L^2}) }   \frac{\det (\langle \psi( \beta_i), \psi( \beta_j ) \rangle_{L^2}) }{\det (\langle  \beta_i,  \beta_j  \rangle_{L^2})}  \frac{ \det {}^t (\langle  \alpha_i,  \alpha_j  \rangle_{L^2}) }{ \det (\langle  \alpha_i,  \alpha_j  \rangle_{L^2}) } \right\} \\
			&= \tilde{c}_1(\psi) ,
		\end{align*}
		which completes the proof.
	\end{proof}

	\begin{prop}\label{l-n-3}
		The following identity holds:
		$$
			c_1(f_*\Omega^1_{\xss}, h_{L^2}) -c_1(R^1f_*\mathcal{O}_{\mathscr{X}^{\iota}}, h_{L^2}) =2(k+1) c_1(g_* \Omega^1_{\mathscr{Y}^{\sigma}/S}, h_{L^2}).
		$$
	\end{prop}
	
	\begin{proof}
		By the definition of $\tilde{c}_1(\phi)$ and $\tilde{c}_1(\psi)$, we have
		\begin{align}\label{al-n-31}
		\begin{aligned}
			c_1(f_*\Omega^1_{\xss}, h_{L^2}) -c_1((g_* \Omega^1_{\mathscr{Y}^{\sigma}/S})^{\oplus k+1}, h_{L^2}) =dd^c \tilde{c}_1(\phi) ,\\
			c_1(R^1f_*\mathcal{O}_{\mathscr{X}^{\iota}}, h_{L^2}) -c_1( (R^1g_*\mathcal{O}_{\mathscr{Y}^{\sigma}})^{\oplus k+1}, h_{L^2}) =dd^c \tilde{c}_1(\psi).
		\end{aligned}
		\end{align}
		By the Serre duality, we have
		\begin{align}\label{al-n-32}
			c_1(R^1g_*\mathcal{O}_{\mathscr{Y}^{\sigma}}, h_{L^2}) =- c_1(g_* \Omega^1_{\mathscr{Y}^{\sigma}/S}, h_{L^2}).
		\end{align}
		By (\ref{al-n-31}) and (\ref{al-n-32}) and Lemma \ref{l-n-2}, we have the desired result.
	\end{proof}

	\begin{lem}\label{l4-n-2-8}
		Let $(E, h)$ be a hermitian holomorphic vector bundle on $S$ of rank $g \geqq 1$.
		Then we have
		$$
			c_1( \wedge^2 E, h ) = (g-1) c_1( E, h).
		$$ 
	\end{lem}
	
	\begin{proof}
		It suffices to show that there exists a canonical isometry
		$$
			( \det \wedge^2 E, h) \cong ( (\det E)^{\otimes g-1}, h).
		$$
		Let $e_1, \dots, e_g$ be a local holomorphic frame of $E$.
		Set $e_{ij} = e_i \wedge e_j$.
		Then $\{ e_{ij}; i<j \}$ forms a local holomorphic frame of $\wedge^2 E$
		and 
		\begin{align}\label{al4-n-2-9}
			e_{12} \wedge e_{13} \wedge \dots \wedge e_{1g} \wedge e_{23} \wedge \dots \wedge e_{(g-1), g}
		\end{align}
		is a nowhere vanishing local holomorphic section of $\det \wedge^2 E$.
		On the other hand, 
		\begin{align}\label{al4-n-2-10}
			(e_1 \wedge \dots \wedge e_g) \otimes \dots \otimes (e_1 \wedge \dots \wedge e_g)
		\end{align}
		is a nowhere vanishing local holomorphic section of $(\det E)^{\otimes g-1}$.
		
		Let $e'_1, \dots, e'_g$ be another local holomorphic frame of $E$
		and let $A=(a_{ij})$ be the $g \times g$-matrix defined by
		$$
			e'_i = \sum_j a_{ij} e_j.
		$$
		Then we have
		$$
			e'_{12} \wedge e'_{13} \wedge \dots \wedge e'_{1g} \wedge e'_{23} \wedge \dots \wedge e'_{(g-1), g} = (\det A)^{g-1} e_{12} \wedge e_{13} \wedge \dots \wedge e_{1g} \wedge e_{23} \wedge \dots \wedge e_{(g-1), g}
		$$
		and
		$$
			(e'_1 \wedge \dots \wedge e'_g) \otimes \dots \otimes (e'_1 \wedge \dots \wedge e'_g) =( \det A)^{g-1} (e_1 \wedge \dots \wedge e_g) \otimes \dots \otimes (e_1 \wedge \dots \wedge e_g).
		$$
		Therefore the correspondence from the section (\ref{al4-n-2-9}) to the section (\ref{al4-n-2-10}) is independent of the choice of $e_1, \dots, e_g$
		and induces an isomorphism of holomorphic line bundles
		\begin{align}\label{al4-n-2-11}
			\det \wedge^2 E \cong  (\det E)^{\otimes g-1}.
		\end{align}
		Moreover, if $e_1(s), \dots, e_g(s)$ is an orthonormal basis of $(E_s,h_s)$ at $s \in S$,
		then the norm of the sections (\ref{al4-n-2-9}) and (\ref{al4-n-2-10}) at $s$ are $1$.
		Since $s \in S$ is arbitrary, the isomorphism (\ref{al4-n-2-11}) is an isometry.
	\end{proof}

	\begin{prop}\label{p4-n-1}
		The following identity holds:
		$$
			c_1(f_*\Omega^2_{\xss}, h_{L^2}) =(g-1) c_1(g_* \Omega^1_{\mathscr{Y}^{\sigma}/S}, h_{L^2}).
		$$
	\end{prop}
	
	\begin{proof}
		If $M_0=E_8(2) \oplus U$, then we have by (\ref{al-n-3}), (\ref{al-n-2})
		$$
			H^0(Y_s^{\sigma}, \Omega^1_{Y_s^{\sigma}}) = H^0(E, \Omega^1_{E}) \oplus H^0(F, \Omega^1_{F}), \quad H^0(X_s^{\iota}, \Omega^2_{X_s^{\iota}}) = H^0(E \times F, \Omega^2_{E \times F}),
		$$
		where $E$ and $F$ are elliptic curves.
		By the K\"unneth formula, we have the canonical isomorphism
		\begin{align}\label{al6-3-2-1}
			H^0(X_s^{\iota}, \Omega^2_{X_s^{\iota}}) =  H^0(E \times F, \Omega^2_{E \times F}) \cong H^0(E, \Omega^1_{E}) \otimes H^0(F, \Omega^1_{F}) = \wedge^2 H^0(Y_s^{\sigma}, \Omega^1_{Y_s^{\sigma}}). 
		\end{align}
		
		If $M_0 \neq E_8(2) \oplus U, E_8(2) \oplus U(2)$, then we have by (\ref{al-n-3}), (\ref{al-n-2})
		$$
			H^0(Y_s^{\sigma}, \Omega^1_{Y_s^{\sigma}}) = H^0(C, \Omega^1_{C}), \quad H^0(X_s^{\iota}, \Omega^2_{X_s^{\iota}}) = H^0(C^{(2)}, \Omega^2_{C^{(2)}}),
		$$
		where $C$ is a curve of genus $g$.
		Let $p, q : C \times C \to C$ be the first and second projection.
		For a holomorphic 1-form $\alpha, \beta$ on $C$, 
		$$
			p^* \alpha \wedge q^* \beta -p^* \beta \wedge q^* \alpha
		$$
		induces a holomorphic 2-form $\gamma$ on $C^{(2)}$.
		By \cite[(6.3)]{MR151460}, the following homomorphism is an isomorphism
		$$
			\wedge^2 H^0(C, \Omega^1_{C}) \cong H^0(C^{(2)}, \Omega^2_{C^{(2)}}), \quad \alpha \wedge \beta \mapsto \gamma.
		$$
		Therefore, we have the canonical isomorphism
		\begin{align}\label{al6-3-2-2}
			H^0(X_s^{\iota}, \Omega^2_{X_s^{\iota}}) = H^0(C^{(2)}, \Omega^2_{C^{(2)}}) \cong \wedge^2 H^0(C, \Omega^1_{C})) = \wedge^2 H^0(Y_s^{\sigma}, \Omega^1_{Y_s^{\sigma}}). 
		\end{align}
		
		By (\ref{al6-3-2-1}), (\ref{al6-3-2-2}), we have the isomorphism of holomorphic vector bundles
		$$
			\zeta : \wedge^2 g_* \Omega^1_{\mathscr{Y}^{\sigma}/S}  \cong f_*\Omega^2_{\xss}.
		$$
		Similarly, we have an isomorphism of holomorphic vector bundles
		$$
			\xi : \wedge^2 R^1g_*\mathcal{O}_{\mathscr{Y}^{\sigma}} \cong R^2f_*\mathcal{O}_{\mathscr{X}^{\iota}}.
		$$
		By \cite[(6.3)]{MR151460} and the K\"unneth formula, the following diagram commutes:
		\begin{align}\label{al4-n-2-2}
			\xymatrix{
				\wedge^2 g_* \Omega^1_{\mathscr{Y}^{\sigma}/S} \ar[r]^{\zeta}  \ar[d]^{c} & f_*\Omega^2_{\xss} \ar[d]^c   \\
				\wedge^2 R^1g_*\mathcal{O}_{\mathscr{Y}^{\sigma}} \ar[r]^{\xi} & R^2f_*\mathcal{O}_{\mathscr{X}^{\iota}}   
			}
		\end{align}
		
		Let $\gamma_1, \dots, \gamma_{m'}$ be a local holomorphic frame of $\wedge^2 g_* \Omega^1_{\mathscr{Y}^{\sigma}/S}$
		and let $\delta_1, \dots, \delta_{m'}$ be a local holomorphic frame of $\wedge^2 R^1g_*\mathcal{O}_{\mathscr{Y}^{\sigma}}$.
		We define real-valued functions $\tilde{c}_1(\zeta)$ and $\tilde{c}_1(\xi)$ by
		\begin{align}\label{al4-n-2-3}
			\tilde{c}_1(\zeta) =-\log \frac{\det (\langle \zeta( \gamma_i), \zeta( \gamma_j ) \rangle_{L^2}) }{\det (\langle  \gamma_i,  \gamma_j  \rangle_{L^2})}, \quad
			\tilde{c}_1(\xi) =-\log \frac{\det (\langle \xi( \delta_i), \xi( \delta_j ) \rangle_{L^2}) }{\det (\langle  \delta_i,  \delta_j  \rangle_{L^2})}
		\end{align}
		Note that $\tilde{c}_1(\zeta)$ and $\tilde{c}_1(\xi)$ are independent of the choice of holomorphic frames and they are globally defined on $S$. 
		
		Let $C =(c_{ij})$ be a $m \times m$-matrix defined by $\bar{\gamma}_i = \sum_j c_{ij} \delta_j$. 
		By (\ref{al4-n-2-2}), we have
		\begin{align}\label{al4-n-2-4}
		\begin{aligned}
			 \frac{\det (\langle \overline{\zeta( \gamma_i)}, \overline{\zeta( \gamma_j )} \rangle_{L^2}) }{\det (\langle  \bar{\gamma_i}, \bar{ \gamma_j}  \rangle_{L^2})}
			&=  \frac{\det \left( C \cdot (\langle \xi( \delta_i), \xi( \delta_j ) \rangle_{L^2}) \cdot {}^t\overline{C} \right)}{\det \left( C \cdot (\langle  \delta_i,  \delta_j  \rangle_{L^2}) \cdot {}^t\overline{C} \right)} \\
			&=  \frac{\det (\langle \xi( \delta_i), \xi( \delta_j ) \rangle_{L^2}) }{\det (\langle  \delta_i,  \delta_j  \rangle_{L^2})}.
		\end{aligned}
		\end{align}
		By (\ref{al4-n-2-3}) and (\ref{al4-n-2-4}), we have
		\begin{align}\label{al4-n-2-5}
		\begin{aligned}
			\tilde{c}_1(\zeta) &=-\log \frac{  \det (\langle \zeta( \gamma_i), \zeta( \gamma_j ) \rangle_{L^2})  }{  \det (\langle  \gamma_i,  \gamma_j  \rangle_{L^2})  }  \\
			&=-\log \left\{ \frac{   \det (\langle \zeta( \gamma_i), \zeta( \gamma_j ) \rangle_{L^2})  }{  \det (\langle \overline{\zeta( \gamma_i)}, \overline{\zeta( \gamma_j )} \rangle_{L^2})  }  
			\frac{  \det (\langle \overline{\zeta( \gamma_i)}, \overline{\zeta( \gamma_j )} \rangle_{L^2})  }{  \det (\langle  \bar{\gamma_i}, \bar{ \gamma_j}  \rangle_{L^2})  } 
			\frac{  \det (\langle  \bar{\gamma_i}, \bar{ \gamma_j}  \rangle_{L^2})  }{  \det (\langle  \gamma_i,  \gamma_j  \rangle_{L^2})  } \right\} \\
			&=-\log \left\{ \frac{   \det (\langle \zeta( \gamma_i), \zeta( \gamma_j ) \rangle_{L^2})  }{  \det {}^t (\langle \zeta( \gamma_i), \zeta( \gamma_j ) \rangle_{L^2})  }   
			\frac{\det (\langle \xi( \delta_i), \xi( \delta_j ) \rangle_{L^2}) }{\det (\langle  \delta_i,  \delta_j  \rangle_{L^2})}
			\frac{ \det {}^t (\langle  \gamma_i,  \gamma_j  \rangle_{L^2}) }{ \det (\langle  \gamma_i,  \gamma_j  \rangle_{L^2}) } \right\} \\
			&= \tilde{c}_1(\xi) .
		\end{aligned}
		\end{align}
		By the definition of $\tilde{c}_1(\zeta)$ and $\tilde{c}_1(\xi)$, we have
		\begin{align}\label{al4-n-2-6}
		\begin{aligned}
			c_1(f_*\Omega^2_{\xss}, h_{L^2}) -c_1( \wedge^2 g_* \Omega^1_{\mathscr{Y}^{\sigma}/S}, h_{L^2}) =dd^c \tilde{c}_1(\zeta) ,\\
			c_1(R^2f_*\mathcal{O}_{\mathscr{X}^{\iota}}, h_{L^2}) -c_1( \wedge^2 R^1g_*\mathcal{O}_{\mathscr{Y}^{\sigma}}, h_{L^2}) =dd^c \tilde{c}_1(\xi).
		\end{aligned}
		\end{align}
		By Lemma \ref{l4-n-2-8} and by the Serre duality, we have
		\begin{align}\label{al4-n-2-7}
		\begin{aligned}
			c_1(\wedge^2 g_* \Omega^1_{\mathscr{Y}^{\sigma}/S}, h_{L^2}) &= (g-1) c_1(g_* \Omega^1_{\mathscr{Y}^{\sigma}/S}, h_{L^2}) ,\\
			c_1( \wedge^2 R^1g_*\mathcal{O}_{\mathscr{Y}^{\sigma}}, h_{L^2}) &= (g-1) c_1(R^1g_*\mathcal{O}_{\mathscr{Y}^{\sigma}}, h_{L^2}) =-(g-1) c_1(g_* \Omega^1_{\mathscr{Y}^{\sigma}/S}, h_{L^2}).
		\end{aligned}
		\end{align}
		By (\ref{al4-n-2-5}), (\ref{al4-n-2-6}) and (\ref{al4-n-2-7}), we have
		\begin{align*}
			2 c_1(f_*\Omega^2_{\xss}, h_{L^2}) =c_1(f_*\Omega^2_{\xss}, h_{L^2}) -c_1(R^2f_*\mathcal{O}_{\mathscr{X}^{\iota}}, h_{L^2})=2(g-1) c_1(g_* \Omega^1_{\mathscr{Y}^{\sigma}/S}, h_{L^2}),
		\end{align*}
		which completes the proof.
	\end{proof}
	
	\begin{cor}\label{c6-3-2-8}
		The following identity holds:
		$$
			\omega_{H^{\cdot}(\xss)} = (t+1) c_1(g_* \Omega^1_{\mathscr{Y}^{\sigma}/S}, h_{L^2}).
		$$
	\end{cor}
	
	\begin{proof}
		By (\ref{al6-2-1-1}) and by Propositions \ref{l-n-3} and \ref{p4-n-1}, we have
		\begin{align}\label{al4-n-2-12}
			\omega_{H^{\cdot}(\xss)} = 2(k-g+2)c_1(g_* \Omega^1_{\mathscr{Y}^{\sigma}/S}, h_{L^2}).
		\end{align}
		Let $(Y, \sigma)$ be a 2-elementary K3 surface of type $M_0$ and set $(X, \iota)=(Y^{[2]}, \sigma^{[2]})$.
		By \cite[3.1]{MR2805992}, we have 
		$
			e(Y^{\sigma}) =t-1 .
		$
		Since the fixed locus $Y^{\sigma}$ of $(Y, \sigma)$ is given by (\ref{al-n-3}), we have 
		\begin{align}\label{al-n-41}
			2(k-g+2) = t+1.
		\end{align}
		By (\ref{al4-n-2-12}) and (\ref{al-n-41}), we have the desired result.
	\end{proof}

\subsection{An invariant of $K3^{[2]}$-type manifolds with natural involution}\label{ss-4-6}
		
	Recall that the period domain for 2-elementary K3 surfaces of type $M_0$ is defined by
	$$
		\Omega_{M_0^{\perp}} =\{ [\eta] \in \mathbb{P}(M_0^{\perp} \otimes \mathbb{C} ) ; (\eta, \eta)=0, (\eta, \bar{\eta})>0 \} .
	$$
	and 
	$$
		\mathcal{M}_{M_0} = \Omega_{M_0^{\perp}} / O(M_0^{\perp}) .
	$$
	This is a quasi-projective variety.
	(See \cite[Theorems 10.4 and 10.11]{MR216035})
	
	The period of a 2-elementary K3 surface $(Y, \sigma)$ is defined by
	$$
		\bar{\pi}_{M_0}(Y, \sigma) = [\alpha(H^{2,0}(Y))] \in \mathcal{M}_{M_0} ,
	$$
	where $\alpha$ is a marking of $(Y, \sigma)$.
	For the family $g : (\mathscr{Y}, \sigma) \to S$ of 2-elementary K3 surfaces of type $M_0$, we define the period map by
	$$
		\bar{\pi}_{M_0}: S \to \mathcal{M}_{M_0}, \qquad s \mapsto \bar{\pi}_{M_0}(Y_s, \sigma_s) .
	$$ 
	
	Fix a vector $l_0 \in M^{\perp}_{0,\mathbb{R}}$ satisfying $(l_0, l_0) \geqq 0$.
	Since $M_0$ is hyperbolic and hence $\operatorname{sign}M_0^{\perp} = (2, 20-r(M_0))$,
	$\Omega_{M_0^{\perp}}$ consists of two connected components, both of which are isomorphic to a bounded symmetric domain of type IV of dimension $20 - r(M_0)$.
	The Bergman metric $\omega_{M_0^{\perp}}$ on $\Omega_{M_0^{\perp}}$ is defined by
	$$
		\omega_{M_0^{\perp}} ([\eta]) = -dd^c \log B_{M_0^{\perp}} ([\eta]) \quad ([\eta] \in \Omega_{M_0^{\perp}})
	$$
	where
	$$
		B_{M_0^{\perp}} ([\eta]) = \frac{(\eta, \bar{\eta})}{|(\eta, l_0)|^2}  \quad ([\eta] \in \Omega_{M_0^{\perp}}) .
	$$
	Since $\omega_{M_0^{\perp}}$ is $O(M_0^{\perp})$-invariant, it induces a K\"ahler form $\omega_{\mathcal{M}_{M_0}}$ on the orbifold $\mathcal{M}_{M_0} = \Omega_{M_0^{\perp}} / O(M_0^{\perp})$.
	Since the period maps of $f : (\mathscr{X}, \iota) \to S$ and $g : (\mathscr{Y}, \sigma) \to S$ coincide, 
	we have
	\begin{align}\label{al38}
		2 \bar{\pi}_{M_0}^* \omega_{\mathcal{M}_{M_0}} = 2 P^*_{M, \mathcal{K}} \omega_{\mathcal{M}_{M, \mathcal{K}}} = c_1(f_*K_{\xs}, h_{L^2}).
	\end{align}
	
	Let $Sp(2g, \mathbb{Z}) $ be the Siegel modular group.
	We define the Siegel upper half space $\mathfrak{S}_g$ of degree $g$ by
	$$
		\mathfrak{S}_g = \left\{ \tau \in M_g(\mathbb{C}) ; ^t\!\tau = \tau , \operatorname{Im} \tau >0 \right\} ,
	$$
	and define the Siegel modular variety $\mathcal{A}_g$ of degree $g$ by
	$$
		\mathcal{A}_g = \mathfrak{S}_g / Sp(2g, \mathbb{Z}) ,
	$$
	where the $Sp(2g, \mathbb{Z})$-action on $\mathfrak{S}_g$ is given by
	$$
		\gamma \cdot \tau = (A \tau +B)(C \tau +D)^{-1} ,\quad \tau \in \mathfrak{S}_g, \gamma=
		\begin{pmatrix}
			A & B \\
			C & D \\
		\end{pmatrix}
		\in Sp(2g, \mathbb{Z})  .
	$$
		
	The period map for the family of curves $g: \mathscr{Y}^{\sigma} \to S$ is denoted by
	$$
		J_{M_0, \mathscr{Y}^{\sigma}/S} : S \to \mathcal{A}_g, \qquad s \mapsto [\Omega (Y_s^{\sigma})],
	$$	
	where $\Omega (Y_s^{\sigma})$ is the period matrix of $Y_s^{\sigma}$.
	The Bergman metric $\omega_{\mathfrak{S}_g}$ on $\mathfrak{S}_g$ is defined by
	$$
		\omega_{\mathfrak{S}_g}(\tau) = -dd^c \log \det \operatorname{Im} \tau , \qquad \tau \in \mathfrak{S}_g .
	$$
	Since $\omega_{\mathfrak{S}_g}$ is $Sp(2g, \mathbb{Z})$-invariant, it induces a K\"ahler form $\omega_{\mathcal{A}_g}$ on the Siegel modular variety $\mathcal{A}_g = \mathfrak{S}_g / Sp(2g, \mathbb{Z})$.
	We have
	\begin{align}\label{al39}
		c_1(g_* \Omega^1_{\mathscr{Y}^{\sigma}/S}, h_{L^2}) = -dd^c \log \det \langle \psi_i , \psi_j \rangle_{L^2} = J_{M_0, \mathscr{Y}^{\sigma}/S}^* \omega_{\mathcal{A}_g} ,
	\end{align} 
	where $\psi_1, \dots, \psi_g$ is a local frame of $g_*\Omega^1_{ \mathscr{Y}^{\sigma}/S}$.
	
	By the formulas (\ref{f-n-1}) and (\ref{al-n-41}), we have
	\begin{align}\label{al40}
		2r-19 = t ,
	\end{align} 
	where $r=r(M_0)$ is the rank of $M_0$.
	
	\begin{thm}\label{t-4-1}
		Let $g: (\mathscr{Y}, \sigma) \to S$ be a family of 2-elementary K3 surfaces of type $M_0$ 
		and let $f : (\mathscr{X}, \iota) \to S$ be the family of $K3^{[2]}$-type manifolds with antisymplectic involution of type $(\widetilde{M}_0, \mathcal{K})$ induced from $g: (\mathscr{Y}, \sigma) \to S$.
		Namely, 
		$$
			X_s = Y_s^{[2]}, \qquad \iota_s = \sigma_s^{[2]}
		$$
		for all $s \in S$.
		Then the following equality holds:
		$$
			-dd^c \log \tau_{\widetilde{M}_0, \mathcal{K}, \xs} = 2(r-9) \left\{ \frac{r-6}{4} \bar{\pi}_{M_0}^* \omega_{\mathcal{M}_{M_0}} + J_{M_0, \mathscr{Y}^{\sigma}/S}^* \omega_{\mathcal{A}_g} \right\}.
		$$
	\end{thm}
	
	\begin{proof}
		By Theorem \ref{p-3-4}, by Corollary \ref{c6-3-2-8}, and by (\ref{al38}), (\ref{al39}) and (\ref{al40}), we have
		\begin{align*}
			-dd^c \log \tau_{\widetilde{M}_0, \mathcal{K}, \xs} &= \frac{(t+1)(t+7)}{16} c_1(f_*K_{\xs}, h_{L^2}) + \omega_{H^{\cdot}(\xss)} \\
			&= \frac{(t+1)(t+7)}{8} \bar{\pi}_{M_0}^* \omega_{\mathcal{M}_{M_0}} +(t+1)J_{M_0, \mathscr{Y}^{\sigma}/S}^* \omega_{\mathcal{A}_g} \\
			&= \frac{(r-6)(r-9)}{2} \bar{\pi}_{M_0}^* \omega_{\mathcal{M}_{M_0}} +2(r-9)J_{M_0, \mathscr{Y}^{\sigma}/S}^* \omega_{\mathcal{A}_g} \\
			&= 2(r-9) \left\{ \frac{r-6}{4} \bar{\pi}_{M_0}^* \omega_{\mathcal{M}_{M_0}} + J_{M_0, \mathscr{Y}^{\sigma}/S}^* \omega_{\mathcal{A}_g} \right\},
		\end{align*}
		and we obtain the result.
	\end{proof}

	Let us recall the invariant $\tau_{M_0}$ of 2-elementary K3 surface of type $M_0$ introduced in \cite[Definition~5.1]{MR2047658}.
	Let  $(Y, \sigma)$ be 2-elementary K3 surface of type $M_0$.
	Choose a $\sigma$-invariant \K metric $h_Y$ on $Y$and its associated \K form is denoted by $\omega_Y$.
	We define the volume of $(Y, \omega_Y)$ by
	$$
		\vol(Y, \omega_{Y}) = \int_Y \frac{\omega_{Y}^2}{2!}.
	$$
	
	The fixed locus of $\sigma : Y \to Y$ is denoted by $Y^{\sigma}$.
	Assume that $Y^{\sigma} \neq \emptyset$. 
	Let $Y^{\sigma} = \sqcup_i C_i$ be the decomposition into the connected components.
	We define the volume of $(Y^{\sigma}, \omega_{Y^{\sigma}})$ by
	$$
		\vol(Y^{\sigma}, \omega_{Y^{\sigma}}) =\prod_i \vol(C_i, \omega_Y|_{C_i}) = \prod_i \int_{C_i} \omega_Y|_{C_i}.
	$$
	We define a positive number $A(Y, \sigma, h_Y) \in \mathbb{R}_{>0}$ by
	\begin{align*}
		A(Y, \sigma, h_Y) =\exp \left[ \frac{1}{8} \int_{Y^{\sigma}} \log \left( \frac{\eta \wedge \bar{\eta}}{\omega_{Y}^2/2!} \frac{\vol(Y, \omega_{Y})}{\|\eta\|_{L^2}^2} \right) c_1(TY^{\sigma}, h_Y|_{Y^{\sigma}}) \right],
	\end{align*}
	where $\eta$ is a holomorphic 2-form on $Y$.
	If $Y^{\sigma} = \emptyset$, we set $\vol(Y^{\sigma}, \omega_{Y^{\sigma}}) = A(Y, \sigma, h_Y)=1$.
	
	Let $\tau_{\sigma}(\bar{\mathcal{O}}_Y)$ be the equivariant analytic torsion of the trivial line bundle $\bar{\mathcal{O}}_Y$ with respect to the canonical metric,
	and let $\tau(\bar{\mathcal{O}}_{Y^{\sigma}})$ be the analytic torsion of the trivial line bundle $\bar{\mathcal{O}}_{X^{\iota}}$ with respect to the canonical metric.
	If $Y^{\sigma} = \emptyset$, we set $\tau(\bar{\mathcal{O}}_{Y^{\sigma}})=1$.
	
	\begin{dfn}\label{d7-3-3-1}
		Let  $(Y, \sigma)$ be 2-elementary K3 surface of type $M_0$
		and let $h_Y$ be a $\sigma$-invariant \K metric on $Y$.
		We define a real number $\tau_{M_0}(Y, \sigma)$ by
		\begin{align*}
			\tau_{M_0}(Y, \sigma)=\tau_{\sigma}(\bar{\mathcal{O}}_Y) \vol(Y, \omega_{Y})^{\frac{14 -r}{4}}
			\tau(\bar{\mathcal{O}}_{Y^{\sigma}}) \vol(Y^{\sigma}, \omega_{Y^{\sigma}})A(Y, \sigma, h_Y).
		\end{align*}
	\end{dfn}
	
	By \cite[Theorem~5.7]{MR2047658}, $\tau_{M_0}(Y, \sigma)$ is independent of the choice of $h_Y$
	and is an invariant of $(Y, \sigma)$.
	
	Let $g: (\mathscr{Y}, \sigma) \to S$ be a family of 2-elementary K3 surfaces of type $M_0$.
	We define the function $\tau_{M_0, \mathscr{Y}/S}$ on $S$ by
	$$
		\tau_{M_0, \mathscr{Y}/S}(s) = \tau_{M_0}(Y_s, \sigma_s) \quad (s \in S).
	$$
	By \cite[Theorem~5.6]{MR2047658}, the following equality holds:
	$$
		dd^c \log \tau_{M_0, \mathscr{Y}/S} =\frac{r-6}{4} \bar{\pi}_{M_0}^* \omega_{\mathcal{M}_{M_0}} + J_{M_0, \mathscr{Y}^{\sigma}/S}^* \omega_{\mathcal{A}_g}.
	$$
	Therefore, if $r \neq 9$, then the functions ${\tau_{\widetilde{M}_0, \xs}}^{\frac{1}{2r-18}}$ and $\tau_{M_0, \mathscr{Y}/S}$ satisfies the same curvature equation.

\section{Comparison of the invariants}\label{s-5}

	In this section, we compare the invariants $\tau_{\widetilde{M}_0, \mathcal{K}}$ and $\tau_{M_0}$ when $\mathcal{K}$ is natural.
	An example with non-natural $\mathcal{K}$ is also discussed.

\subsection{An equation of currents on $\mathcal{M}_{M_0}$}\label{ss-5-1}
	
	Let $M_0$ be a primitive hyperbolic 2-elementary sublattice of $L_{K3}$.
	Let $\Pi_{M_0^{\perp}} : \Omega_{M_0^{\perp}} \to \mathcal{M}_{M_0}$ be the natural projection.
	For $\delta \in \Delta(M_0^{\perp}) = \{ \delta \in M_0^{\perp} ; \delta^2 =-2 \}$, we set
	$$
		H^{\circ}_{\delta}= \{ [\eta] \in \Omega_{M_0^{\perp}} ; O(M_0^{\perp})_{[\eta]} = \{ \pm1, \pm s_{\delta} \} \},
	$$
	where $O(M_0^{\perp})_{[\eta]}$ is the stabilizer of $[\eta]$.
	We define an open subset $\mathscr{D}_{M_0^{\perp}}^{\circ}$ of $\mathscr{D}_{M_0^{\perp}}$ by
	$$
		\mathscr{D}_{M_0^{\perp}}^{\circ} = \bigcup_{\delta \in \Delta(M_0^{\perp})} H^{\circ}_{\delta}.
	$$
	The isometry group $O(M_0^{\perp})$ preserves $\mathscr{D}_{M_0^{\perp}}^{\circ}$ and we set
	$$
		\overline{\mathscr{D}}_{M_0^{\perp}}^{\circ} = \mathscr{D}_{M_0^{\perp}}^{\circ} / O(M_0^{\perp}).
	$$
	It is an open subset of $\mathscr{D}_{M_0^{\perp}}$.
	
	\begin{lem}\label{l2-5-1}
		Assume that $r(M_0) \leqq 18$.
		Then the following hold.
		\begin{enumerate}[ label= \rm{(\arabic*)} ]
			\item $\mathscr{D}_{M_0^{\perp}} \setminus \mathscr{D}_{M_0^{\perp}}^{\circ}$ is an analytic subset of $\Omega_{M_0^{\perp}}$ with $\operatorname{codim}_{\Omega_{M_0^{\perp}}} (\mathscr{D}_{M_0^{\perp}} \setminus \mathscr{D}_{M_0^{\perp}}^{\circ}) \geqq 2$.
				In particular, $\Omega_{M_0^{\perp}}^{\circ} \cup \mathscr{D}_{M_0^{\perp}}^{\circ} = \Omega_{M_0^{\perp}} \setminus (\mathscr{D}_{M_0^{\perp}} \setminus \mathscr{D}_{M_0^{\perp}}^{\circ})$ is an $O(M_0^{\perp})$-invariant Zariski open subset of $\Omega_{M_0^{\perp}}$.
			\item $\mathscr{D}_{M_0^{\perp}}^{\circ}$ is a smooth hypersurface of $ \Omega_{M_0^{\perp}}^{\circ} \cup \mathscr{D}_{M_0^{\perp}}^{\circ}$.
			\item $\overline{\mathscr{D}}_{M_0^{\perp}}^{\circ} \subset \mathcal{M}_{M_0} \setminus \operatorname{Sing} \mathcal{M}_{M_0}$, and $\overline{\mathscr{D}}_{M_0^{\perp}}^{\circ} \subset \overline{\mathscr{D}}_{M_0^{\perp}} \setminus \operatorname{Sing} \overline{\mathscr{D}}_{M_0^{\perp}}$.
		\end{enumerate}
	\end{lem}
	
	\begin{proof}
		See \cite[Proposition~1.9.]{MR2047658}.
	\end{proof}
	
	Set $\widetilde{M}_0 = M_0 \oplus \mathbb{Z} e$.
	Then $\widetilde{M}_0$ is an admissible sublattice of $L_2 = L_{K3} \oplus \mathbb{Z} e$.
	
	\begin{lem}\label{l6-4-1-1}
		We have $\Delta(M_0^{\perp_{L_{K3}}}) = \Delta(\widetilde{M}_0^{\perp_{L_2}})$.
	\end{lem}
	
	\begin{proof}
		Suppose that there exists an element $\delta \in \widetilde{M}_0^{\perp}$ such that $\delta^2 =-10, (\delta, L_2) =2 \mathbb{Z}$.
		Since $\delta \in \widetilde{M}_0^{\perp} =M_0^{\perp} \subset L_{K3}$, we have $\delta' := \frac{\delta}{2} \in L_{K3, \mathbb{Q}}$.
		Since $(\delta, L_2) =2 \mathbb{Z}$, we have $(\delta', L_{K3}) \subset \mathbb{Z}$.
		Since $L_{K3}$ is unimodular, we have $\delta' \in L_{K3}^{\vee} =L_{K3}$.
		Therefore $\mathbb{Z} \ni (\delta')^2 = \frac{1}{4} \delta^2 = -\frac{5}{2} \notin \mathbb{Z}$,
		which is a contradiction.
	\end{proof}
	
	Therefore the isomorphisms
	$$
		\phi : \mathcal{M}_{M_0} \cong \mathcal{M}_{\widetilde{M}_0, \mathcal{K}} \quad \text{and} \quad \psi : \mathcal{M}_{M_0} \cong \mathcal{M}_{\widetilde{M}_0, \mathcal{K}'}
	$$
	constructed in Theorems \ref{t-3-1-1} and \ref{p-2-2} satisfy
	$$
		\phi(\overline{\mathscr{D}}_{M_0^{\perp}}) = \overline{\mathscr{D}}_{\widetilde{M}_0^{\perp}} \quad \text{and} \quad \psi(\overline{\mathscr{D}}_{M_0^{\perp}}) = \overline{\mathscr{D}}_{\widetilde{M}_0^{\perp}}.
	$$
	In the following, we identify $\mathcal{M}_{\widetilde{M}_0, \mathcal{K}}$ and $\mathcal{M}_{\widetilde{M}_0, \mathcal{K}'}$ with $\mathcal{M}_{M_0}$
	and identify $\overline{\mathscr{D}}_{\widetilde{M}_0^{\perp}}$ with $\overline{\mathscr{D}}_{M_0^{\perp}}$.

	Let $u$ be a pluriharmonic function on $\mathcal{M}_{M_0}^{\circ}$.
	We assume the following:
	Let $C$ be an irreducible projective curve on the Baily-Borel compactification $\mathcal{M}_{M_0}^*$ satisfying Assumption \ref{as4-2-3-1}
	and let $p \in C \cap \overline{\mathscr{D}}_{M_0^{\perp}}$ be a smooth point of $C$.
	Choose a coordinate $(\D, s)$ on $C$ centered at $p$ such that $\D \cap \overline{\mathscr{D}}_{M_0^{\perp}} = \{p \}$ and $s(\D)$ is the unit disk.
	Then there exists a rational number $a \in \mathbb{Q}$ such that
	\begin{align}\label{al4-4-1-1}
		(u|_{\D})(s) = a \log |s|^2 +O( \log \log |s|^{-1} ).
	\end{align}
	
	\begin{lem}\label{l2-5-2}
		The equation of currents
		$$
			dd^c u = \sum_i a_i \delta_{\Gamma_i}
		$$
		holds on $\mathcal{M}_{M_0} $, where $\bar{\mathscr{D}}_{M_0^{\perp}} = \sum_i \Gamma_i$ is the irreducible decomposition.
	\end{lem}
	
	\begin{proof}
		Let $x \in \mathscr{D}_{M_0^{\perp}}^{\circ}$.
		By \cite[Section 7 (Step 1)]{MR2047658}, there exists a system of local coordinates $(U \cong \D^m, t, z_1, \dots , z_{m-1})$ on $\Omega_{M_0^{\perp}}^{\circ} \cup \mathscr{D}_{M_0^{\perp}}^{\circ}$ that satisfies
		\begin{enumerate}[ label= \rm{(\arabic*)} ]
			\item $\mathscr{D}_{M_0^{\perp}}^{\circ} \cap U =\{ (t,z) \in U ; t=0 \}$.
			\item For any $z \in \D^{m-1}$, there exists an irreducible curve $C_z$ on $\mathcal{M}_{M_0}^*$ such that $C_z$ contains $\Pi_{M_0^{\perp}}(\D \times \{ z \})$ as an open set.
		\end{enumerate}
		By the assumption of $u$, for each $z \in \D^{m-1}$, $u|_{\Pi_{M_0^{\perp}}(\D \times \{ z \})}$ satisfies
		$$
			u|_{\Pi_{M_0^{\perp}}(\D \times \{ z \})} = a_z \log |t|^2  +O(\log \log |t|^{-1}) \quad (t \to 0),
		$$
		where $a_z \in \mathbb{Q}$.
		Set $H= \mathscr{D}_{M_0^{\perp}}^{\circ} \cap U$ and $\D_z =\Pi_{M_0^{\perp}}(\D \times \{ z \})$.
		It suffices to show that $a=a_z$ is independent of the choice of $z \in \D^{m-1}$, and the equation of currents 
		$$
			dd^c u =a \delta_H
		$$ 
		holds on $U$.\par
		Since $\bar{\partial} \partial u=0$, $\partial u$ is a holomorphic 1-form on $\mathcal{M}_{M_0}^{\circ}$.
		Let $z \in \D^{m-1}$.
		Since $(u|_{\D_z})(t) -a_z \log |t|^2$ is a harmonic function on $\D_z \setminus \{ z \}$ and it satisfies
		\begin{align}\label{al4-4-1-2}
			(u|_{\D_z})(t) -a_z \log |t|^2 = O(\log\log|t|^{-1}),
		\end{align}
		it extends to a harmonic function on $\D_z$.
		Moreover $a_z$ is the residue of $\partial u |_{\D_z}$ at $t=0$.
		Since $a_z \in \mathbb{Q}$ and $a_z$ depends continuously on $z \in \D^{m-1}$,  
		$a=a_z$ is independent of the choice of $z \in \D^{m-1}$.
		By (\ref{al4-4-1-2}), there exists a non-negative function $C(z)$ on $\D^{m-1}$ which is not necessarily continuous such that 
		$$
			| u(t,z) -a\log|t|^2 | \leqq C(z)\log\log|t|^{-1}
		$$
		holds on $U$.
		By \cite[Proposition 3.11.]{MR2047658}, we have
		$
			dd^c u =a \delta_H,
		$
		and the proof is completed.
	\end{proof}

\subsection{The invariants for the case where $r(M_0) \leqq 17$ and $\bar{\mathscr{D}}_{M_0^{\perp}}$ is irreducible}\label{ss-5-2}

	Set $\widetilde{M}_0=M_0 \oplus \mathbb{Z} e$.
	Then $\widetilde{M}_0$ is an admissible sublattice of $L_2=L_{K3} \oplus \mathbb{Z} e$.
	 Let $\mathcal{K} \in \operatorname{KT}(\widetilde{M}_0)$ be a natural K\"ahler-type chamber.
	
	\begin{assumption}\label{as4-4-2-1}
		We assume that $r(M_0) \leqq 17$, and $\bar{\mathscr{D}}_{M_0^{\perp}}$ is irreducible.
	\end{assumption}
	Note that $\bar{\mathscr{D}}_{M_0^{\perp}}$ is irreducible if and only if $\# \Delta(M_0^{\perp})/O(M_0^{\perp})=1$.
	By \cite[11.3]{MR3039773} and \cite[FIGURE 1]{MR2424924}, there exist $23$ types of primitive hyperbolic 2-elementary sublattice of $L_{K3}$ which satisfy the above assumption.
	
	By Theorem \ref{t-3-1-1}, we have an isomorphism
	$$
		\phi : \mathcal{M}_{M_0} \to \mathcal{M}_{\widetilde{M}_0, \mathcal{K}}.
	$$
	We define a real-valued smooth function $\sigma_{\widetilde{M}_0, \mathcal{K}}$ on $\mathcal{M}_{M_0}^{\circ}$ by $\sigma_{\widetilde{M}_0, \mathcal{K}} =\tau_{\widetilde{M}_0, \mathcal{K}} \circ \phi$.
	Namely, it is defined by
	$$
		\sigma_{\widetilde{M}_0, \mathcal{K}}(Y, \sigma) =\tau_{\widetilde{M}_0, \mathcal{K}}(Y^{[2]}, \sigma^{[2]})
	$$
	for any 2-elementary K3 surface $(Y, \sigma)$ of type $M_0$.
	By Proposition \ref{t-4-1} and by \cite[Theorem 5.6.]{MR2047658}, the function $\log \left( \sigma_{\widetilde{M}_0, \mathcal{K}} \cdot \tau_{M_0}^{2(r-9)} \right)$ is a pluriharmonic function on $\mathcal{M}_{M_0}^{\circ}$.
	
	\begin{lem}\label{l2-5-3}
		The function $\log \left( \sigma_{\widetilde{M}_0, \mathcal{K}} \cdot \tau_{M_0}^{2(r-9)} \right)$ satisfies (\ref{al4-4-1-1}).
	\end{lem}
	
	\begin{proof}
		The proof is straightforward from \cite[Theorem 6.5]{MR2047658} and Theorem \ref{t3-2-3-1}.
	\end{proof}
	
	\begin{thm}\label{t2-5-4}
		If $r(M_0) \leqq 17$ and $\bar{\mathscr{D}}_{M_0^{\perp}}$ is irreducible, 
		then there exists a constant $C$ depending only on $M_0$ such that
		$$
			\tau_{\widetilde{M}_0, \mathcal{K}}(Y^{[2]}, \sigma^{[2]}) =C \tau_{M_0} (Y, \sigma)^{-2(r-9)}
		$$
		for any 2-elementary K3 surface $(Y, \sigma)$ of type $M_0$. 
	\end{thm}
	
	\begin{proof}
		We show that $u= \log \left( \sigma_{\widetilde{M}_0, \mathcal{K}} \cdot \tau_{M_0}^{2(r-9)} \right)$ is a constant function on $\mathcal{M}_{M_0}^{\circ}$.
		By Lemma \ref{l2-5-2}, the equation of currents
		\begin{align}\label{al4-4-2-2}
			dd^c u = a \delta_{\bar{\mathscr{D}}_{M_0^{\perp}}}
		\end{align}
		holds on $\mathcal{M}_{M_0}$.
		We choose an irreducible projective curve $C$ on $\mathcal{M}_{M_0}^*$ that satisfies the following conditions:
		\begin{align*}
			&(1)\hspace{3pt} C \subset \mathcal{M}_{M_0} \quad (2)\hspace{3pt} C \cap \bar{\mathscr{D}}_{M_0^{\perp}} \neq \emptyset \quad (3)\hspace{3pt} \operatorname{Sing}C \cap \bar{\mathscr{D}}_{M_0^{\perp}} =\emptyset \\
			&(4)\hspace{3pt} \text{$C$ intersects $\bar{\mathscr{D}}_{M_0^{\perp}}$ transversally}.
		\end{align*}
		Let $\varphi : \tilde{C} \to C$ be the normalization.
		Note that $\varphi^* \partial u$ is a meromorphic 1-form on $\tilde{C}$ and that its residue at $p \in \varphi^{-1}(C \cap \bar{\mathscr{D}}_{M_0^{\perp}})$ is equal to $a$ by (\ref{al4-4-2-2}) and $(4)$.
		By the residue theorem, we have 
		$$
			\sum_{p \in \varphi^{-1}(C \cap \bar{\mathscr{D}}_{M_0^{\perp}})} a =0.
		$$
		Thus we obtain $a=0$, and $u$ extends to a pluriharmonic function on $\mathcal{M}_{M_0}$.
		Since $r(M_0) \leqq 17$, we have $ \operatorname{codim} \mathcal{M}_{M_0}^* \setminus \mathcal{M}_{M_0} \geqq 2$.
		Therefore $u$ extends to a pluriharmonic function on $\mathcal{M}_{M_0}^*$ by \cite[Satz 4]{MR0081960}.
		Since $\mathcal{M}_{M_0}^*$ is compact, $u$ is constant.
	\end{proof}

\subsection{The invariants for the case where $M_0 = \langle 2 \rangle$}\label{ss-5-3}
	
	In Subsection \ref{ss-5-3}, we always suppose that $r=r(M_0)=1$ and $t=-17$.
	Hence $M_0= \mathbb{Z} h$ with $h^2=2$.
	Set $\widetilde{M}_0=M_0\oplus \mathbb{Z}e$.
	By Example \ref{e-2-35}, $\operatorname{KT}(\widetilde{M}_0) / \Gamma_{\widetilde{M}_0}$ consists of two classes, one natural and the other non-natural.
	Let $\mathcal{K} \in \operatorname{KT}(\widetilde{M}_0)$ be a natural K\"ahler-type chamber and let $\mathcal{K}' \in \operatorname{KT}(\widetilde{M}_0)$ be a non-natural K\"ahler-type chamber.

	\begin{lem}\label{l-5-4}
		Let $f : (\mathscr{X}, \iota) \to S$ be a family of $K3^{[2]}$-type manifolds with antisymplectic involution of type $(\widetilde{M}_0, \mathcal{K}')$,
		and let $g : (\mathscr{Y}, \sigma) \to S$ be a family of 2-elementary K3 surfaces of type $M_0$ such that
		$$
			X_s = \elm_{Y_s/\sigma_s}(Y_s^{[2]}), \quad \text{ and } \quad \iota_s = \elm_{Y_s/\sigma_s}(\sigma_s^{[2]})
		$$
		for each $s \in S$.
		For the definitions of $\elm_{Y_s/\sigma_s}(Y_s^{[2]})$ and $\elm_{Y_s/\sigma_s}(\sigma_s^{[2]})$, see Example \ref{e-2-35}.
		The following equality of $(1,1)$-forms on $S$ holds:
		$$
			-dd^c \log \tau_{\widetilde{M}_0, \mathcal{K}', \mathscr{X}/S} = 2(r-9) \left\{ \frac{r-6}{4} \bar{\pi}_{M_0}^* \omega_{\mathcal{M}_{M_0}} + J_{M_0, \mathscr{Y}^{\sigma}/S}^* \omega_{\mathcal{A}_g} \right\}.
		$$
	\end{lem}

	\begin{proof}
		By Proposition \ref{p-3-4}, we have
		\begin{align}\label{al4-4-3-1}
			-dd^c \log \tau_{\widetilde{M}_0, \mathcal{K}, \xs} = \frac{(t+1)(t+7)}{16} c_1(f_*K_{\xs}, h_{L^2}) +\omega_{H^{\cdot}(\xss)}.
		\end{align}
		Let $h : (\mathscr{Z}, \theta) \to S$ be a family of $K3^{[2]}$-type manifolds with antisymplectic involution of type $(\widetilde{M}_0, \mathcal{K})$ such that
		$$
			Z_s = Y_s^{[2]}, \quad \text{ and } \quad \theta_s = \sigma_s^{[2]}
		$$
		for each $s \in S$.
		Then we have
		$$
			X_s = \elm_{Y_s/\sigma_s}(Z_s), \quad \text{ and } \quad \iota_s = \elm_{Y_s/\sigma_s}(\theta_s)
		$$
		for each $s \in S$.
		Since the birational transform $Z_s \dashrightarrow X_s$ does not change the period (\cite[Proposition 25.14]{MR1963559}), we have
		\begin{align}\label{al4-4-3-2}
			c_1(f_*K_{\xs}, h_{L^2})= c_1(h_*K_{\mathscr{Z}/S}, h_{L^2})=2\bar{\pi}_{M_0}^* \omega_{\mathcal{M}_{M_0}}.
		\end{align}
		On the other hand, since $Y_s/\sigma_s = \mathbb{P}^2$ for all $s \in S$, we have
		\begin{align}\label{al4-4-3-3}
			\omega_{H^{\cdot}(\xss)} =\omega_{H^{\cdot}(\mathscr{Z}^{\theta}/S)}=2(r-9)J_{M_0, \mathscr{Y}^{\sigma}/S}^* \omega_{\mathcal{A}_g}.
		\end{align}
		Combining (\ref{al4-4-3-1}), (\ref{al4-4-3-2}) and (\ref{al4-4-3-3}), we obtain the desired result.
	\end{proof}
	
	We set
	\begin{align*}
		\Delta^+=\{ d \in \Delta(M_0^{\perp}); d/2 \in (M_0^{\perp})^{\vee} \}&, \quad \Delta^-=\{ d \in \Delta(M_0^{\perp}); d/2 \notin (M_0^{\perp})^{\vee} \}, \\
		\mathscr{D}^{\pm}_{{M_0}^{\perp}} =\sum_{d \in \Delta^{\pm} } d^{\perp},  \qquad \qquad &\qquad \qquad  \bar{\mathscr{D}}^{\pm}_{{M_0}^{\perp}} = \mathscr{D}^{\pm}_{M_0} /O(M_0^{\perp}).
	\end{align*}
	By \cite[11.3.]{MR3039773} and \cite[FIGURE 1.]{MR2424924}, $\bar{\mathscr{D}}_{M_0^{\perp}}$ consists of two connected components $\bar{\mathscr{D}}^{+}_{M_0}$ and  $\bar{\mathscr{D}}^{-}_{M_0}$.
	
	By Theorem \ref{t-3-1-1} and Proposition \ref{p-2-2}, we have the isomorphisms
	$$
		\phi : \mathcal{M}_{M_0} \to \mathcal{M}_{\widetilde{M}_0, \mathcal{K}} \quad \text{ and } \quad  \psi : \mathcal{M}_{M_0} \to \mathcal{M}_{\widetilde{M}_0, \mathcal{K}'}.
	$$
	We define real-valued smooth functions $\sigma_{\widetilde{M}_0, \mathcal{K}}$ and $\sigma_{\widetilde{M}_0, \mathcal{K}'}$ on $\mathcal{M}_{M_0}^{\circ}$ by $\sigma_{\widetilde{M}_0, \mathcal{K}} =\tau_{\widetilde{M}_0, \mathcal{K}} \circ \phi$ and $\sigma_{\widetilde{M}_0, \mathcal{K}'} =\tau_{\widetilde{M}_0, \mathcal{K}'} \circ \psi$.
	Namely, they are defined by
	$$
		\sigma_{\widetilde{M}_0, \mathcal{K}}(Y, \sigma) =\tau_{\widetilde{M}_0, \mathcal{K}}(Y^{[2]}, \sigma^{[2]}), \quad \sigma_{\widetilde{M}_0, \mathcal{K}'}(Y, \sigma) =\tau_{\widetilde{M}_0, \mathcal{K}'}(  \elm_{Y/\sigma}(Y^{[2]}), \elm_{Y/\sigma}(\sigma^{[2]}) ) 
	$$
	for any 2-elementary K3 surface $(Y, \sigma)$ of type $M_0$.
	By Proposition \ref{t-4-1}, Lemma \ref{l-5-4} and \cite[Theorem 5.6.]{MR2047658}, the functions $\log \left( \sigma_{\widetilde{M}_0, \mathcal{K}} \cdot \tau_{M_0}^{2(r-9)} \right)$ and $\log \left( \sigma_{\widetilde{M}_0, \mathcal{K}'} \cdot \tau_{M_0}^{2(r-9)} \right)$ are pluriharmonic functions on $\mathcal{M}_{M_0}^{\circ}$.
	
	\begin{lem}\label{l2-5-5}
		The functions $\log \left( \sigma_{\widetilde{M}_0, \mathcal{K}} \cdot \tau_{M_0}^{2(r-9)} \right)$ and $\log \left( \sigma_{\widetilde{M}_0, \mathcal{K}'} \cdot \tau_{M_0}^{2(r-9)} \right)$ satisfy (\ref{al4-4-1-1}).
	\end{lem}
	
	\begin{proof}
		The proof is straightforward from \cite[Theorem 6.5.]{MR2047658} and Theorem \ref{t3-2-3-1}.
	\end{proof}
	
	\begin{thm}\label{t2-5-6}
		There exist positive constants $C_1, C_2 > 0$ such that for any 2-elementary K3 surface $(Y, \sigma)$ of type $M_0=\langle 2 \rangle$, the following identities hold:
		\begin{align*}
			&\tau_{\widetilde{M}_0, \mathcal{K}}(Y^{[2]}, \sigma^{[2]})  =C_1\tau_{M_0}(Y, \sigma)^{-2(r-9)}, \\
			&\tau_{\widetilde{M}_0, \mathcal{K}'}(\elm_{Y/\sigma}(Y^{[2]}), \elm_{Y/\sigma}(\sigma^{[2]})) =C_2\tau_{M_0}(Y, \sigma)^{-2(r-9)}.
		\end{align*}
	\end{thm}
	
	\begin{proof}
		We follow \cite[Lemma 8.7]{MR4177283}.
		We set $u= \log \left( \sigma_{\widetilde{M}_0, \mathcal{K}} \cdot \tau_{M_0}^{2(r-9)} \right) \text{ or } \log \left( \sigma_{\widetilde{M}_0, \mathcal{K}'} \cdot \tau_{M_0}^{2(r-9)} \right)$.
		It suffices to show that $u$ is a constant function on $\mathcal{M}_{M_0}^{\circ}$.
		By Lemma \ref{l2-5-2}, the equation of currents
		\begin{align*}
			dd^c u = a \delta_{\bar{\mathscr{D}}_{M_0^{\perp}}^+} +b \delta_{\bar{\mathscr{D}}_{M_0^{\perp}}^-}
		\end{align*}
		holds on $\mathcal{M}_{M_0}$, where $a, b \in \mathbb{Q}$.
		Therefore we have 
		\begin{align*}
			dd^c u =b \delta_{\bar{\mathscr{D}}_{M_0^{\perp}}^-}
		\end{align*}
		as currents on $\mathcal{M}_{M_0} \setminus \bar{\mathscr{D}}_{M_0^{\perp}}^+$.
		
		Let $U \subset |\mathcal{O}_{\mathbb{P}^2}(6)|$ be the space of smooth plane sextics,
		and let $V \subset |\mathcal{O}_{\mathbb{P}^2}(6)|$ be that of sextics with at most one node.
		By \cite{MR595204}, the geometric quotient $U/\operatorname{PGL}_3$ is contained in $V/\operatorname{PGL}_3$ as an open subset,
		there is an isomorphism $U/\operatorname{PGL}_3 \cong \mathcal{M}^{\circ}_{M_0}$,
		and it extends to an open embedding $V/\operatorname{PGL}_3 \to \mathcal{M}_{M_0}\setminus \bar{\mathscr{D}}_{M_0^{\perp}}^+$. \par
		Let $f : V \to \mathcal{M}_{M_0}\setminus \bar{\mathscr{D}}_{M_0^{\perp}}^+$ be the composition of the quotient map $V \to V/\operatorname{PGL}_3$ and the open embedding $V/\operatorname{PGL}_3 \to \mathcal{M}_{M_0}\setminus \bar{\mathscr{D}}_{M_0^{\perp}}^+$.
		We have the equation of currents
		\begin{align}\label{al4-4-3-4}
			dd^c f^*u = \nu b \delta_{f^* \bar{\mathscr{D}}_{M_0^{\perp}}^-}
		\end{align}
		on $V$, where $\nu \in \mathbb{Z}_{> 0}$.
		
		Note that the codimension of $ |\mathcal{O}_{\mathbb{P}^2}(6)| \setminus V$ is $2$.
		There exists an irreducible projective curve $C$ in $ |\mathcal{O}_{\mathbb{P}^2}(6)|$ such that 
		\begin{align*}
			&(1)\hspace{3pt} C \subset V \quad (2)\hspace{3pt} C \cap f^*\bar{\mathscr{D}}^-_{M_0^{\perp}} \neq \emptyset \quad (3)\hspace{3pt} \operatorname{Sing}C \cap f^*\bar{\mathscr{D}}^-_{M_0^{\perp}} =\emptyset \\
			&(4)\hspace{3pt} \text{$C$ intersects $f^*\bar{\mathscr{D}}^-_{M_0^{\perp}}$ transversally}.
		\end{align*}
		Let $\varphi : \tilde{C} \to C$ be the normalization.
		Note that $\varphi^*f^* \partial u$ is a meromorphic 1-form on $\tilde{C}$ whose residue at $p \in \varphi^{-1}(C \cap f^*\bar{\mathscr{D}}^-_{M_0^{\perp}})$ is a positive multiple of $\nu b$, say $m_p \nu b$, $m_p >0$.
		By the residue theorem, we have 
		\begin{align}\label{al4-4-3-5}
			\sum_{p \in \varphi^{-1}(C \cap f^*\bar{\mathscr{D}}^-_{M_0^{\perp}})} m_p \nu b =0.
		\end{align}

		By (\ref{al4-4-3-5}), we have $b=0$ so that $f^* u$ is a pluriharmonic function on $V$ by (\ref{al4-4-3-4}).
		Since $\operatorname{codim}_{|\mathcal{O}(6)|}(|\mathcal{O}_{\mathbb{P}^2}(6)| \setminus V) =2$, $f^*u$ extends to a pluriharmonic function on $ |\mathcal{O}_{\mathbb{P}^2}(6)|$ by \cite[Satz 4]{MR0081960}.
		Thus $f^*u$ is a constant function on $|\mathcal{O}(6)|$, and $u$ is also constant.
	\end{proof}

\appendix

\section{}

\begin{center}
by  Ken-Ichi  Yoshikawa
\end{center}

	In this appendix, we determine the singularity of $\G$-equivariant Quillen metrics.
	Although this is a special case of \cite[Section 4]{Y10}, we write the detail for the sake of completeness.
	
	Acknowledgements. 
	This work was supported by JSPS KAKENHI Grant Numbers 23K20797, 21H04429.

\subsection{Smoothness of $\G$-equivariant Quillen metrics}

	Let $\mathscr{M}$ be a complex manifold endowed with a holomorphic $\G$-action 
	and let $B$ be a complex manifold endowed with a trivial $\G$-action.
	Let $\pi : \mathscr{M} \to B$ be a $\G$-equivariant proper holomorphic submersion.
	Assume that there is a $\G$-equivariant relatively ample line bundle $\mathcal{O}_{ \mathscr{M} }(1)$ on $\mathscr{M}$.
	Let $F \to \mathscr{M}$ be a $\G$-equivariant holomorphic vector bundle.
	We set $M_b = \pi^{-1}(b)$ and $F_b = F|_{ M_b }$ for $b \in B$.
	Then $\G$ preserves the fibers $M_b$ and $F_b$,
	and $\pi : \mathscr{M} \to B$ is a family of projective manifolds endowed with $\G$-action.
	We define $( R^q \pi_* F )_{\pm}$ as the sheaves on $B$ associated to the presheaves 
	$$
		U \mapsto H^q( \pi^{-1}(U), F|_{ \pi^{-1}(U) } )_{\pm}, 
	$$
	where $H^q( \pi^{-1}(U), F|_{ \pi^{-1}(U) } )_{\pm}$ are the $(\pm 1)$-eigenspaces of $H^q( \pi^{-1}(U), F|_{ \pi^{-1}(U) } )$ with respect to the $\G$-action.
	Then $ R^q \pi_* F = ( R^q \pi_* F )_{+} \oplus ( R^q \pi_* F )_{-}$.
	
	\begin{lem}\label{l-ap-1-1}
		Let $K \subset B$ be a compact subset.
		Then there exist a complex of $\G$-equivariant holomorphic vector bundles $ F_{\bullet} : 0 \to F_0 \to F_1 \to \dots \to F_m \to 0$ on a neighborhood of $\pi^{-1}(K)$ 
		and a homomorphism $i : F|_{ \pi^{-1}(K) } \to F_0$ of $\G$-equivariant vector bundles satisfying the following conditions:
		\begin{enumerate}[ label= \rm{(\arabic*)} ]
			\item The sequence of vector bundles $0 \to F|_{\pi^{-1}(K)} \to F_0 \to F_1 \to \dots \to F_m \to 0$ is exact.
			\item $H^q( M_b, F_i|_{M_b}) =0$ for all $q>0$, $i \geqq 0$ and $b \in K$.
		\end{enumerate}
	\end{lem}
	
	\begin{proof}
		We follow \cite[Sect. 7.2.7]{MR0338129}. 
		For $\nu \in \mathbb{N}$, we set $E_{ \nu } = \pi_* \mathcal{O}_{ \mathscr{M} }( \nu )$.
		Since $\mathcal{O}_{\mathscr{M}}(1)$ is relatively ample, 
		there exists an open neighborhood $U$ of $K$ and $m_0 \in \mathbb{N}$ such that $E_{ \nu }$ is a $\G$-equivariant vector bundle over $U$ and 
		such that the evaluation homomorphism $\pi^* E_{ \nu } \to \mathcal{O}_{ \mathscr{M} }( \nu )|_{\pi^{-1}(U)} \to 0$ is surjective and $\G$-equivariant for $\nu \geqq m_0$.
		Considering its dual, we get a $\G$-equivariant injective homomorphism of vector bundles $0 \to \mathcal{O}_{ \mathscr{M} }|_{\pi^{-1}(U)} \to (\pi^* E_{ \nu }^{\vee})( \nu )$.
		Tensoring with $F$, we get a $\G$-equivariant exact sequence of vector bundles $0 \to F|_{\pi^{-1}(U)} \to (\pi^* E_{ \nu }^{\vee}) \otimes F( \nu )|_{  \pi^{-1}(U) }$ over $\pi^{-1}(U)$, where $F(\nu) = F \otimes \mathcal{O}_{\mathscr{M}} (\nu)$.
		We set $F_0 = (\pi^* E_{ \nu }^{\vee}) \otimes F( \nu )|_{  \pi^{-1}(U) }$.
		Then there exist an open neighborhood $U_1$ of $K$ and $m_1 \in \mathbb{N}$ such that for $\nu \geqq m_1$, $H^q( M_b, F( \nu )|_{M_b} ) =0$ $( \forall q >0, \forall b \in U_1)$.
		Applying the  same construction to $F_0 / F$, we obtain an open neighborhood $U_2$ of $K$ with $K \subset U_2 \subset U_1$ a $\G$-equivariant vector bundle $F_1$ on $\pi^{-1}(U_2)$ and an exact sequence of $\G$-equivariant vector bundles $0 \to F|_{\pi^{-1}(U_2)} \to F_0|_{\pi^{-1}(U_2)} \to F_1$ on $\pi^{-1}(U_2)$ satisfying $(2)$ for $F_0$ and $F_1$.
		Repeating this $m-1$ times, we get an exact sequence $0 \to F|_{\pi^{-1}(U_m)} \to F_0|_{\pi^{-1}(U_m)} \to F_1|_{\pi^{-1}(U_m)} \to \dots \to F_{m-1}|_{\pi^{-1}(U_m)}$ defined on an open neighborhood $U_m$ of $K$ with $K \subset U_m \subset U_{m-1} \subset \dots \subset U_1 \subset U$, 
		where $m = \dim M_b$.  
		Then we set $F_m = F_{m-1} / F_{m-2}$ and we get the exact sequence of $\G$-equivariant holomorphic vector bundles $0 \to F \to F_0 \to F_1 \to \dots \to F_m \to 0$ on $\pi^{-1}(U_{m})$.
		By construction, $H^q(M_b, F_i|_{M_b}) =0$ for all $q >0$, $0 \leqq i \leqq m-1$ and $b \in U_m$.
		Then $H^q(M_b, F_m|_{M_b})=0$ for all $q >0$ and $b \in U_m$.
	\end{proof}
	
	By Lemma \ref{l-ap-1-1}, $\pi_*F_{\bullet} : 0 \to \pi_*F_0 \to \pi_*F_1 \to \dots \to \pi_*F_m \to 0 $ is a complex of $\G$-equivariant locally free sheaves of finite rank over $K$
	such that for all $q \geqq 0$,
	\begin{align}\label{al-ap-1-2}
		R^q \pi_* F|_{K} \cong \mathscr{H}^q( \pi_* F_{\bullet} ) := \operatorname{Ker} \left\{ \pi_* F_q \to \pi_* F_{q+1} \right\} / \operatorname{Im} \left\{ \pi_* F_{q-1} \to \pi_* F_{q} \right\} .
	\end{align}
	Let $(\pi_* F_{i})_{\pm}$ be the $(\pm 1)$-eigensubbundle of $\pi_* F_i$ with respect to the $\G$-action.
	Since (\ref{al-ap-1-2}) is an isomorphism of $\G$-equivariant sheaves on $K$, we have the following canonical isomorphisms of coherent sheaves for all $q \geqq 0$
	\begin{align}\label{al-ap-1-3}
		(R^q \pi_* F)_{\pm}|_{K} \cong \mathscr{H}^q( (\pi_* F_{\bullet})_{\pm} ).
	\end{align}
	
	For any coherent sheaf $\mathcal{F}$ on $B$, one can associate the invertible sheaf $\det \mathcal{F}$ on $B$ (\cite[Th. 2]{MR0437541}, \cite[Sect. 3 a)]{MR931666}). 
	We define 
	$$
		\lambda_{\pm}(F) = \otimes_{q \geqq 0} ( \det(R^q \pi_* F)_{\pm} )^{(-1)^q}, \quad  \lambda_{\G}(F) = \lambda_+(F) \oplus \lambda_-(F).
	$$
	Since $\det (\pi_* F_{\bullet})_{\pm}|_{K} \cong \det \mathscr{H}^{\bullet}( (\pi_* F_{\bullet})_{\pm} )$ on $K$ in the canonical way (cf. \cite[p. 43 b)]{MR0437541}),
	we deduce from (\ref{al-ap-1-3}) the canonical isomorphisms of line bundles on $K$ (see also \cite[p.337-p.338]{MR931666})
	$$
		\lambda_{\pm}(F) \cong \otimes_{i \geqq 0} \lambda_{\pm}(F_i)^{(-1)^i}.
	$$
	Let $\mathcal{I}_b \subset \mathcal{O}_B$ be the ideal sheaf defining $b \in B$.
	Since $ (\pi_* F_{i})_{\pm} \otimes (\mathcal{O}_B / \mathcal{I}_b) \cong H^0( M_b, F_i )_{\pm}$ for all $b \in K$ and $i \geqq 0$ by Lemma \ref{l-ap-1-1} (2),
	we have the following canonical isomorphisms of lines for all $b \in K$:
	\begin{align}\label{al-ap-1-4}
	\begin{aligned}
		\lambda_{\pm}(F)_b :&= \lambda_{\pm}(F) \otimes (\mathcal{O}_B / \mathcal{I}_b) \\
		&\cong \otimes_{i \geqq 0} \lambda_{\pm}(F_i)^{(-1)^i} \otimes (\mathcal{O}_B / \mathcal{I}_b)\\
		&\cong \otimes_{i \geqq 0} (\det(\pi_* F_i)_{\pm})^{(-1)^i} \otimes (\mathcal{O}_B / \mathcal{I}_b)\\
		&\cong \otimes_{i \geqq 0} (\det H^0( M_b, F_i|_{M_b})_{\pm} )^{(-1)^i}\\
		&\cong \otimes_{q \geqq 0} (\det H^q( M_b, F|_{M_b})_{\pm} )^{(-1)^q}\\
		&= \lambda_{\pm} (F|_{M_b}).
	\end{aligned}
	\end{align}
	
	Let $h_{\mathscr{M}}$ be a $\G$-invariant \K metric on $\mathscr{M}$ 
	and let $h_F$ be a $\G$-invariant hermitian metric on $F$.
	Via the canonical isomorphism (\ref{al-ap-1-4}), $\lambda_{\G}(F)$ is endowed with the equivariant Quillen metric $\| \cdot \|_{Q, \lambda_{\G}(F)}(\iota)$ with respect to $h_{\mathscr{M}}|_{T\mathscr{M}/B}$, $h_F$ such that for all $b \in B$,
	$$
		\| \cdot \|_{Q, \lambda_{\G}(F)}(\iota)(b) := \| \cdot \|_{Q, \lambda_{\G}(F_b)}(\iota).
	$$
	
	For an open subset $U \subset B$, a holomorphic section $\sigma = ( \sigma_+, \sigma_-)$ is called admissible if both $\sigma_{+}$ and $\sigma_-$ are nowhere vanishing on $U$.
	
	\begin{prop}\label{p-ap-1-5}
		Let $U \subset B$ be a small open subset such that 
		$\lambda_{\G}(F)|_U$ admits an admissible holomorphic section $\sigma = ( \sigma_+, \sigma_-)$. 
		Then $\log \| \sigma \|^2_{Q, \lambda_{\G}(F)}(\iota)$ is a smooth function on $U$.
	\end{prop}
	
	\begin{proof}
		Since $U$ is small, we may assume by Lemma \ref{l-ap-1-1} that there exists a complex of $\G$-equivariant holomorphic vector bundles $F_{\bullet}$ and a $\G$-equivariant injective homomorphism $i : F|_{\pi^{-1}(U)} \to F_0$ satisfying Lemma \ref{l-ap-1-1} (1), (2).
		We have the canonical isomorphism of line bundles on $U$
		$$
			\varphi_{\pm} : \lambda_{\pm}(F) \cong \otimes_{q \geqq 0} \det \mathscr{H}^q( (\pi_* F_{\bullet})_{\pm} )^{(-1)^q} \cong \otimes_{i \geqq 0} \lambda_{\pm}(F_i)^{(-1)^i}.
		$$
		Shrinking $U$ if necessary, we may assume that
		there exist admissible holomorphic sections $\sigma_i = ((\sigma_i)_+, (\sigma_i)_-)$ of $\lambda_{\G}(F_i)$ defined on $U$ such that $\varphi_{\pm}(\sigma_{\pm}) = \otimes_{i \geqq 0}(\sigma_i)_{\pm}^{(-1)^i}$.
		
		Let $h_{F_i}$ be a $\G$-invariant hermitian metric on $F_i$ and let $\| \cdot \|_{Q, \lambda_{\G}(F_i)}(\iota)$ be the equivariant Quillen metric on $\lambda_{\G}(F_i)$ with respect to the $\G$-invariant metrics $h_{\mathscr{M}/B} = h_{\mathscr{M}}|_{T \mathscr{M} /B}$ and $h_{F_i}$.
		Let $\widetilde{ch}_{\iota}(\overline{F}, \overline{F}_{\bullet}) \in \widetilde{A}(\mathscr{M}^{\iota})$ be the Bott-Chern secondary class \cite[e), f)]{MR929146} such that
		$$
			dd^c \widetilde{ch}_{\iota}(\overline{F}, \overline{F}_{\bullet}) = \sum_{i \geqq 0} (-1)^i ch_{\iota}(F_i, h_{F_i}) -ch_{\iota}(F, h_{F}).
		$$
		Applying the Bismut immersion formula for equivariant Quillen metrics \cite[Th.0.1]{MR1316553} (cf. \cite[Th.3.7]{MR1872550}) to the immersion $\emptyset \hookrightarrow M_b$, $b \in U$, we get the following equation of functions on $U$
		\begin{align}\label{al-ap-1-6}
		\begin{aligned}
			\log \| \sigma \|^2_{Q, \lambda_{\G}(F)}(\iota) &= \sum_{i \geqq 0} (-1)^i \log \| \sigma_i \|^2_{Q, \lambda_{\G}(F_i)}(\iota) + \left[ \pi_* \{ Td_{\iota}(\overline{T \mathscr{M} /B}) \widetilde{ch}_{\iota}(\overline{F}, \overline{F}_{\bullet}) \} \right]^{(0)} \\
			&\equiv \sum_{i \geqq 0} (-1)^i \log \| \sigma_i \|^2_{Q, \lambda_{\G}(F_i)}(\iota) \quad \text{ mod $C^{\infty}(U)$},
		\end{aligned}
		\end{align}
		where $\overline{T \mathscr{M} /B} =(T \mathscr{M} /B, h_{ \mathscr{M} /B})$, $\overline{F} =(F, h_F)$ and $[\alpha]^{(2d)}$ denotes the component of degree $2d$ of a differential form $\alpha$.
		Since $\pi : \pi^{-1}(U) \to U$ is a proper holomorphic submersion and since $b \mapsto h^0(M_b, F_i|_{M_b})$ is a constant function on $U$,
		it follows from \cite[Th.3.5]{MR931666}, \cite[Th.2.12]{MR1800127} that $\log \| \sigma_i \|^2_{Q, \lambda_{\G}(F_i)}(\iota)$ is a smooth function on $U$.
		This, together with (\ref{al-ap-1-6}), implies the result.
	\end{proof}

\subsection{Singularity of equivariant Quillen metrics}

	Let $\iota$ be the generator of $\G$.
	Let $\X$ be a projective manifold endowed with a holomorphic $\G$-action.
	Let $C$ be a smooth projective curve endowed with a trivial $\G$-action.
	Let $\pi : \X \to C$ be a $\G$-equivariant surjective holomorphic map.
	Hence $\G$ preserves the fibers of $\pi$.
	Let $\Sigma_{\pi} = \{ x \in \X ; d \pi (x)=0 \}$ be the critical locus of $\pi$
	and let $\Delta_{\pi} = \pi(\Sigma_{\pi}) \subset C$ be the discriminant locus of $\pi$.
	Assume that there exists a $\G$-equivariant ample line bundle on $\X$.
	We set $X_s = \pi^{-1}(s)$ for $s \in C$ and 
	$$
		C^{\circ} = C \setminus \Delta_{\pi}, \quad \X^{\circ} =\pi^{-1}(C^{\circ}), \quad \pi^{\circ} = \pi|_{\X^{\circ}}.
	$$ 
	Then $\pi^{\circ} : \X^{\circ} \to C^{\circ}$ is a family of projective algebraic manifolds with $\G$-action.
	
	Let $T\X/C$ be the $\G$-equivariant subbundle of $T\X|_{\X \setminus \Sigma_{\pi}}$ defined by $T\X/C = \operatorname{Ker} \pi_*|_{\X \setminus \Sigma_{\pi}}$.
	Let $h_{\X}$ be a $\G$-invariant \K metric on $\X$ and set $h_{\X/C}=h_{\X}|_{T\X/C}$.
	Let $\xi \to \X$ be a $\G$-equivariant holomorphic vector bundle on $\X$ endowed with a $\G$-invariant hermitian metric $h_{\xi}$.
	We set $\xi_s = \xi|_{X_s}$ for $s \in C$.
	
	Let $0 \in \Delta_{\pi}$ be a critical value of $\pi$.
	Let $(S, s)$ be a coordinate neighborhood of $C$ centered at $0$ such that $S \cap \Delta_{\pi} = \{ 0 \}$.
	We set $X = \pi^{-1}(S)$ and $S^{\circ} = S \setminus \{ 0 \}$.
	
	For $s \in S^{\circ}$, let $\tau_{\G}(X_s, \xi_s)(\iota)$ be the equivariant analytic torsion of $(X_s, \xi_s)$ with respect to $h_{X_s} = h_{\X}|_{X_s}$ and $h_{\xi_s} = h_{\xi}|_{X_s}$.
	Let $\lambda_{\G}(\xi)$ be the equivariant determinant of the cohomology of $\xi$.
	Then $ \lambda_{\G}(\xi) |_{C^{\circ}}$ is equipped with the equivariant Quillen metric $\| \cdot \|_{Q, \lambda_{\G}(\xi)}(\iota)$ with respect to the $\G$-invariant metrics $h_{\X/C}$, $h_{\xi}$.
	Let $\sigma$ be an admissible holomorphic section of $\lambda_{\G}(\xi) |_{S}$.
	By Proposition \ref{p-ap-1-5}, $\log \| \sigma (s) \|^2_{Q, \lambda_{\G}(\xi)}(\iota)$ is a smooth function on $S^{\circ}$.
	In this subsection, following Bismut \cite{MR1486991} and Yoshikawa \cite{MR2262777}, we determine the behavior of $\log \| \sigma (s) \|^2_{Q, \lambda_{\G}(\xi)}(\iota)$ as $s \to 0$.
	
	Let $\Gamma \subset \X \times C$ be the graph of $\pi$.
	Then $\Gamma$ is a smooth divisor on $\X \times C$ preserved by the $\G$-action on $\X \times C$.
	Let $[\Gamma]$ be the $\G$-equivariant holomorphic line bundle on $\X \times C$ associated to $\Gamma$.
	Let $\varsigma_{\Gamma} \in H^0(\X \times C, [\Gamma])$ be the canonical section of $[\Gamma]$ such that $\operatorname{div}(\varsigma_{\Gamma}) = \Gamma$.
	We identify $\X$ with $\Gamma$ via the projection $\Gamma \to \X$.
	
	Let $i : \Gamma \hookrightarrow \X \times C$ be the inclusion.
	Let $p_1 : \X \times C \to \X$ and $p_2 : \X \times C \to C$ be the projections.
	By the $\G$-equivariances of $i, p_1, p_2$, we have the exact sequence of $\G$-equivariant coherent sheaves on $\X \times C$,
	\begin{align}\label{al-ap-2-1}
		0 \to \mathcal{O}_{\X \times C}([\Gamma]^{-1} \otimes p_1^*\xi) \xrightarrow{ \otimes \varsigma_{\Gamma} } \mathcal{O}_{\X \times C}(p_1^*\xi) \to i_*\mathcal{O}_{\Gamma}(p_1^*\xi) \to 0.
	\end{align}
	
	Let $\lambda_{\G}(p_1^* \xi)$, $\lambda_{\G}( [\Gamma]^{-1} \otimes p_1^*\xi )$ and $\lambda_{\G}(\xi)$ be the equivariant determinants of the direct images $R(p_2)_*\mathcal{O}(p_1^*\xi)$, $R(p_2)_* \mathcal{O}_{\X \times C}([\Gamma]^{-1} \otimes p_1^*\xi)$, $R \pi_* \mathcal{O}_{\X}(\xi)$, respectively.
	Under the isomorphism $p_1^*\xi|_{\Gamma} \cong \xi$ induced from the identification $p_1 : \Gamma \to \X$, we define the holomorphic vector bundles $\lambda_{\G}$, $\lambda_{\pm}$ on $C$ by
	\begin{align*}
		\lambda_{\G} &:= \lambda_{\G}( [\Gamma]^{-1} \otimes p_1^*\xi ) \otimes \lambda_{\G}(p_1^*\xi)^{-1} \otimes \lambda_{\G}(\xi) = \lambda_{+} \oplus \lambda_{-} , \\
		\lambda_{\pm} &:= \lambda_{\pm}( [\Gamma]^{-1} \otimes p_1^*\xi ) \otimes \lambda_{\pm}(p_1^*\xi)^{-1} \otimes \lambda_{\pm}(\xi) .
	\end{align*}
	Since $\lambda_{+}$ (resp. $\lambda_{-}$) is the determinant of the acyclic complex of coherent sheaves on $C$ obtained as the $(+ 1)$-component (resp. $(-1)$-component) of the long exact sequence of direct image sheaves on $C$ associated to (\ref{al-ap-2-1}),
	$\lambda_{+}$ and $\lambda_{-}$ are canonically isomorphic to $\mathcal{O}_C$ (\cite{MR1316553}, \cite{MR1188532}, \cite{MR0437541}). 
	The canonical nowhere vanishing section of $\lambda_{+}$ (resp. $\lambda_{-}$) is denoted by $( \sigma_{KM} )_{+}$ (resp. $( \sigma_{KM} )_{-}$).
	Then we set $\sigma_{KM} := (( \sigma_{KM} )_{+}, ( \sigma_{KM} )_{-})$.
	
	Let $U \subset S$ be a relatively compact neighborhood of $0 \in \Delta_{\pi}$ and set $U^{\circ} := U \setminus \{ 0 \}$.
	On $X = \pi^{-1}(S)$, we identify $\pi$ (resp. $d\pi$) with $s \circ \pi$ (resp. $d(s \circ \pi)$).
	Hence $\pi \in \mathcal{O}(X)$ and $d\pi \in H^0(X, \Omega^1_X)$ in what follows.
	
	Let $h_{[\Gamma]}$ be a $\G$-invariant $C^{\infty}$ hermitian metric on $[\Gamma]$ such that
	\begin{align}\label{al-ap-2-2}
		h_{[\Gamma]}(\varsigma_{\Gamma}, \varsigma_{\Gamma})(w, t) = \left\{
			\begin{aligned}
			\begin{split}
				& | \pi (w) -t |^2 \hspace{20pt} \text{if} \quad (w, t) \in \pi^{-1}(U) \times U ,\\
				& \qquad 1 \hspace{46pt} \text{if} \quad (w, t) \in (\X \setminus X) \times U 
			\end{split} \end{aligned} \right. 
	\end{align}
	and let $h_{[\Gamma]^{-1}}$ be the metric on $[\Gamma]^{-1}$ induced from $h_{[\Gamma]}$.
	
	Let $\| \cdot \|_{Q, \lambda_{\G}(\xi)}(\iota)$ be the equivariant Quillen metric on $\lambda_{\G}(\xi)$ with respect to $h_{\X/C}$, $h_{\xi}$.
	Let $\| \cdot \|_{Q, \lambda_{\G}( [\Gamma]^{-1} \otimes p_1^*\xi )}(\iota)$ (resp. $\| \cdot \|_{Q, \lambda_{\G}( p_1^*\xi )}(\iota)$ ) be the equivariant Quillen metric on $\lambda_{\G}( [\Gamma]^{-1} \otimes p_1^*\xi )$ (resp. $\lambda_{\G}( p_1^*\xi )$ ) 
	with respect to $h_{\X}$, $h_{[\Gamma]^{-1}} \otimes h_{\xi}$ (resp. $h_{\X}$, $h_{\xi}$ ).
	Let $\| \cdot \|_{Q, \lambda_{\G}}(\iota)$ be the equivariant Quillen metric on $\lambda_{\G}$ defined by the tensor product of those on $\lambda_{\G}( [\Gamma]^{-1} \otimes p_1^*\xi )$, $\lambda_{\G}( p_1^*\xi )$, $\lambda_{\G}(\xi)$.
	
	Let $\Pi : \mathbb{P}(T\X)^{\vee} \to \X$ be the projective space bundle over $\X$ with fibers $ \mathbb{P}(T\X)_x^{\vee} := \mathbb{P}(T_x\X)^{\vee}$.
	Let $\mathcal{U} \to \mathbb{P}(T\X)^{\vee}$ be the universal hyperplane bundle and set $\mathcal{H} := \Pi^*T\X /\mathcal{U}$.
	As before, we have the following exact sequence of holomorphic vector bundles on $\mathbb{P}(T\X)^{\vee}$ :
	$$
		\mathcal{S} : 0 \to \mathcal{U} \to \Pi^*T\X \to \mathcal{H} \to 0.
	$$
	We define the Gauss map $\nu : \X \setminus \Sigma_{\pi} \to \mathbb{P}(T\X)^{\vee}$ by $\nu(x) := [T_x X_{\pi(x)}]$.
	Then we have the equality of vector bundles $T\X /C =\nu^* \mathcal{U}$ on $\X \setminus \Sigma_{\pi}$.
	Let
	$$
		q : (\widetilde{\X}, E) \to (\X, \Sigma_{\pi})
	$$
	be a resolution of the indeterminacy of $\nu$ such that $\tilde{\nu} := \nu \circ q$ extends to a morphism from $\widetilde{\X}$ to $\mathbb{P}(T\X)^{\vee}$.
	Here $E := q^{-1}(\Sigma_{\pi})$ is the exceptional divisor of $q$, whose ideal sheaf is defined by $\mathcal{I}_E := q^*\mathcal{I}_{\Sigma_{\pi}}$.
	In this way, $\tilde{\nu}^* \mathcal{U}$ is viewed as an extension of $T\X /C|_{\X \setminus \Sigma_{\pi}}$ to a vector bundle on $\widetilde{\X}$.
	Set
	$$
		E_0 := E \cap q^{-1}(X_0).
	$$
	
	Let $\X^{\iota} \subset \X$ be the set of fixed points of $\iota$.
	Then we have the decomposition 
	$$
		\X^{\iota} = \X_H^{\iota} \sqcup \X_V^{\iota},
	$$
	where $ \X_H^{\iota}$ is the horizontal component and $ \X_V^{\iota}$ is the vertical component.
	Namely, any connected component of $\X_H^{\iota}$ is flat over $C$ and $\X_V^{\iota} \subset \pi^{-1}(\Delta_{\pi})$.
	Let $\widetilde{\X_H^{\iota}} \subset \widetilde{\X}$ be the proper transform of $\X_H^{\iota}$.
	
	We define a topological invariant of the germ $( \pi : (\X, X_0) \to (S, 0), \xi)$ by
	\begin{align*}
		\alpha_{\iota}(X_0, \xi) =& \int_{E_0 \cap \widetilde{\X}^{\iota}_H} \widetilde{\nu}^* \left\{ \frac{1-Td(\mathcal{H})^{-1}}{c_1(\mathcal{H})} \right\} q^*\left\{ Td_{\iota}(T\X) ch_{\iota}(\xi) \right\} \\
		&\quad - \int_{ \X^{\iota}_V \cap X_0} Td_{\iota}(T\X) ch_{\iota}(\xi).
	\end{align*}
	For functions $f, g$ on $U^{\circ}$, we write $f \equiv g$ if $f- g \in C^{0}(U)$.
	
	\begin{thm}\label{t-ap-2-3}
		The following identity of functions on $S^{\circ}$ holds:
		$$
			\log \| \sigma_{KM} \|^2_{Q, \lambda_{\G}}(\iota) \equiv \alpha_{\iota}(X_0, \xi) \log |s|^2 .
		$$
	\end{thm}
	
	\begin{proof}
		We follow \cite[Sect.5]{MR1486991}, \cite[Theorem 6.3]{MR2047658}, \cite[Theorem 5.1]{MR2262777}. 
		The proof is quite parallel to that of \cite[Theorem 5.1]{MR2262777}. 
		The major differences stem from the fact that $\X^{\iota}$ consists of the horizontal component $\X_H^{\iota}$ and vertical component $\X_V^{\iota}$ 
		and these components give different contributions to the singularity of $\log \| \sigma_{KM} \|^2_{Q, \lambda_{\G}}(\iota)$. \par
		\textbf{ Step 1 } 
		Let $[X_s] = [\Gamma]|_{X_s}$ be the holomorphic line bundle on $\X$ associated to the divisor $X_s$.
		The canonical section of $[X_s]$ is defined by $\varsigma_s = \varsigma_{\Gamma}|_{\X \times \{ s \}} \in H^0(\X, [X_s])$.
		Then $\operatorname{div} (\varsigma_s) =X_s$.
		Let $i_s : X_s \hookrightarrow \X$ be the natural embedding.
		By (\ref{al-ap-2-1}), we get the exact sequence of $\G$-equivariant coherent sheaves on $\X$,
		\begin{align}\label{al-ap-2-4}
			0 \to \mathcal{O}_{\X}([X_s]^{-1} \otimes \xi) \xrightarrow{ \otimes \varsigma_{s} } \mathcal{O}_{\X}(\xi) \to (i_s)_*\mathcal{O}_{X_s}(\xi) \to 0
		\end{align}  
		and the canonical isomorphism $(\lambda_{\G})_s \cong \lambda_{\G}([X_s]^{-1} \otimes \xi) \otimes \lambda_{\G}(\xi)^{-1} \otimes \lambda_{\G}(\xi_s)$.
		
		Set $h_{[X_s]} = h_{[\Gamma]}|_{\X \times \{ s \}}$, which is a $\G$-invariant hermitian metric on $[X_s]$.
		Let $h_{[X_s]^{-1}}$ be the $\G$-invariant hermitian metric on $[X_s]^{-1}$ induced from $h_{[X_s]}$.
		Let $N_s = N_{X_s / \X}$ (resp. $N_s^{\vee} = N_{X_s / \X}^{\vee}$) be the normal (resp. conormal) bundle of $X_s$ in $\X$. 
		Then $d \pi|_{X_s} \in H^0(X_s, N_s^{\vee})$ generates $N_s^{\vee}$ for $s \in S^{\circ}$ and $d \pi|_{X_s}$ is $\G$-invariant (cf. \cite[(2.2)]{MR1316553}). 
		Let $a_{N_s^{\vee}}$ be the $\G$-invariant hermitian metric on $N_s^{\vee}$ defined by $a_{N_s^{\vee}}( d \pi|_{X_s}, d \pi|_{X_s}) =1$.
		Let $a_{N_s}$ be the $\G$-invariant hermitian metric on $N_s$ induced from $a_{N_s^{\vee}}$.
		We have the equality $c_1(N_s, a_{N_s}) =0$ for $s \in U^{\circ}$.
		By \cite[Proof of Theorem 5.1 Step 1]{MR2262777}, the $\G$-invariant metrics $h_{[X_s]^{-1}} \otimes h_{\xi}$ and $h_{\xi}$ verify assumption $(A)$ of Bismut \cite[Definition 1.5]{MR1029391} with respect to $a_{N_s}$ and $h_{\xi}|_{x_s}$. \par 
		\textbf{ Step 2 }
		Let $\mathcal{E}_s$ be the exact sequence of $\G$-equivariant holomorphic vector bundles on $X_s$ defined by
		$$
			\mathcal{E}_s : 0 \to TX_s \to T\X|_{X_s} \to N_s \to 0.
		$$
		By \cite{MR929146}, one has the Bott-Chern class $\widetilde{Td}_{\iota} (\mathcal{E}_s ; h_{X_s}, h_{\X}, a_{N_s}) \in \tilde{A}_{X_s^{\iota}}$ such that 
		$$
			dd^c \widetilde{Td}_{\iota} (\mathcal{E}_s ; h_{X_s}, h_{\X}, a_{N_s}) = Td_{\iota}(TX_s, h_{X_s}) Td(N_s, a_{N_s})|_{X_s^{\iota}} -Td_{\iota}(T\X, h_{\X}). 
		$$
		Here, to get the equality $Td_{\iota}(N_s, a_{N_s}) =Td(N_s, a_{N_s})|_{X_s^{\iota}}$, we used the triviality of the $\G$-action on $N_s|_{X_s^{\iota}}$.
		Set 
		$$
			(X^{\iota}_H)_s := \X^{\iota}_H \cap X_s.
		$$
		Applying the embedding formula of Bismut \cite{MR1316553} (see also \cite[Theorem 5.6]{MR1486991}) to the $\G$-equivariant embedding $i_s : X_s \hookrightarrow \X$ and to the exact sequence (\ref{al-ap-2-4}), 
		we get for all $s \in U^{\circ}$
		\begin{align}\label{al-ap-2-5}
		\begin{aligned}
			\log \| \sigma_{KM}(s) \|^2_{Q, \lambda_{\G}}(\iota) =& \int_{ ( \X^{\iota}_V \times \{s\} ) \sqcup ( \X^{\iota}_H \times \{s\} )} -\frac{ Td_{\iota}(T\X, h_{\X}) ch_{\iota}(\xi, h_{\xi}) }{ Td([\Gamma], h_{[\Gamma]}) } \log h_{[\Gamma]}( \varsigma_{\Gamma}, \varsigma_{\Gamma} )\\
			&- \int_{( X^{\iota}_H )_s} \frac{ \widetilde{Td}_{\iota} (\mathcal{E}_s ; h_{X_s}, h_{\X}, a_{N_s}) ch_{\iota}(\xi, h_{\xi}) }{ Td(N_s, a_{N_s}) } +C,
		\end{aligned}
		\end{align}
		where $C$ is a topological constant independent of $s \in U^{\circ}$.
		Here we used the triviality of the $\G$-action on $[X_s]|_{X_s^{\iota}}$
		and the explicit formula for the Bott-Chern current \cite[Remark 3.5 especially (3. 23), Theorem 3.15, Theorem 3.17]{MR1086887} to get the first term of the right hand side of (\ref{al-ap-2-5}). 
		Substituting (\ref{al-ap-2-2}) and $c_1(N_s, a_{N_s}) =0$ into (\ref{al-ap-2-5}), we get for $s \in U^{\circ}$
		\begin{align}\label{al4-ap-2-5}
		\begin{aligned}
			\log \| \sigma_{KM}(s) \|^2_{Q, \lambda_{\G}}(\iota) \equiv -\int_{ \X^{\iota}_H \times \{s\} } Td_{\iota}(T\X, h_{\X}) ch_{\iota}(\xi, h_{\xi}) \log |\pi -s|^2 \\
			- \left\{ \int_{ (\X^{\iota}_V \cap X_0) \times \{s\} } Td_{\iota}(T\X) ch_{\iota}(\xi) \right\} \log |s|^2 -\int_{( X^{\iota}_H )_s} \widetilde{Td}_{\iota} (\mathcal{E}_s ; h_{X_s}, h_{\X}, a_{N_s}) ch_{\iota}(\xi, h_{\xi}) \\
			\equiv - \left\{ \int_{ \X^{\iota}_V \cap X_0 } Td_{\iota}(T\X) ch_{\iota}(\xi) \right\} \log |s|^2 -\int_{( X^{\iota}_H )_s} \widetilde{Td}_{\iota} (\mathcal{E}_s ; h_{X_s}, h_{\X}, a_{N_s}) ch_{\iota}(\xi, h_{\xi}),
		\end{aligned}
		\end{align}
		where we used the equalities $h_{[\Gamma]}( \varsigma_{\Gamma}, \varsigma_{\Gamma})|_{ X_0 \times \{s\} } = |s|^2$ and $ \X^{\iota}_V \cap \pi^{-1}(U) = \X^{\iota}_V \cap X_0 $ to get the first equality and \cite[Theorem 9.1]{MR2262777} to get the second equality. \par 
		\textbf{ Step 3 }
		Let $\mathcal{L}$ be the tautological line bundle on $\mathbb{P}( \Omega^1_{\X} \otimes \pi^*TC )$,
		which is canonically identified with a line bundle on $\mathbb{P}( T\X )^{\vee}$.
		The vector bundles $\mathcal{U}$ and $\mathcal{L}$ are endowed with the hermitian metrics induced from the one on $\Pi^*T\X$ via the inclusions $\mathcal{U} \hookrightarrow \Pi^* T\X$ and $ \mathcal{L} \hookrightarrow \Pi^* ( \Omega^1_{\X} \otimes \pi^*TC )$, respectively.
		Namely, $h_{\mathcal{U}} := (\Pi^* h_{\X})|_{\mathcal{U}}$ and $h_{\mathcal{L}} := \Pi^* (h_{\X}^{\vee} \otimes \pi^* h_C)|_{\mathcal{L}}$.
		Write $\overline{\mathcal{U}} := (\mathcal{U}, h_{\mathcal{U}})$ and $\overline{\mathcal{L}} := (\mathcal{L}, h_{\mathcal{L}})$.
		
		Let $h_{N_s}$ be the hermitian metric on $N_s$ induced from $h_{\X}$ by the $C^{\infty}$ isomorphism $N_s \cong (TX_s)^{\perp}$.
		Let $\widetilde{Td}( N_s; a_{N_s}, h_{N_s} ) \in \tilde{A}_{X_s} $ be the Bott-Chern secondary class such that
		$$
			dd^c \widetilde{Td}( N_s; a_{N_s}, h_{N_s} ) = Td( N_s, a_{N_s} ) -Td( N_s, h_{N_s} ).
		$$
		Since $\G$ acts trivially on $N_s|_{X_s^{\iota}}$, we deduce from \cite[Propositions 1.3.2 and 1.3.4]{MR1038362} that 
		\begin{align}\label{al-ap-2-6}
		\begin{aligned}
			\widetilde{Td}_{\iota} (\mathcal{E}_s ; h_{X_s}, h_{\X}, a_{N_s}) = \widetilde{Td}_{\iota} (\mathcal{E}_s ; h_{X_s}, h_{\X}, h_{N_s}) +Td_{\iota} (TX_s, h_{X_s})\widetilde{Td} (N_s ; a_{N_s}, h_{N_s}) \\
			= \widetilde{Td}_{\iota} (\mathcal{E}_s ; h_{X_s}, h_{\X}, h_{N_s}) +\nu^* Td_{\iota} (\overline{\mathcal{U}}) \nu^* \left\{ \frac{ 1-Td( -c_1( \overline{\mathcal{L}} )) }{ -c_1( \overline{\mathcal{L}} ) } \right\} \log \| d\pi \|^2|_{ (X^{\iota}_H)_s } \\
			= \left[ \nu^* \widetilde{Td}_{\iota} ( \mathcal{S}; h_{\mathcal{U}}, \Pi^*h_{\X}, h_{\mathcal{H}} ) +\nu^* \left\{ Td_{\iota} (\overline{\mathcal{U}}) \frac{ 1-Td( -c_1( \overline{\mathcal{L}} )) }{ -c_1( \overline{\mathcal{L}} ) } \right\} \log \| d\pi \|^2 \right] |_{ (X^{\iota}_H)_s }.
		\end{aligned}
		\end{align}
		We used \cite[(13)]{MR2262777} and $( TX_s, h_{X_s} ) = \nu^* \overline{\mathcal{U}}|_{X_s}$ to get the second equality and we used the relation $(\mathcal{E}_s ; h_{X_s}, h_{\X}, h_{N_s}) = \nu^* ( \mathcal{S}; h_{\mathcal{U}}, \Pi^*h_{\X}, h_{\mathcal{H}} )|_{X_s}$ 
		and the functoriality of the Bott-Chern class \cite{MR929146} to get the third equality. 
		Set $\tilde{\pi} := \pi \circ q$.
		Substituting (\ref{al-ap-2-6}) into (\ref{al4-ap-2-5}) and using the fact that $[ ( \pi|_{\widetilde{\X_H^{\iota}}} )_* \omega ]^{(0)} \equiv 0$ for any smooth differential form $\omega$ on $\widetilde{\X_H^{\iota}}$ (cf. \cite[Theorem 1]{MR0666639}), we deduce from the same argument as in \cite[p.74 l.1-l.13]{MR2262777} that 
		\begin{align}\label{al-ap-2-7}
		\begin{aligned}
			\log \| \sigma_{KM}(s) \|^2_{Q, \lambda_{\G}}(\iota) \equiv -\left\{ \int_{ \X^{\iota}_V \cap X_0 } Td_{\iota}(T\X) ch_{\iota}(\xi) \right\} \log |s|^2 \\
			+ ( \tilde{\pi}|_{\widetilde{\X_H^{\iota}}} )_* \left[ \tilde{\nu}^* \left\{ Td_{\iota} (\overline{\mathcal{U}}) \frac{ Td( -c_1( \overline{\mathcal{L}} )) -1}{ -c_1( \overline{\mathcal{L}} ) } \right\} q^*ch_{\iota}(\xi, h_{\xi}) (q^* \log \| d\pi \|^2) \right]^{(0)} .
		\end{aligned}
		\end{align}
		By \cite[Corollary 4.6]{MR2262777} applied to the second term of the right hand side of (\ref{al-ap-2-7}), we get 
		\begin{align*}
			\log \| \sigma_{KM}(s) \|^2_{Q, \lambda_{\G}}(\iota) \equiv& -\left\{ \int_{ \X^{\iota}_V \cap X_0 } Td_{\iota}(T\X) ch_{\iota}(\xi) \right\} \log |s|^2 \\
			&+ \left[ \int_{ \widetilde{\X_H^{\iota}} \cap E_0 } \tilde{\nu}^* \left\{ Td_{\iota} (\mathcal{U}) \frac{ Td(\mathcal{H})-1 }{ c_1(\mathcal{H}) } \right\} q^*ch_{\iota}(\xi) \right] \log |s|^2 \\
			=& \alpha_{\iota}(X_0, \xi) \log |s|^2, 
		\end{align*} 
		where the last equality follows from the identity $Td_{\iota}(\mathcal{U}) Td(\mathcal{H}) =\Pi^* Td_{\iota}(T\X)$.
		This completes the proof.
	\end{proof}
	
	\begin{thm}\label{t-ap-2-8}
		The following identity of functions on $S^{\circ}$ holds:
		$$
			\log \| \sigma \|^2_{Q, \lambda_{\G}(\xi) }(\iota) \equiv \alpha_{\iota}(X_0, \xi) \log |s|^2 .
		$$
	\end{thm}
	
	\begin{proof}
		There exist admissible holomorphic sections 
		$$
			\alpha = (\alpha_+, \alpha_-) \in \Gamma( U, \lambda_{\G}(p_1^*\xi) ), \quad \beta = (\beta_+, \beta_-)  \in \Gamma( U, \lambda_{\G}([\Gamma]^{-1} \otimes p_1^*\xi) )  
		$$
		such that $\sigma_{KM} = \beta \otimes \alpha^{-1} \otimes \sigma$ on $U$, i.e., $(\sigma_{KM})_{\pm} = \beta_{\pm} \otimes \alpha_{\pm}^{-1} \otimes \sigma_{\pm}$.
		Then we have
		\begin{align*}
			\log \| \sigma \|^2_{Q, \lambda_{\G}(\xi) }(\iota) &= \log \| \sigma_{KM} \|^2_{Q, \lambda_{\G} }(\iota) +\log \| \alpha \|^2_{Q, \lambda_{\G}(p_1^*\xi) }(\iota) -\log \| \beta \|^2_{Q, \lambda_{\G}([\Gamma]^{-1} \otimes p_1^*\xi) }(\iota) \\
			&\equiv \alpha_{\iota}(X_0, \xi) \log |s|^2
		\end{align*}
		by Proposition \ref{p-ap-1-5} and Theorem \ref{t-ap-2-3}.
		This proves the theorem.
	\end{proof}

\bibliography{Reference}
\bibliographystyle{amsplain}

\end{document}